%% file: arxiv_article.tex
\documentclass{m2an}
\usepackage{mathtools}
\usepackage{dsfont}

\usepackage{lipsum}
\usepackage{amsfonts}
\usepackage{amssymb}
\usepackage{graphicx}
\usepackage{epstopdf}
\usepackage{algorithmic}
\usepackage{svg}
\usepackage{hyperref}
\usepackage{cleveref}
\usepackage{subcaption}
\usepackage{lmodern}
\usepackage{ifoddpage}
\usepackage{orcidlink}
\usepackage{amsopn}

\usepackage{comment}
\usepackage{url}
\usepackage{cancel}

\newcommand{\arxivshow}[1]{ #1 }
\newcommand{\arxivskip}[1]{ }
\newcommand{\vecn}{\textbf{n}}
\newcommand{\real}[1]{\ensuremath{\text{Re}\left\{ #1 \right\}}}

\newcommand{\trace}[2]{#2}

\def\ts{\tilde{s}}

\def\tb{\tilde{b}}
\def\teta{\tilde{\eta}}

\def\tu{\tilde{u}}
\def\tut{\tilde{u}_t}
\def\tutt{\tilde{u}_{tt}}

\def\ut{u_t}
\def\utt{u_{tt}}

\newcommand{\compemb}{\mathrel{\mathpalette\comp@emb\relax}}
\newcommand{\un}{u_n}
\def\unt{u_{n\,t}}
\def\untt{u_{n\,tt}}
\def\CPF{C_{\text{PF}}}
\def\Comega{C_{\Omega}}
\def\Ctrace{C_{\text{trace}}}
\newcommand{\g}{g}
\def\vx{\vec{x}}
\def\vxd{\vec{x}^{\,\dag}}
\def\vxz{\vec{x}^{\,0}}

\def\remainder{v}
\def\solspace{U}
\def\hh{\tilde{f}}

\def\du{(du)}
\def\hdu{\widehat{(du)}}
\def\hu{\widehat{u}}

\theoremstyle{plain}
\newtheorem{theorem}{Theorem}[section]
\newtheorem{lemma}[thrm]{Lemma}

\newtheorem{remark}[thrm]{Remark}

\begin{document}
\title{
Multi parameter identification in the nonlinear periodic Westervelt equation
}
\thanks{This research was funded in part by the Austrian Science Fund (FWF) [10.55776/P36318]. }
\author{Benjamin Rainer}
\address{Department of Mathematics, University of Klagenfurt, Austria, Carinthia (\email{benjamin.rainer@aau.at})\\ 
Austrian Institute of Technology GmbH, Vienna, Austria (\email{benjamin.rainer@ait.ac.at})}

\author{Barbara Kaltenbacher}
\address{Department of Mathematics, University of Klagenfurt, Austria, Carinthia (\email{barbara.kaltenbacher@aau.at}.)}
\date{Received....}
\begin{abstract} 
\input{abstract}
\end{abstract}
%
%
\maketitle

\input{article}

\bibliographystyle{siamplain}
\bibliography{references}

\input{supplemental_material}
\end{document}

%% file: abstract.tex
 Nonlinear ultrasound imaging leverages harmonic wave generation to enhance contrast and spatial resolution beyond the capabilities of conventional linear techniques. This behavior is commonly modeled by the Westervelt equation, which captures finite-amplitude acoustic wave propagation in heterogeneous media. In this work, we investigate an inverse problem for a periodic nonlinear Westervelt equation in $\mathbb{R}^d$, where $d\in\{2,3\}$ with spatially varying coefficients and Robin-type boundary conditions. The objective is to simultaneously reconstruct the sound speed, diffusivity, and nonlinearity parameters from (partial) boundary measurements. We first establish the Fréchet differentiability of the forward solution operator with respect to the unknown parameters, providing a rigorous analytical foundation for parameter identification. To address uniqueness, we introduce a reference-state framework and prove linearized uniqueness of an all-at-once forward operator without requiring the reference states to satisfy the governing equation. Building on these results, we develop an iterative reconstruction scheme based on a frozen Newton-type method, supported by an exact range invariance property. Numerical simulations are presented to illustrate the feasibility and performance of the proposed approach. 

%% file: article.tex
\section{Introduction}
Ultrasound imaging has undergone significant improvements through the exploitation of nonlinear acoustic phenomena~\cite{medicalUltrasoundImaging}. Unlike conventional linear imaging, which relies solely on the fundamental frequency, nonlinear ultrasound enables enhanced contrast and spatial resolution, leading to more accurate diagnosis and visualization in medical applications. The nonlinear response of acoustic waves in biological tissues~\cite{duck2002nonlinear} and artificial contrast agents, such as microbubbles, plays a central role in these advancements. When subjected to sufficiently high acoustic pressures, the compressibility of tissue and the oscillatory behavior of microbubbles become amplitude-dependent, giving rise to harmonic components in the propagated wavefield.

Harmonic imaging techniques that utilize these higher-order harmonics—rather than the fundamental frequency—offer several advantages~\cite{ANVARI2015, BURNS2000S19}. In particular, they improve lateral resolution and reduce imaging artifacts such as side and grating lobes, which commonly degrade image quality in conventional ultrasound. These benefits have motivated extensive study of nonlinear propagation models, among which the Westervelt equation serves as a widely accepted description of finite-amplitude sound propagation in fluids and soft tissues~\cite{karamalis2010fast,westervelt_deriv,solovchuk2013simulation}. There has been significant progress in studying the well-posedness of the periodic nonlinear Westervelt equation and its multiharmonic expansion as well as its numerical analysis in two and three space dimensions for bounded, open, and connected domains with sufficiently smooth boundary~\cite{Kaltenbacher21, RK2025}. The Westervelt equation for the real-valued acoustic pressure $p$ reads
\begin{equation}
\label{eq:westervelt}
    \frac{1}{c(x)^2} p_{tt} - \Delta p - \frac{\mathfrak{b}(x)}{c(x)^2} \Delta p_{t} = \frac{\beta_a(x)}{2\rho_0 c(x)^4} (p^2)_{tt},
\end{equation}
where $\beta_a(x) = 1 + \frac{B}{2A}(x)$ is the non-linearity parameter, $\mathfrak{b}(x)$ the diffusivity, $c(x)$ the speed of sound, and $\rho_0$ denotes the mass density, and we notationally emphasize space dependence of coefficients in the PDE. A physically meaningful assumption is to impose strict positivity on the speed of sound $c(x)$ and on $\mathfrak{b}(x)$ as well as non-negativity on $\beta_a(x)$.  Hence, multiplying~\eqref{eq:westervelt} by $c^2(x) > 0$ and dividing by $\mathfrak{b}(x) > 0$ is justified. Setting
\begin{equation}\label{etabs}
\eta(x) :=\frac{\beta_a(x)}{2 \rho_0 c^2(x) \mathfrak{b}(x)}, \quad
b(x) := \frac{1}{\mathfrak{b}(x)}, \quad
s(x)=\frac{c^2(x)}{\mathfrak{b}(x)},
\end{equation} 
and equipping~\eqref{eq:westervelt} with Robin boundary conditions we obtain the periodic nonlinear Westervelt equation on an open and connected domain $\Omega \subset\mathbb{R}^d$, $d\in\{2,3\}$ with $\text{C}^{1,1}$ boundary
\begin{equation}
\label{eq:westervelt:quadratic-nonlinearity:periodic}
\begin{cases}
b(x) p_{tt}(t,x) - s(x) \Delta p(t,x) - \Delta p_t(t,x) = \eta(x) (p(t,x)^2)_{tt} + f(t,x) & \text{in}\, (0,T) \times \Omega,\\
\beta(x) p_t(t,x) + \gamma(x) p(t,x) + \nabla p(t,x) \cdot \vecn = g(t,x) & \text{on}\,  (0,T) \times \partial\Omega, \\
p(0,x) = p(T,x), \, p_t(0,x) = p_t(T,x) & x \in \Omega, 
\end{cases}   
\end{equation}
where 
$\beta,\gamma \geq  0$ are the parameters for specifying absorbing or impedance conditions on $\partial\Omega$, and $\vecn$ denotes the outer normal on $\partial \Omega$. 

The available observations in this application context are measurements of the pressure at an array of transducers attached to the boundary of the computational domain $\Omega$
\begin{equation}\label{eq:obs}
h(t,x) = p(t,x),\qquad (t,x) \in (0,T) \times \Sigma,
\end{equation}
where $\Sigma\subseteq\partial \Omega$. 

Hence, our goal is to identify $s$, $b$ and $\eta$ from the Dirichlet traces of the solutions to \eqref{eq:westervelt:quadratic-nonlinearity:periodic} obtained from three different sources.
The actual physical quantities $c$, $\mathfrak{b}$, and $\beta_a$ can then be obtained from \eqref{etabs}.

Since typically excitation is imposed by the same or another ultrasound transducer array, the sources $g$ in the formulation above  correspond to a Neumann trace -- possibly modified by an impedance term -- 
on the boundary $\partial\Omega$ or part of it, see the second
line of  \eqref{eq:westervelt:quadratic-nonlinearity:periodic}.

A monofrequent source at the fundamental frequency $\omega = \frac{2 \pi}{T}$ can be written as
$g(t,x)=\text{Re}\{\hat{g}(x) e^{\iota \omega t}\}$
$x\in\partial\Omega$. 
As a consequence of nonlinearity, possible solutions to~\eqref{eq:westervelt:quadratic-nonlinearity:periodic} will not only consist of the fundamental frequency but also exhibit contributions at multiples of $\omega$, so-called higher harmonics. 
This also applies to sources that excite at multiple frequencies $\omega_i = \frac{2 \pi}{T_i}$. 

Taking a closer look at~\eqref{eq:westervelt:quadratic-nonlinearity:periodic} and applying the identity $(p(t,x)^2)_{tt} = 2(p_t(t,x)^2 + p(t,x) p_{tt}(t,x))$ we obtain
%
%
\begin{equation}
\nonumber
    (b(x)- 2\eta(x)p(t,x))p_{tt}(t,x) - s(x) \Delta p(t,x) - \Delta p_t(t,x) = 2\eta(x) p_t(t,x)^2.
\end{equation}
This shows that the Westervelt equation degenerates if $b(x)-2\eta(x)p(t,x) = 0$ for some $x \in \Omega$. Hence, it is natural to impose a ``smallness" condition on the source in order to avoid degeneracy. 

There exists a vast corpus of literature that aims to put the value of the non-linearity parameter, in most cases $B/A$, in relation to different tissue types, e.g.,~\cite{nonlinparam1, duck2002nonlinear, nonlinparam3,  nonlinparam2, ZHANG20011359}. Nonlinearity parameter tomography relies on the tissue dependence of $B/A$, but imaging $B/A$ alone would require all other (possibly spatially varying) coefficients to be known, which is typically not the case in applications. Moreover, also sound speed and attenuation coefficient come with their own diagnostic value and are often used as imaging quantity ~\cite{Lucka2022,schmitt2002ultrasound}.
This motivates our aim to simultaneously reconstruct $s$, $b$, and $\eta$ in \eqref{eq:westervelt:quadratic-nonlinearity:periodic}.

In~\cite{Kaltenbacher_2023} it has been shown that one can reconstruct the sound speed and the nonlinearity coefficient from boundary measurements using tailored sources within the domain $\Omega$;
\cite{nonlinearity_imaging_JMGTmulticoeff} considers uniqueness of sound speed, attenuation and nonlinearity coefficient in a third order in time model of nonlinear acoustics.
Also the authors of~\cite{Acosta2026SBETACOMPLEX} study the inverse problem of identifying all three coefficients in the periodic Westervelt equation using single frequency complex sources on the boundary; however their observation setting differs from the one considered here, as they assume the whole Dirichlet-to-Neumann map (rather than just three observations) to be available. On the other hand, they only need two harmonics by showing that the first harmonic suffices to determine the sound speed and the diffusivity coefficient, while adding the second harmonic enables them to determine the nonlinearity coefficient.
Closely related to this, the problem of reconstructing several space dependent coefficients in a system of coupled semilinear Helmholtz equations is studied in \cite{RenSoedjak:2024}.
The appearance of higher harmonics due to nonlinearity, as studied in the above cited  papers is another aspect of the so-called blessing of nonlinearity, that has been highlighted and quantified in, e.g., \cite{KurylevLassasUhlmann2018,LassasLiimatainenPotenciano-MachadoTyni2021}.

\medskip

The contributions of this paper are as follows:
\begin{enumerate}
    \item First, we formulate the inverse problem of identifying the spatially varying parameters $s$, $b$, and $\eta$ from boundary measurements and show Fr\'{e}chet differentiability of the forward operator with respect to these parameters in appropriate function spaces.
    \item Second, we prove linearized uniqueness at well-chosen reference states and reference parameters in an all-at-once  formulation of the inverse problem. The reference states do not necessarily need to solve~\eqref{eq:westervelt:quadratic-nonlinearity:periodic}.
    The three boundary sources applied for this purpose are obtained by an excitation at two different frequencies as well as an amplitude modulation, thus making explicit use of nonlinearity invalidating   the linear superposition principle.  
    \item Third, we formulate an iterative reconstruction scheme using a frozen Newton type method and establish its convergence by means of an exact range invariance property of the all-at-once forward operator as well as its linearized injectivity (see the second item).
    \item Fourth, we show some numerical experiments illustrating the developed theory.
\end{enumerate}

\medskip

The paper is organized as follows. We start by defining the inverse problem in~\cref{sec:inverse:problem} and study well-posedness and Fréchet differentiability of the underlying forward operator. 
Our main results on linearized uniqueness and convergence of a frozen Newton scheme can be found in~\cref{sec:recon:inverse:problem}, numerical experiments in \cref{sec:experiments}, and a discussion and conclusions follow in
\cref{sec:conclusions}. 

\medskip

Notation. In the rest of this paper we skip the dependence on time and space whenever it is apparent from the context.

\section{The inverse and forward problem}
\label{sec:inverse:problem}
%
%
To formulate the inverse problem as an operator equation 
\[
F(s,b,\eta)=y,
\]
we introduce the parameter-to-state map $S: \mathcal{D}(F) \rightarrow \solspace$, $(s,b,\eta) \mapsto u$ where $u$ solves~\eqref{eq:westervelt:quadratic-nonlinearity:periodic}, the observation operator 
defined by the Dirichlet trace
$\text{tr}_\Sigma: v \mapsto v\vert_\Sigma$ 
and the forward operator $F = \text{tr}_\Sigma \circ S: (s, b, \eta) \mapsto \text{tr}_\Sigma(u)$.
The domain of $F$ is defined as
\begin{equation}\label{eq:defX}
\begin{aligned}    
    \mathcal{D}(F) := \{(s,b,\eta) :  s , b, \tfrac{1}{s}, \tfrac{1}{b} \in L^\infty(\Omega)^+,\, \eta \in L^\infty(\Omega)\} \subseteq 
    X:=L^\infty(\Omega)^3
\end{aligned}
\end{equation}
where we define $L^\infty(\Omega)^+ := \{v \in L^\infty(\Omega) : v > 0 \text{ a.e. in } \Omega\}$. 
The location, shape and 
amplitude 
of exterior (supported on the boundary $\partial \Omega$) and interior (supported in the domain $\Omega$) sources $f$ and $g$ are assumed to be known. 
A natural space for observations of the pressure over time on $\Sigma$ that are possibly contaminated with random noise is 
\begin{equation}\label{eq:defY}
Y = \begin{cases}L^p(0,T;L^q(\Sigma))\text{ if $\Sigma$ is a smooth compact manifold,}\\
L^p(0,T;l^q(\Sigma))\text{ if $\Sigma$ is a discrete set}, 
\end{cases}
p,q\in[1,\infty].
\end{equation}
The PDE solution space $\solspace$ must therefore continuously be mapped into $Y$ by the trace operator $\text{tr}_\Sigma$ (applied point wise in time on $(0,T)$).
Its definition (cf. \eqref{eq:defH} below) results from the following analysis.
\subsection{Well-posedness of the parameter-to-state-map and Fr\'{e}chet differentiability of the forward operator}

To study Fr\'{e}chet differentiability of the forward operator, we begin our investigation by considering the difference $\text{tr}_\Sigma(w) := F(\ts,\tb,\teta) - F(s,b,\eta)$, where with $ds := \ts  - s $, $db:= \tb  - b $, $d\eta := \teta  - \eta $, $w:= \tu - u$ solves
\begin{align}
\label{eq:westervelt:error:pde}
    (b - 2\eta  u) w_{tt} - s \Delta w - \Delta w_t - &    2\eta  (\ut +\tut) w_t - 2 \eta  \tutt w  \\ \nonumber
    &= ds\Delta \tu - db \tutt + d \eta(\tu^2)_{tt}.
\end{align}
The linearity and continuity of the trace operator yields the formal linearisation 
\[
\text{tr}_{\Sigma} \du = F^\prime(s, b, \eta)(ds, db, d\eta)
\]
in $u$, where $\du$ solves
\begin{align}
\label{eq:westervelt:frechet:derivative}
     (b - 2\eta  u) \du_{tt} - s  \Delta \du - \Delta \du_t -& 4\eta  u_t \du_t \\ \nonumber
    &- 2 \eta  \utt \du =  ds \Delta u - db \utt + d\eta(u^2)_{tt}
\end{align}
The remainder reads $\text{tr}_{\Sigma} \remainder:= F(\ts,\tb,\teta) - F(s,b,\eta) - F^\prime(s,b,\eta)(ds, db, d\eta)$, where $\remainder$ solves
\begin{align}
\label{eq:Westervelt:frechet:derivative:remainder}
    &(b - 2\eta  u) \remainder_{tt} - s  \Delta \remainder - \Delta \remainder_t - 4\eta u_t \remainder_t - 2 \eta  \utt \remainder \\ \nonumber
    &= d\eta [(\tu^2)_{tt} - (u^2)_{tt}] + ds\Delta w - db w_{tt} - 2\eta  w_t^2 - 2 \eta  w w_{tt},
\end{align}
The PDEs \eqref{eq:westervelt:error:pde}, \eqref{eq:westervelt:frechet:derivative},
\eqref{eq:Westervelt:frechet:derivative:remainder} are equipped with homogeneous absorbing / impedance and T-periodicity conditions
\begin{equation}\label{eq:hombndy-Tperiodicity}
\begin{aligned}
&\beta \du_t + \gamma \du + \nabla \du \cdot \vecn = 0\text{ on }(0,T)\times\partial\Omega\\
&\du(0,x) = \du(T,x), \quad \du_t(0,x) = \du_t(T,x), \quad x\in\Omega
\end{aligned}
\end{equation}
and likewise for $w$ and $\remainder$.
%
Our goal is to show that $\lim_{(ds,db,d\eta) \rightarrow 0} \frac{\|\remainder\|_\solspace}{\|(ds,db,d\eta)\|_X} = 0$. For this purpose, we have to study the generalized linear periodic Westervelt equation with space-time dependent coefficients and space dependent model parameters which we do in the following theorem.
%
%
\begin{theorem}
\label{thm:two:sources:westervelt:linear:solution:existence:uniqueness:mixed:boundary:space:dependent}
    Let $\Omega \subseteq \mathbb{R}^d$, $d \in \{2,3\}$, open, bounded, connected, with its boundary $\partial\Omega \in C^{1,1}$, $T > 0$, $\beta$, $\frac{1}{\beta}$, $\gamma \in L^\infty(\partial\Omega)^+$, $\alpha, \frac{1}{\alpha} \in L^\infty (0,T;L^\infty(\Omega))\cap W^{1,\infty}(0,T,L^{\infty}(\Omega))$, $\alpha(0) = \alpha(T)$, $\mathfrak{b}, \frac{1}{\mathfrak{b}}, c, \frac{1}{c} \in L^\infty(\Omega)$, $b,c>0$, and $\delta \in L^2(0,T;L^2(\Omega))$, $\mu \in L^\infty(0,T;L^{2q/(q-1)}(\Omega))$, $ q \in [1,\infty)$. The sources are given by $f \in L^2(0,T;L^2(\Omega))$ and 
    $g \in H^{1}(0,T;H^{1/2}(\partial \Omega))$. 
Moreover, assume that there exists a constant $\nu>0$ such that 
    \begin{enumerate}
        \item $ \nu < \tfrac{\gamma}{\beta} $ a.e. on $\partial \Omega$,
        \item $\sqrt{2 C_{\text{PF}}} C_{H^{3/2} \rightarrow L^{2q}} C_\Omega \|\nu - \tfrac{\mu}{\alpha}\|_{L^\infty(0,T;L^{2q/(q-1)}(\Omega))} < \tfrac{\mathfrak{b}}{\alpha}$ a.e. in $(0,T) \times \Omega$,
        \item and $\|(\tfrac{1}{2\alpha})_t(c^2 + \mathfrak{b})\|_{L^\infty(0,T;L^\infty(\Omega))}^2 + \sqrt{2 C_a}\|\tfrac{\delta}{\alpha}\|_{L^2(0,T;L^2(\Omega))} < \tfrac{\nu c^2}{2\alpha}$ a.e. in $(0,T) \times \Omega$, where $C_a := \tfrac{1}{\lambda_1^2}$ and $\lambda_1$ denotes the smallest eigenvalue of the negative impedance Laplacian $-\Delta_\gamma$.
    \end{enumerate}  
    
    Then there exists a unique (weak) solution $u$ of
\begin{equation}
\label{eq:two:sources:westervelt:linear:periodic:space:dependent}
\begin{cases}
\alpha u_{tt} - c^2 \Delta u - \mathfrak{b} \Delta u_t + \mu u_t + \delta u= f & \text{in } (0,T) \times \Omega,\\
\beta u_t + \gamma u + \nabla u \cdot \vecn = g  & \text{on } (0,T) \times \partial\Omega, \\
u(0) = u(T), \, u_t(0) = u_t(T) & \text{in } \Omega, \\
\end{cases}   
\end{equation}
%
with 
\begin{equation}\label{eq:defH}
\begin{aligned}
    u \in \solspace:=\{v \in H^2(0,T;L^2(\Omega)) \cap H^1(0,T;H^{3/2}(\Omega)) \cap L^2(0,T;H^2(\Omega)):  \\ 
    \| \nabla\trace{\Omega}{v} \cdot \vecn \|_{H^1(0,T;L^2(\partial\Omega))} < \infty, v(0) = v(T), v_t(0) = v_t(T) \text{ a.e.}\}
\end{aligned}
\end{equation}
and the solution $u$ satisfies
\begin{align}
    \nonumber
    \|u\|_\solspace^2 = & \|u\|_{H^2(0,T;L^2(\Omega))}^2 + \|u\|_{H^1(0,T;H^{3/2}(\Omega))}^2 \\ \nonumber
    & + \|u\|_{L^2(0,T;H^2(\Omega))}^2 + \| \nabla\trace{\Omega}{u} \cdot \vecn \|_{H^1(0,T;L^2(\partial\Omega))}^2 \\ \nonumber
& \leq C(C_\alpha, b,\gamma,c,\beta, T, \Omega)^2 \left( \| f \|_{L^2(0,T;L^2(\Omega))} + \| g \|_{ H^1(0,T;H^{1/2}(\partial\Omega))} \right)^2 ,
\end{align}
where $C(C_\alpha, b,\gamma,c,\beta, T, \Omega) > 0$. 
\end{theorem}
\begin{proof}
    See supplemental material~\ref{appendix:proof:theorem:linear:well-posedness}.
\end{proof}
\begin{remark}
    For $\beta = 0$ on the boundary $\partial \Omega$, Theorem~\ref{thm:two:sources:westervelt:linear:solution:existence:uniqueness:mixed:boundary:space:dependent} is still valid under analogous smallness conditions on the parameters, see \cite{periodicJMGT}. 
\end{remark}
%
%
With the help of Theorem~\ref{thm:two:sources:westervelt:linear:solution:existence:uniqueness:mixed:boundary:space:dependent} we first of all show that there exists a unique solution to~\eqref{eq:westervelt:quadratic-nonlinearity:periodic} provided the sources fulfill a smallness condition. We summarize this in the next theorem.
\begin{theorem}
\label{thm:westervelt:nonlinear:solution:existence:uniqueness:source:boundary}
Let $\Omega \subseteq \mathbb{R}^d$, $d \in \{2,3\}$, open, bounded, connected, with its boundary $\partial\Omega \in C^{1,1}$, $\beta$, $\gamma$, $T > 0$, $b, \frac{1}{b}, s, \frac{1}{s} \in L^\infty(\Omega)$, $b,s>0$. Then there exists $\Lambda > 0$ such that for all 
$\hh \in L^2(0,T;L^2(\Omega))$, 
$g\in H^1(0,T;H^{1/2}(\partial\Omega))$ 
with $\|\hh\|_{L^2(0,T;L^2(\Omega))} 
+ \|g\|_{H^1(0,T;H^{1/2}(\partial\Omega))} 
\leq \Lambda$ there exists a unique (weak) solution $u \in \solspace$ of
\begin{equation}
\label{eq:westervelt:nonlinear:periodic}
\begin{cases}
bu_{tt} - s \Delta u - \Delta u_t = \eta (u^2)_{tt} + \hh & \text{in }  (0,T) \times \Omega,\\
\beta u_t + \gamma u + \nabla u \cdot \vecn = g & \text{on } (0,T) \times \partial\Omega, \\
u(0) = u(T), \, u_t(0) = u_t(T) & \text{in } \Omega, \\
\end{cases}   
\end{equation}
and the solution $u$ satisfies
\begin{align}
    \nonumber
    \|u\|_\solspace^2 = & \|u\|_{H^2(0,T;L^2(\Omega))}^2 + \|u\|_{H^1(0,T;H^{3/2}(\Omega))}^2 + \\ \nonumber
    &\|u\|_{L^2(0,T;H^2(\Omega))}^2 + \| \nabla\trace{\Omega}{v} \cdot \vecn \|_{H^1(0,T;L^2(\partial\Omega))}^2 \\ \nonumber
    & \leq C(C_\alpha, b,\gamma,c,\beta, T, \Omega)^2 (\|\hh\|_{L^2(0,T;L^2(\Omega))} + \| g \|_{ H^1(0,T;H^{1/2}(\partial\Omega))})^2,
\end{align}
where $C(C_\alpha, b,\gamma,c,\beta, T, \Omega) > 0$. 
\end{theorem}
\begin{proof}
See supplemental material~\ref{appendix:proof:theorem:nonlinear:well-posedness}.
\end{proof}
Indeed, the conditions on the coefficients imposed in Theorem~\ref{thm:two:sources:westervelt:linear:solution:existence:uniqueness:mixed:boundary:space:dependent} are fulfilled in the nonlinear case for arbitrary $s,b \in L^\infty(\Omega)^+$, $\eta \in L^\infty(\Omega)$, $\gamma,\beta \in L^\infty(\partial \Omega)^+$ if the sources fulfill the imposed smallness condition (which actually is coupled to $\eta$). Now, that we have set the stage, we can finally  investigate the remainder $\remainder(t,x)$ for Fr\'{e}chet differentiability of $S$, cf. \eqref{eq:Westervelt:frechet:derivative:remainder}. Let the assumptions of Theorems~\ref{thm:two:sources:westervelt:linear:solution:existence:uniqueness:mixed:boundary:space:dependent} and~~\ref{thm:westervelt:nonlinear:solution:existence:uniqueness:source:boundary} hold. First, we apply Theorem~\ref{thm:two:sources:westervelt:linear:solution:existence:uniqueness:mixed:boundary:space:dependent} to \eqref{eq:Westervelt:frechet:derivative:remainder} and obtain that there exists $C > 0$ such that
\begin{align}
\label{eq:frechet:deriv:remainder:estimate:1}
    \|\remainder\|_\solspace \leq & C \|  d\eta 
    [(\tu+u)w]_{tt} 
    + ds\Delta w - db w_{tt} - 2\eta w_t^2 - 2 \eta w w_{tt} \|_{L^2(0,T;L^2(\Omega))}.
\end{align}
Now it suffices to further estimate the right hand side. We do so by considering the fact that 
%
Theorem~\ref{thm:westervelt:nonlinear:solution:existence:uniqueness:source:boundary} as well as \eqref{eq:westervelt:error:pde} together with 
Theorem~\ref{thm:two:sources:westervelt:linear:solution:existence:uniqueness:mixed:boundary:space:dependent} and Sobolev embeddings imply existence of a constant $C > 0$ such that 
\[
\begin{aligned}
&\|[(\tu+u)w]_{tt}  + ds\Delta w - db w_{tt} - 2\eta w_t^2 - 2 \eta w w_{tt} \|_{L^2(0,T;L^2(\Omega))}\\
&    \leq C \|(ds,db,d\eta)\|_X   (\|\hh\|_{L^2(0,T;L^2(\Omega))} + \|g\|_{H^1(0,T;H^{1/2}(\partial\Omega)}).
\end{aligned}
\]
Hence, estimating the right hand side of~\eqref{eq:frechet:deriv:remainder:estimate:1} we obtain that there exists a constant $\Tilde{C} > 0$ independent of $\remainder, ds, db$, and $d\eta$, such that
\begin{align}
    \|\remainder \|_\solspace\leq \Tilde{C} \|(ds,db,d\eta)\|_X^2 (\|\hh\|_{L^2(0,T;L^2(\Omega))} + \|g\|_{H^1(0,T;H^{1/2}(\partial\Omega))}).
\end{align}
This shows that the remainder is $o(\|(ds,db,d\eta)\|_X)$ and establishes Fr\'{e}chet differentiability of $S$. We summarize our findings in the following theorem.
\def\ts{\Tilde{s}}
\def\tb{\Tilde{b}}
\def\t\kappa{\Tilde{\kappa}}
\begin{theorem}
    Let the assumptions of Theorem~\ref{thm:westervelt:nonlinear:solution:existence:uniqueness:source:boundary} on the parameters and sources be fulfilled. 
    
    Then for any $(s,b,\eta)$ in the interior of $\mathcal{D}(F)$, the Fr\'{e}chet derivative $F^\prime(s,b,\eta)$ of the forward operator $F: \mathcal{D}(F)\subseteq X \rightarrow Y$, cf. \eqref{eq:defX}, \eqref{eq:defY} exists, is unique, and $F^\prime(s,b,\eta)(ds,db,d\eta)=\text{tr}_{\Sigma}\du$ where $\du$ solves \eqref{eq:westervelt:frechet:derivative} with homogeneous absorbing / impedance and $T$-periodicity conditions  \eqref{eq:hombndy-Tperiodicity}.
\end{theorem}

\section{Reconstructing \texorpdfstring{$s$, $b$, and $\eta$}{s, b, and eta} from boundary measurements
}
\label{sec:recon:inverse:problem}
In this section we are mainly interested in reconstructing $s, b$ and $\eta$ from possibly incomplete and noisy boundary measurement $h$ cf.~\eqref{eq:obs}, where the acoustic wave propagation through the domain $\Omega$ is excited by a boundary source. We further set $\beta = 0$ on the boundary. In what follows we take the preparatory steps for a frozen Newton-type regularization scheme. We start to do so by investigating linearised uniqueness of an all-at-once formulation that provides additional freedom by not necessarily having to constrain the reference state to be a solution to the  PDE~\eqref{eq:westervelt:quadratic-nonlinearity:periodic}. 

\input{linearised_uniqueness}
\input{IRGNM_L2L2}

\section{Numerical experiments}
\label{sec:experiments}
In this section, we present numerical experiments in two spatial dimensions. The solver used to conduct these experiments can be found in~\cite{RainerSolver26}. It includes a 2D conforming element FEM solver for the periodic nonlinear Westervelt equation (re-parameterised, cf.~\eqref{eq:westervelt}) based on a multiharmonic expansion. (see~\cite{RK2025} and the supplemental material~\ref{appendix:multiharmonic:ansatz}), as well as its Fréchet derivative and the corresponding adjoint in the Hilbert space setting used here. Furthermore, it implements the frozen Newton method outlined in Section~\ref{sec:reconstruction:algo}. However, the lifting of the parameters is practically not required, as we do not need to compute the effective increments once we ensure that start sufficiently near to the actual solution. The iterates defined in~\eqref{identification:iterates:regularized:formulation} are computed using a conjugate gradient method. 

In what follows we present three numerical studies in which we attempt to reconstruct phantoms enclosed in a circular domain with radius of $0.2$ $m$, i.e., $\Omega = B_{0.2}(0)$. In all cases we use as boundary source, with different frequencies, the Robin trace of the example function discussed in Remark~\ref{rem:example} with Robin parameter $\gamma = 1$ on the full boundary $\partial \Omega$. The boundary sources are appropriately scaled to ensure that the assumptions of Theorem~\ref{thm:linearised:uniqueness} are satisfied. For generating the regularisation parameters, we set $\alpha_0 = 1$ and $q = 0.6$. The Dirichlet boundary measurements are taken from all three reference states over the duration of one period. For the numerical experiments we set $\vxz$ to the values of the domain, i.e., $(s^0,b^0,0)$.

\underline{\textit{Case 1 - a single phantom}}. First, we study the reconstruction of a single phantom inside $\Omega$ from full and partial boundary measurements under possible random (white) noise. We choose the following values for the numerical experiments. The inclusion is circular with a radius of $0.03$, and its center is at $(0,0.1)^T$. It has a $B/A$ value of $7$, a speed of sound of $c_\text{domain} = 10.11$ $m/s$ and a diffusivity of $\mathfrak{b} = 0.051$ $m^2 s^{-1}$. For the surrounding domain, we consider a speed of sound of $c_\text{domain} = 10$ $m/s$ and a diffusivity of $\mathfrak{b}_\text{domain} = 0.05$ $m^2 s^{-1}$. We set the mass density $\rho = 1000$ $kg/m^3$. From these values, we compute the transformed parameters according to~\eqref{etabs} (cf. figure~\ref{fig:case:1}, first row). 

The reconstruction results are shown in Figure~\ref{fig:case:1}. The second row shows the reconstruction under full boundary measurements without noise after 20 Newton iterations. The location of the inclusions is correctly reconstructed, and we can clearly see that the parameters of the phantom have the correct qualitative behavior. The third row shows the reconstruction under $0.1\%$ white noise on the measurements. After 12 Newton iterations, artefacts near the boundary become visible due to the noisy data. Despite these perturbations, the location of the phantom and the correct qualitative behavior of the amplitudes are still recovered.

In Figure~\ref{fig:case:1:partialboundary} we show the case where measurements are taken on part of the boundary, indicated by a black arc. The first row shows the reconstruction of a single parameter set ($s,b,\eta$) after the first Newton iteration. The second row shows results after 30 Newton iterations; one observes that the amplitude of the phantom in the parameters have increased/decreased appropriately and that the false positives from the first iteration are barely visible. To contrast the linear case in which we don't expect uniqueness due to the lack of higher harmonics, in the third row shows the reconstruction in the linear case after 20 Newton iterations, i.e., reconstructing $s$ and $b$ from the case $\eta = 0$ in~\eqref{eq:westervelt:quadratic-nonlinearity:periodic}. 
Here, a spurious inclusion persists and the true phantom cannot be identified. 
In Figure~\ref{fig:J:partial:boundary:meas:nonlin:vs:lin} we show $J_n(\vx,\vx_n)$ (cf.~\eqref{eq:cost:function}) for the nonlinear and linear case with measurements taken on the partial boundary.
This clearly illustrates the benefit of nonlinearity also for the reconstruction of linear parameters.

\underline{\textit{Case 2 - two phantoms}}. Second, we study the case of two phantoms, each of them having different parameter values (cf. Figure~\ref{fig:case:2}). The actual parameter values are as follows: $B/A = 7$ for the left phantom and $B/A = 6$ for the right phantom, diffusivity in the domain  $\mathfrak{b} = 0.05$ $m^2 s^{-1}$, diffusivity in the phantoms $\mathfrak{b}=0.051$ $m^2 s^{-1}$, speed of sound in the domain $c = 10$ $m/s$, speed of sound in the phantoms $c = 10.15$ $m/s$, mass density $\rho = 1000$ $kg/m^3$. Hence, the phantoms differ only in the nonlinearity parameter. The reconstructions from twenty Newton iterations are show in Figure~\ref{fig:case:2}. The right phantom is reconstructed with higher intensity despite having a lower nonlinearity parameter. However, since it is closer to the boundary, it has a stronger influence on the measurement data. This phenomenon has also been identified and studied in~\cite{Kaltenbacher_2024}.
\begin{figure}
\centering
    \begin{subfigure}{0.3\textwidth}
        \includegraphics[width=1.0\linewidth]{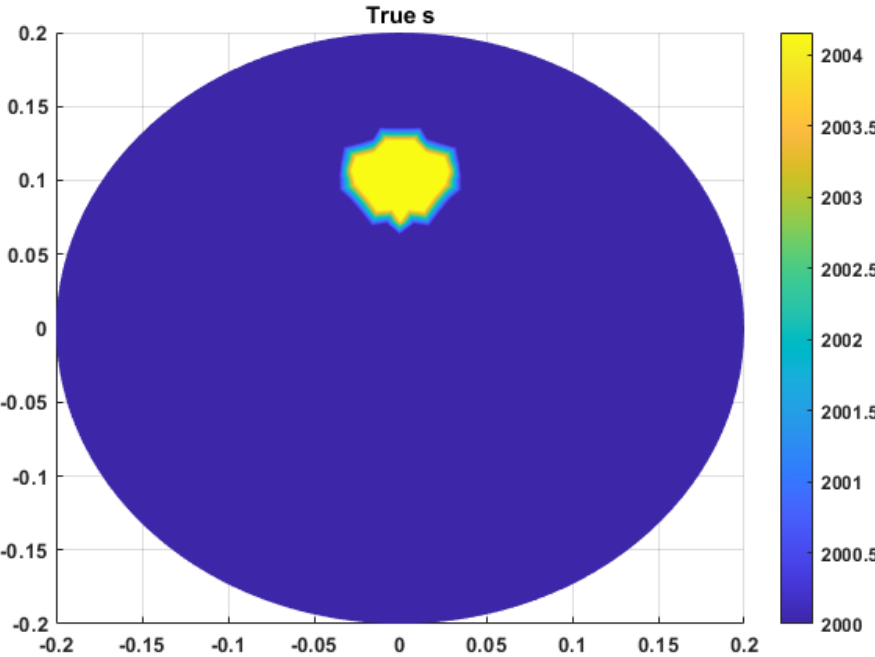} 
        \end{subfigure}
    \begin{subfigure}{0.3\textwidth}
        \includegraphics[width=1.0\linewidth]{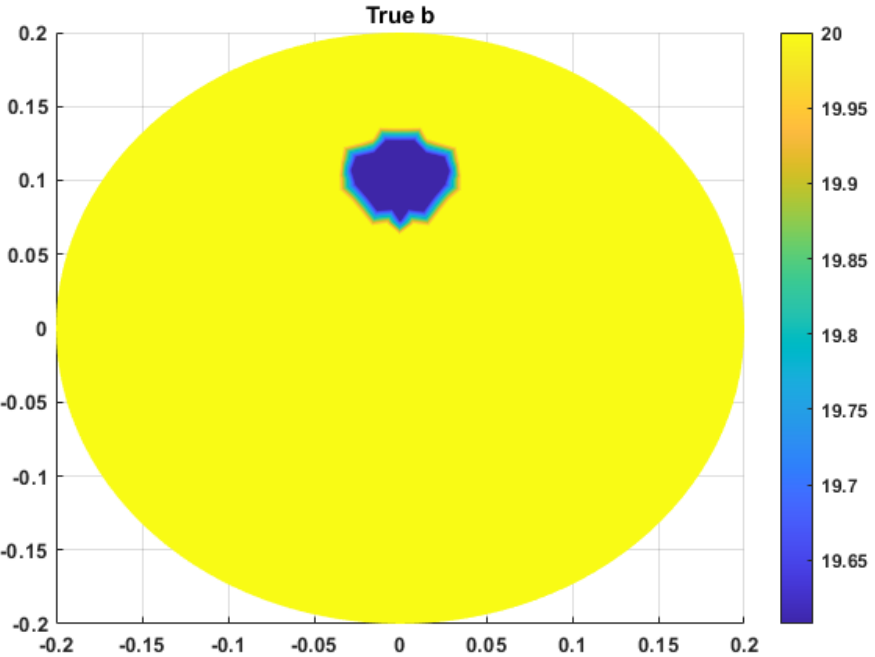} 
        \end{subfigure}
    \begin{subfigure}{0.3\textwidth}
        \includegraphics[width=1.0\linewidth]{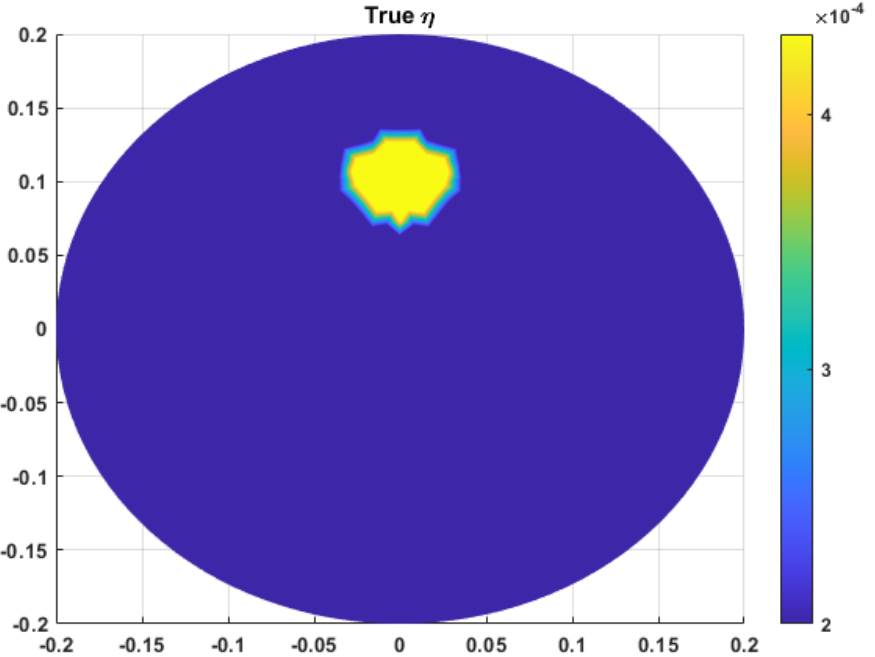} 
        \end{subfigure}
\bigskip 
        \begin{subfigure}{0.3\textwidth}
        \includegraphics[width=1.0\linewidth]{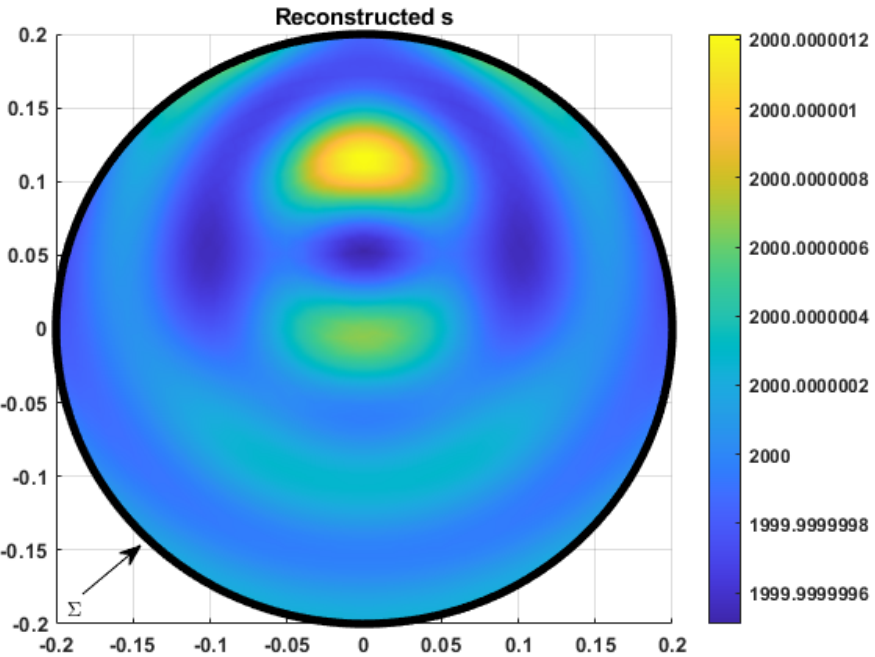} 
        \end{subfigure}
    \begin{subfigure}{0.3\textwidth}
        \includegraphics[width=1.0\linewidth]{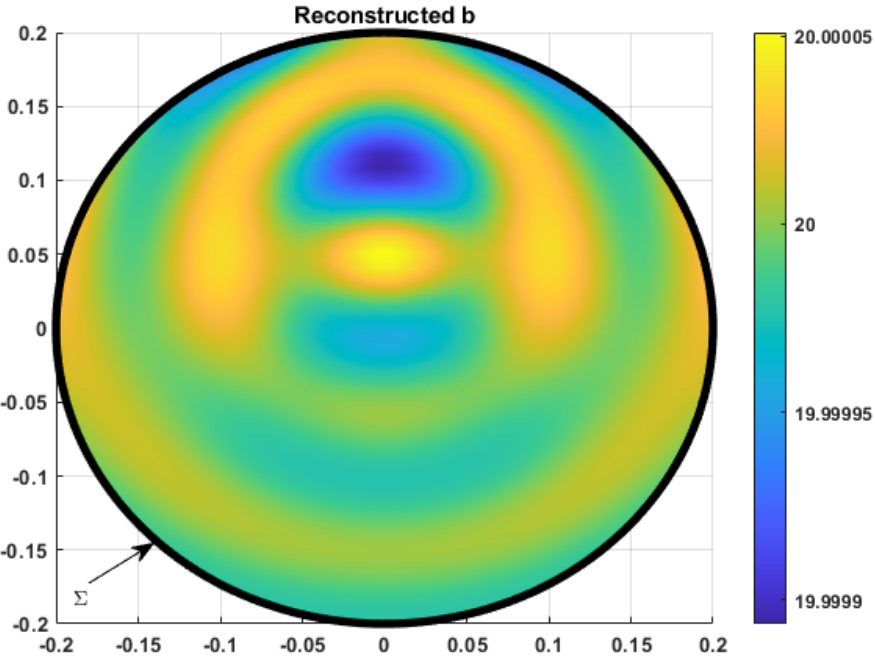} 
        \end{subfigure}
    \begin{subfigure}{0.3\textwidth}
        \includegraphics[width=1.0\linewidth]{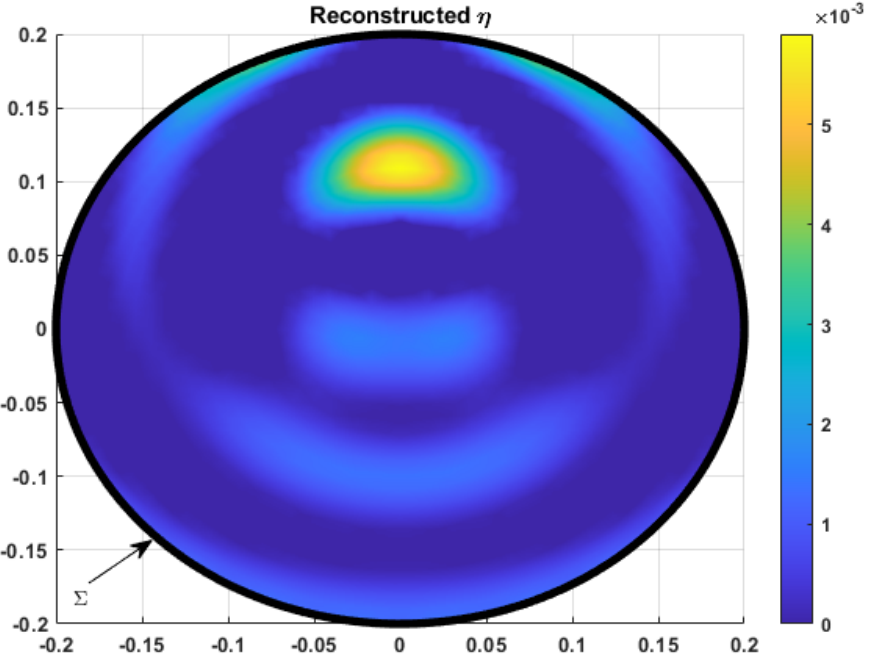} 
     \end{subfigure}
\bigskip 
     \begin{subfigure}{0.3\textwidth}
        \includegraphics[width=1.0\linewidth]{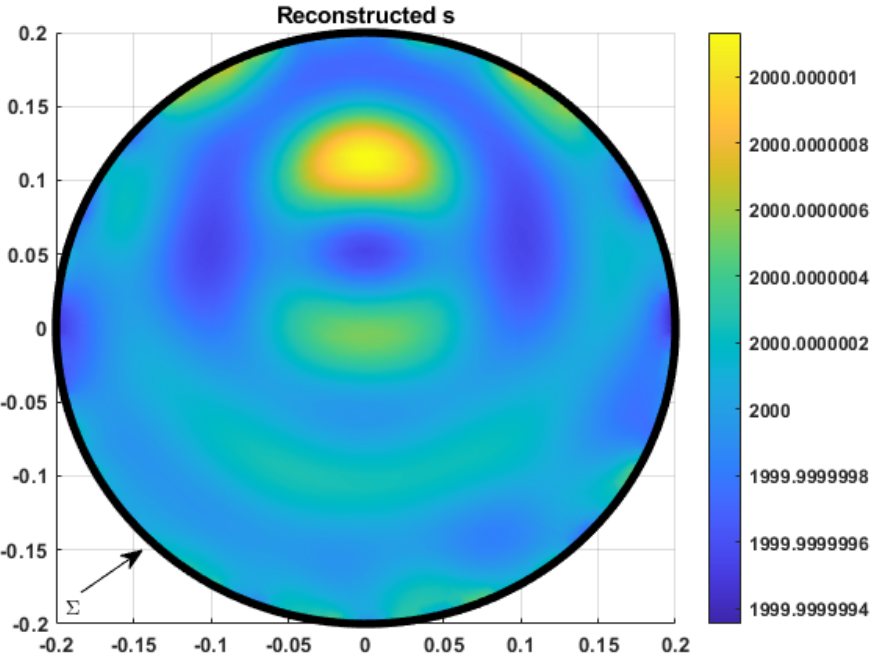} 
        \end{subfigure}
 \begin{subfigure}{0.3\textwidth}
        \includegraphics[width=1.0\linewidth]{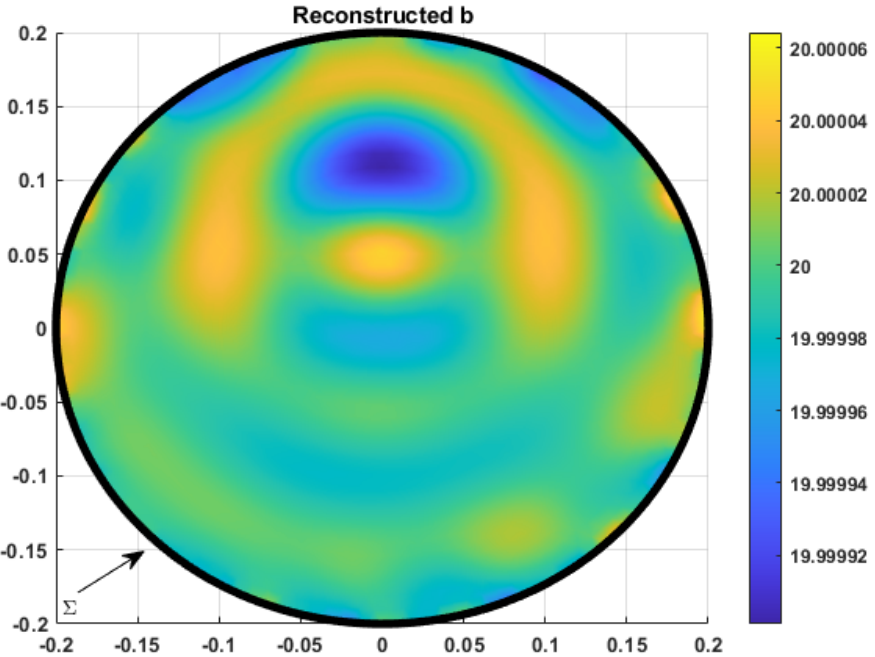} 
        \end{subfigure}
    \begin{subfigure}{0.3\textwidth}
        \includegraphics[width=1.0\linewidth]{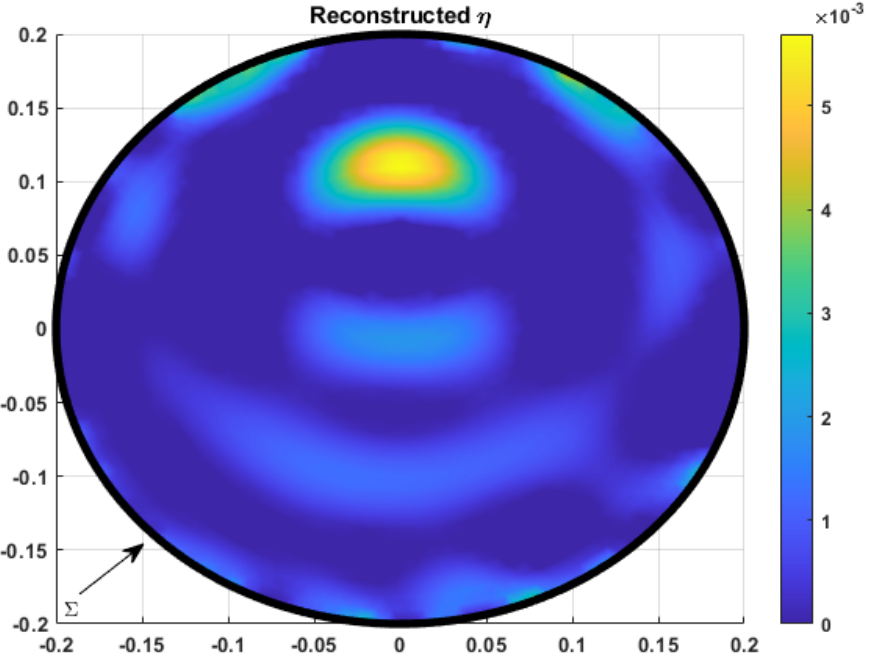} 
        \end{subfigure}
     \caption{Top row: true phantom parameters; mid row: reconstructions after 20 Newton iterations with full boundary information; bottom row: reconstructions after 12 Newton iterations with $0.1\%$ noise having full boundary information. }
     \label{fig:case:1}
\end{figure}
\begin{figure}
\centering
    \begin{subfigure}{0.3\textwidth}
        \includegraphics[width=1.0\linewidth]{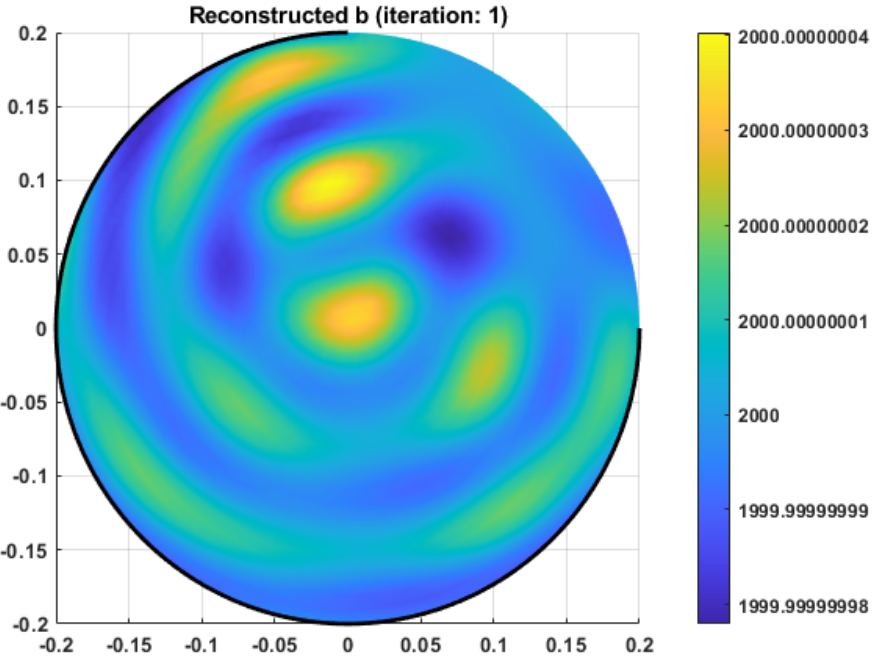} 
        \end{subfigure}
    \begin{subfigure}{0.3\textwidth}
        \includegraphics[width=1.0\linewidth]{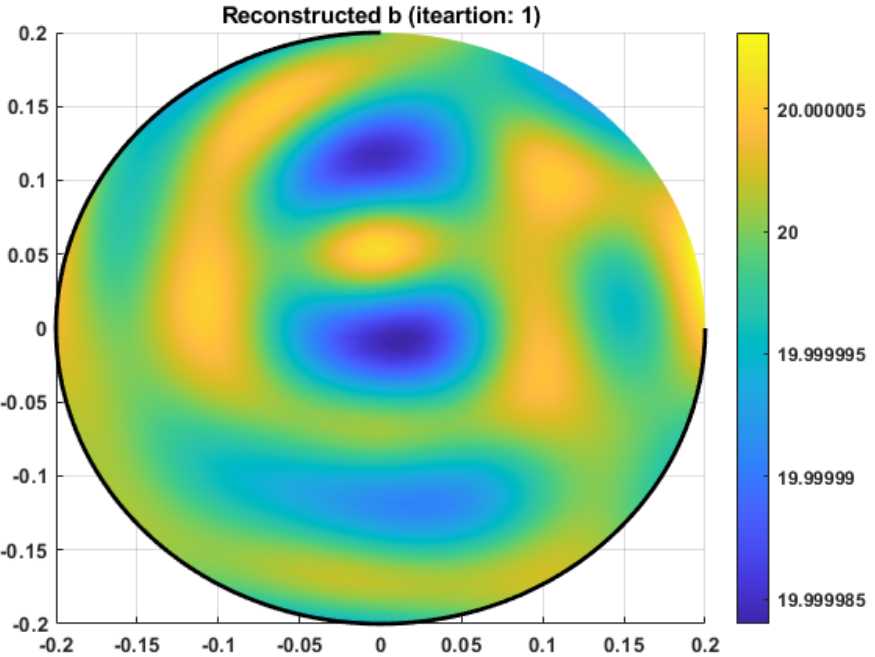} 
        \end{subfigure}
    \begin{subfigure}{0.3\textwidth}
        \includegraphics[width=1.0\linewidth]{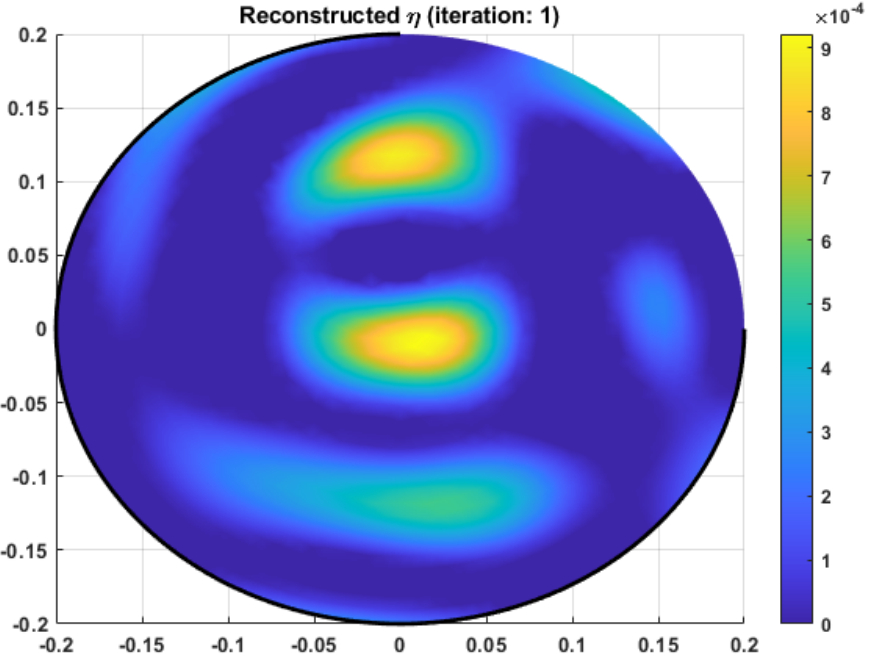} 
        \end{subfigure}
\centering
    \begin{subfigure}{0.3\textwidth}
        \includegraphics[width=1.0\linewidth]{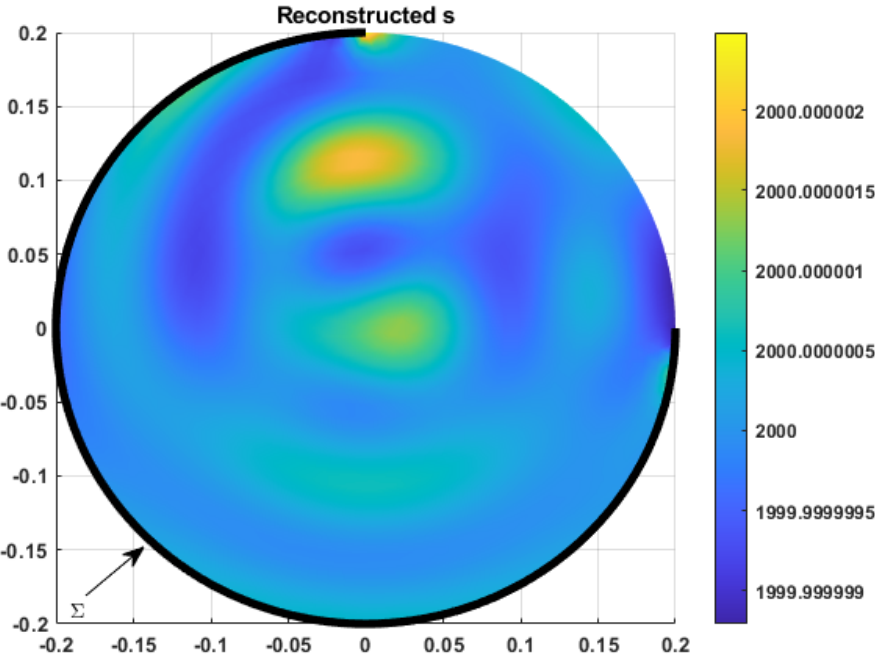} 
        \end{subfigure}
    \begin{subfigure}{0.3\textwidth}
        \includegraphics[width=1.0\linewidth]{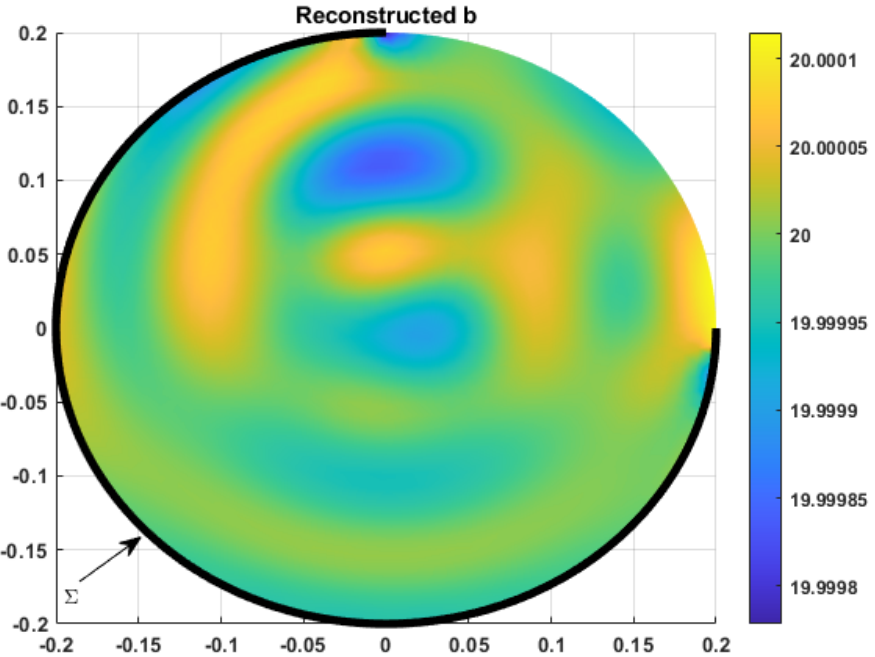} 
        \end{subfigure}
    \begin{subfigure}{0.3\textwidth}
        \includegraphics[width=1.0\linewidth]{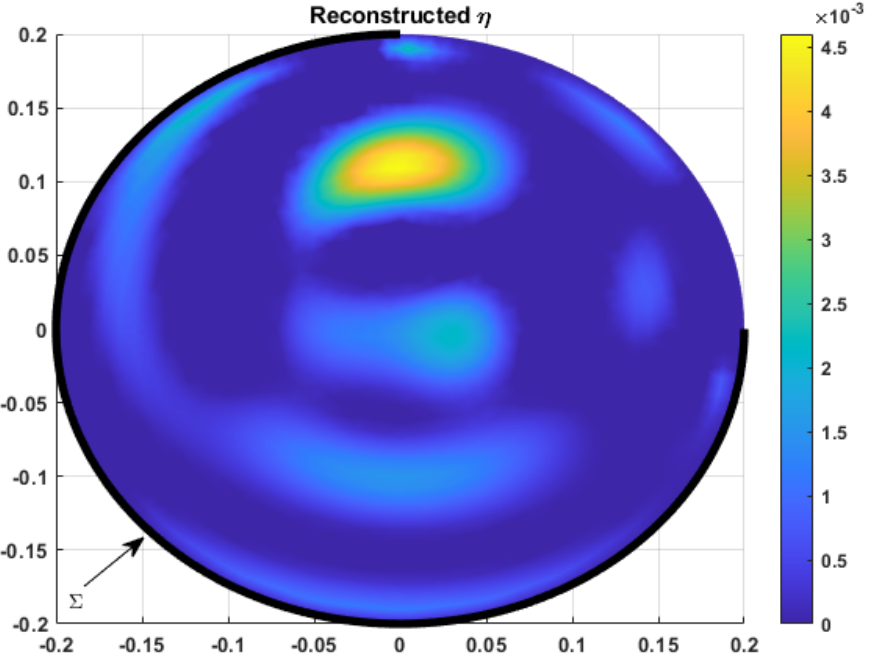} 
    \end{subfigure}
\centering
    \begin{subfigure}{0.3\textwidth}
        \includegraphics[width=1.0\linewidth]{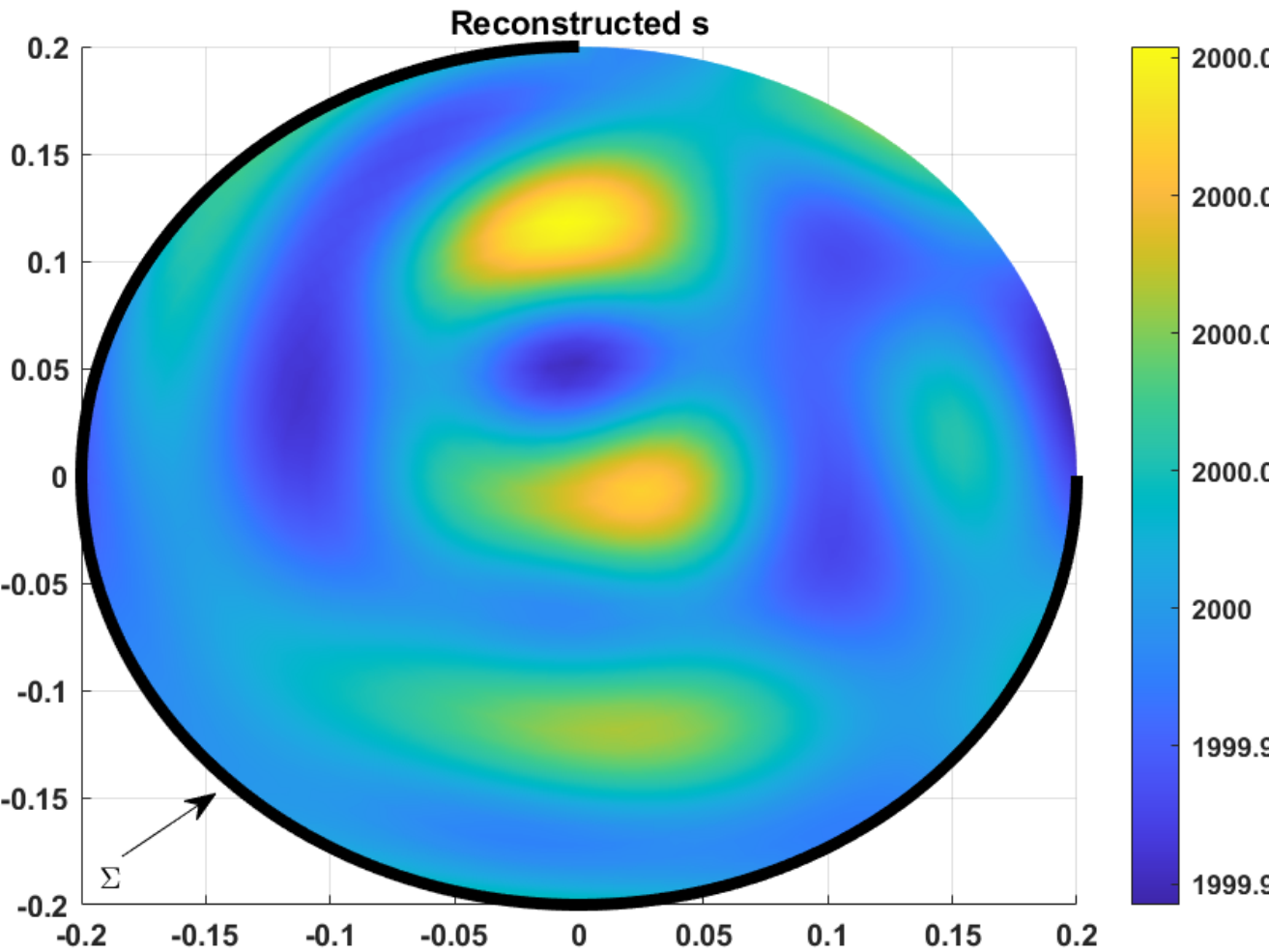} 
        \end{subfigure}
    \begin{subfigure}{0.3\textwidth}
        \includegraphics[width=1.0\linewidth]{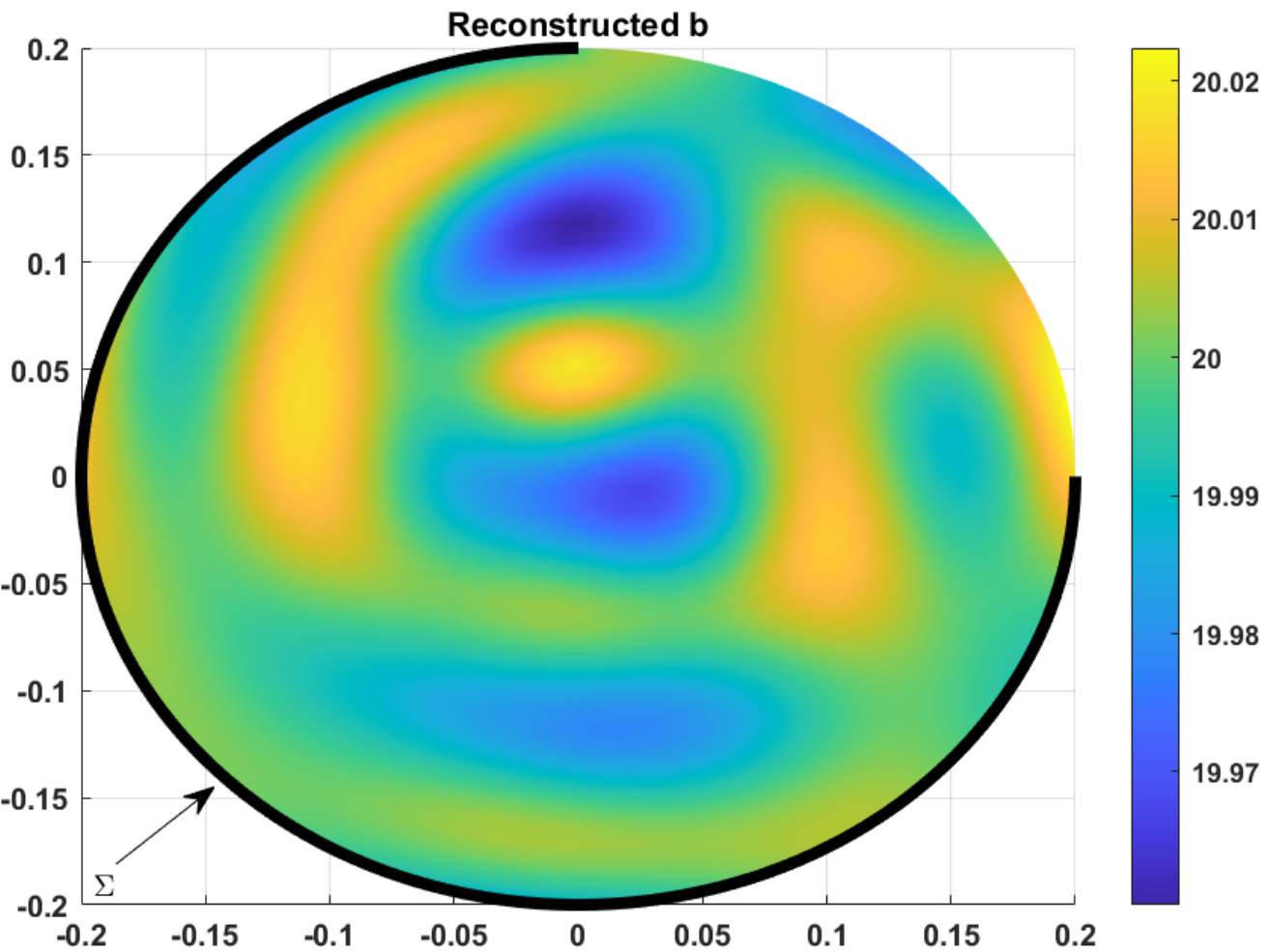} 
        \end{subfigure}
    \begin{subfigure}{0.3\textwidth}
        \includegraphics[width=1.0\linewidth]{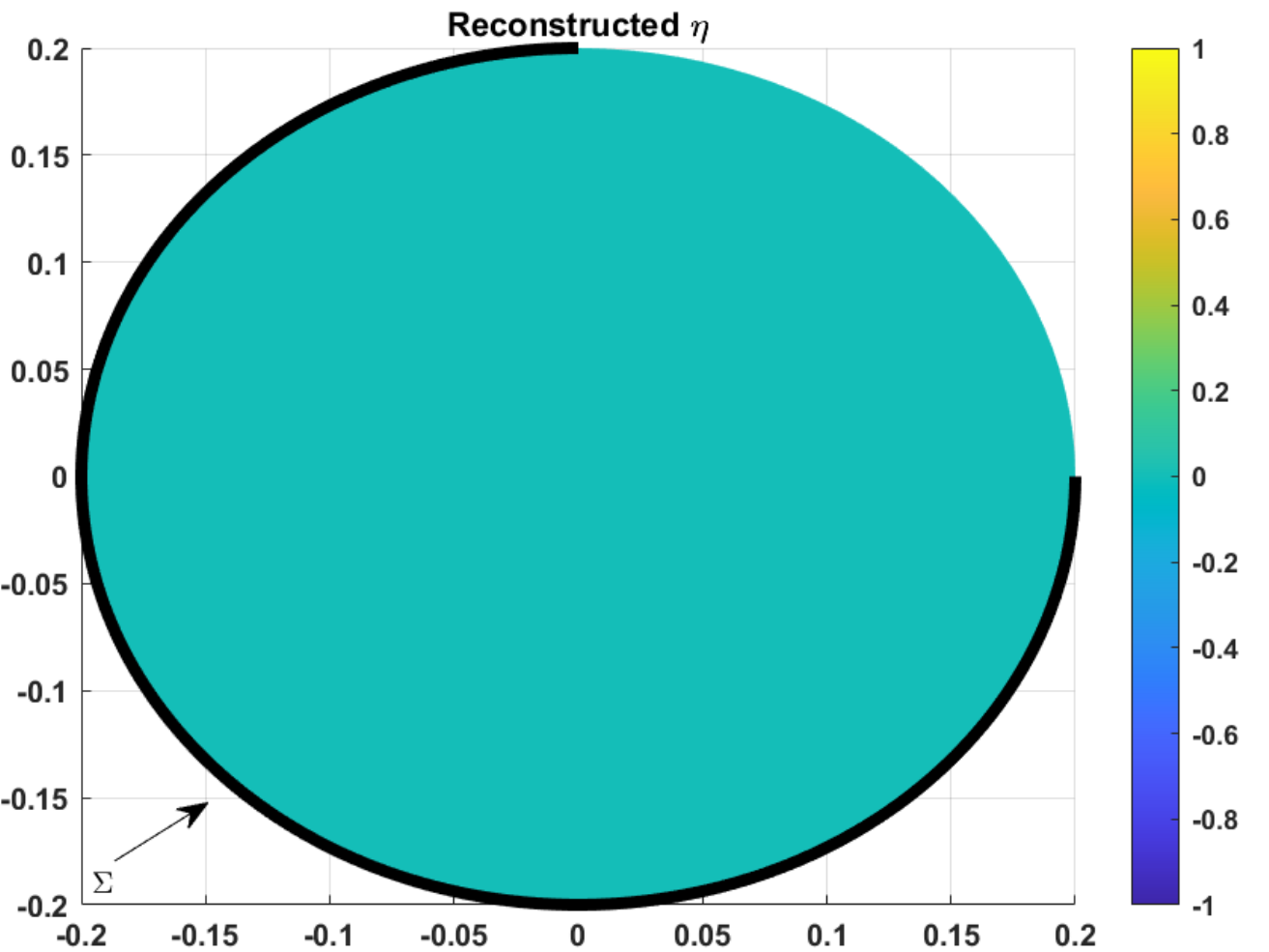} 
        \end{subfigure}
        \caption{Top row: reconstructions after the first Newton iteration under partial boundary measurements; mid row: reconstructions after 20 Newton iterations; bottom row: linear case after 20 Newton iterations.}
\label{fig:case:1:partialboundary}
\end{figure}
\begin{figure}
    \centering
    \includegraphics[width=0.65\textwidth]{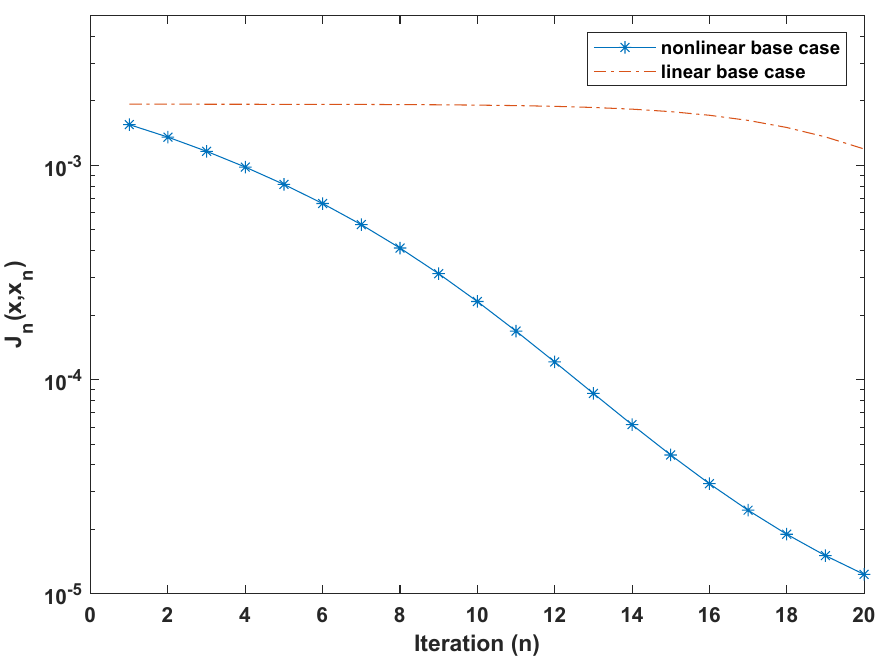}
    \caption{$J_n(\vx,\vx_n)$ on a logarithmic scale for case 1 with measurements taken on the partial boundary, nonlinear vs. linear.}
    \label{fig:J:partial:boundary:meas:nonlin:vs:lin}
\end{figure}
\begin{figure}
\centering
    \begin{subfigure}{0.3\textwidth}
        \includegraphics[width=1.0\linewidth]{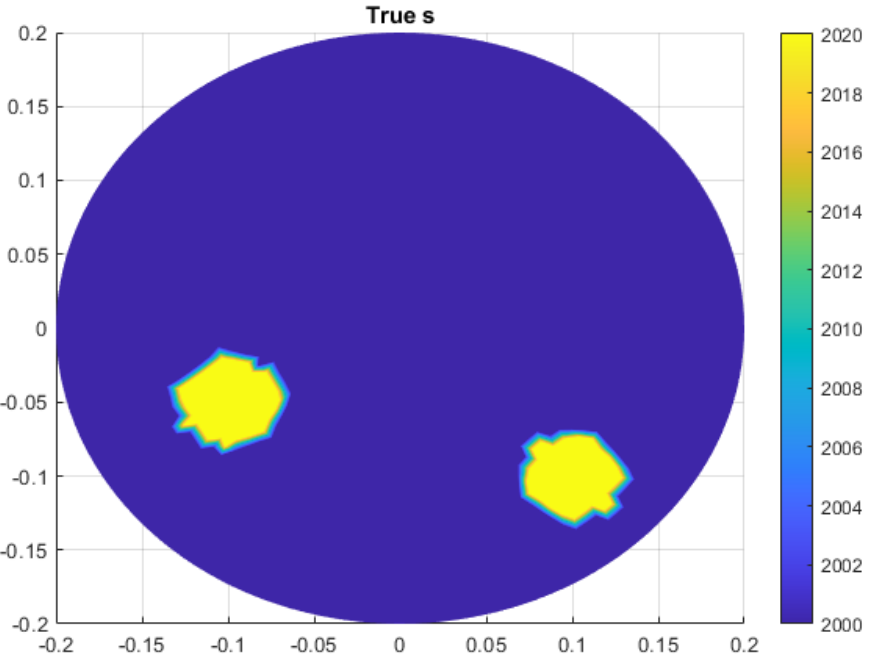} 
        \end{subfigure}
    \begin{subfigure}{0.3\textwidth}
        \includegraphics[width=1.0\linewidth]{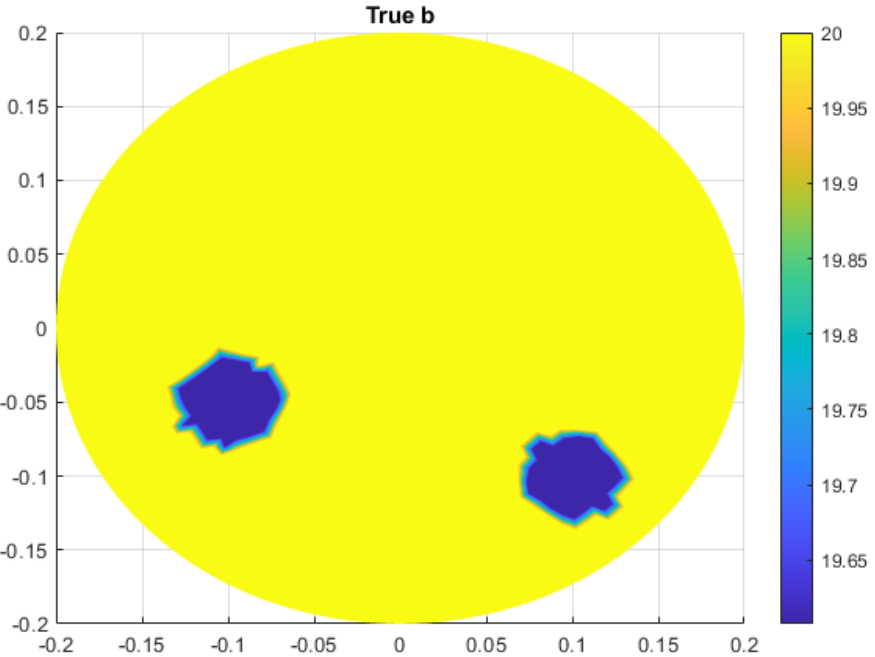} 
        \end{subfigure}
    \begin{subfigure}{0.3\textwidth}
        \includegraphics[width=1.0\linewidth]{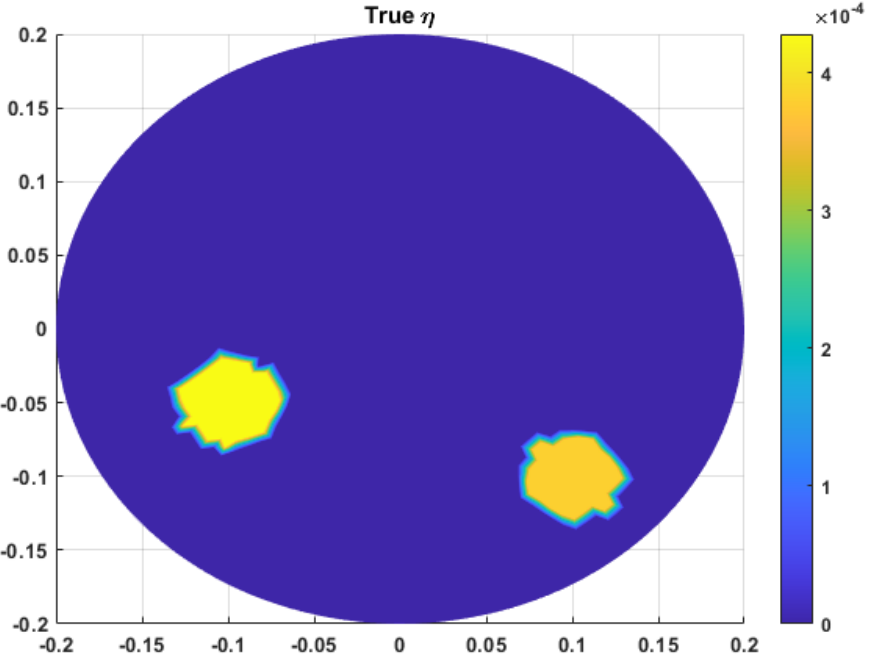} 
        \end{subfigure}
        \bigskip
    \begin{subfigure}{0.3\textwidth}
        \includegraphics[width=1.0\linewidth]{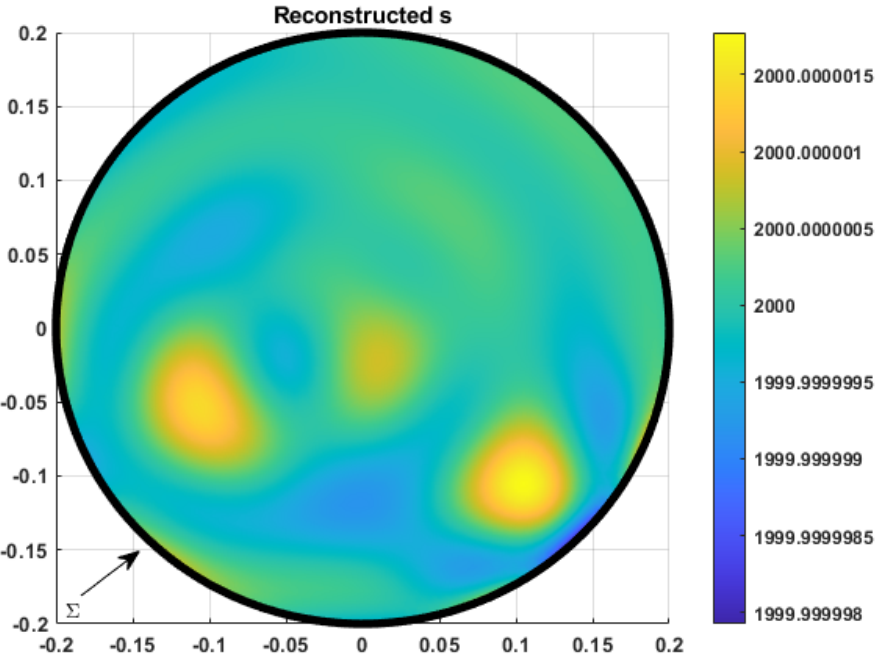} 
    \end{subfigure}
    \begin{subfigure}{0.3\textwidth}
        \includegraphics[width=1.0\linewidth]{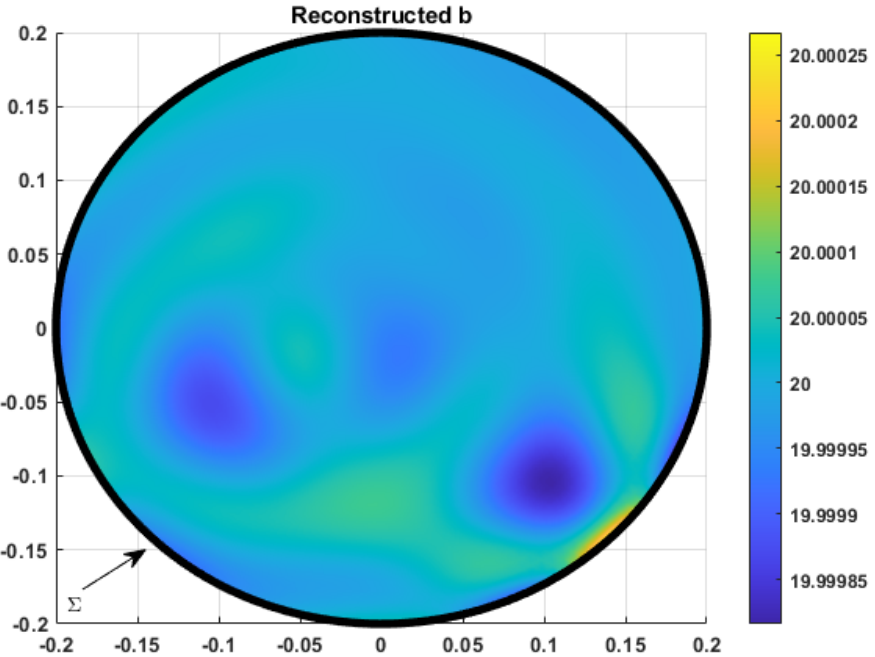} 
        \end{subfigure}
    \begin{subfigure}{0.3\textwidth}
        \includegraphics[width=1.0\linewidth]{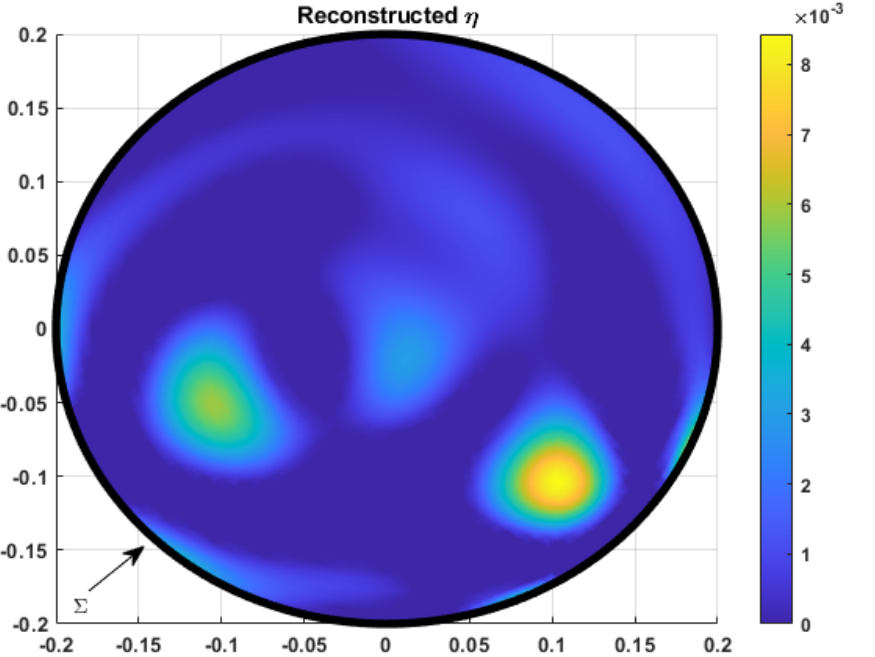} 
       \end{subfigure}
        \caption{First row: actual locations of two phantoms; second row: reconstructions after 20 Newton iterations.}
        \label{fig:case:2}
\end{figure}

\underline{\textit{Case 3 - three phantoms with spatially separated supports}}. Third and finally, we study the most challenging case in which  three inclusions are placed inside the domain, each differing in one of the three parameters. The phantoms are positioned such that they have the same distance to the boundary (cf. figure~\ref{fig:case:3}). The actual parameter values are as follows: $B/A = 7$, phantom diffusivity $\mathfrak{b}=0.052$ $m^2 s^{-1}$, and phantom speed of sound $c = 10.005$ $m/s$. For the domain, we set $\mathfrak{b} = 0.05$ $m^2 s^{-1}$, speed of sound $c = 10$ $m/s$, and mass density $\rho = 1000$ $kg/m^3$. 

Figure~\ref{fig:case:3} in the first row shows re-parameterised parameter values as well as the location and extent of the phantoms. The second row shows the reconstruction of the parameters after 14 Newton iterations. The three phantoms can be clearly distinguished. 
\begin{figure}[ht]
\centering
    \begin{subfigure}{0.3\textwidth}
        \includegraphics[width=1.0\linewidth]{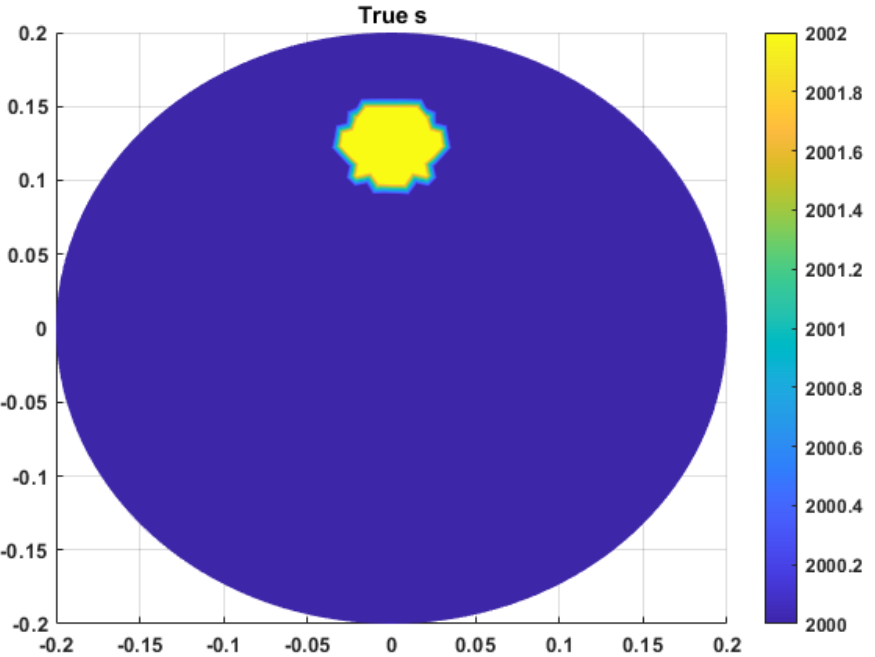} 
        \end{subfigure}
    \begin{subfigure}{0.3\textwidth}
        \includegraphics[width=1.0\linewidth]{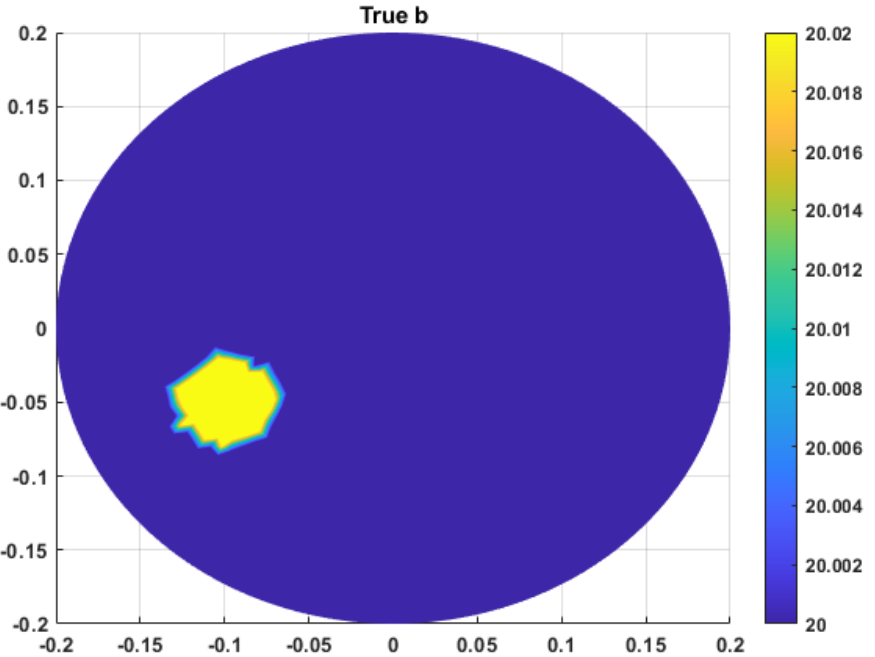} 
        \end{subfigure}
    \begin{subfigure}{0.3\textwidth}
        \includegraphics[width=1.0\linewidth]{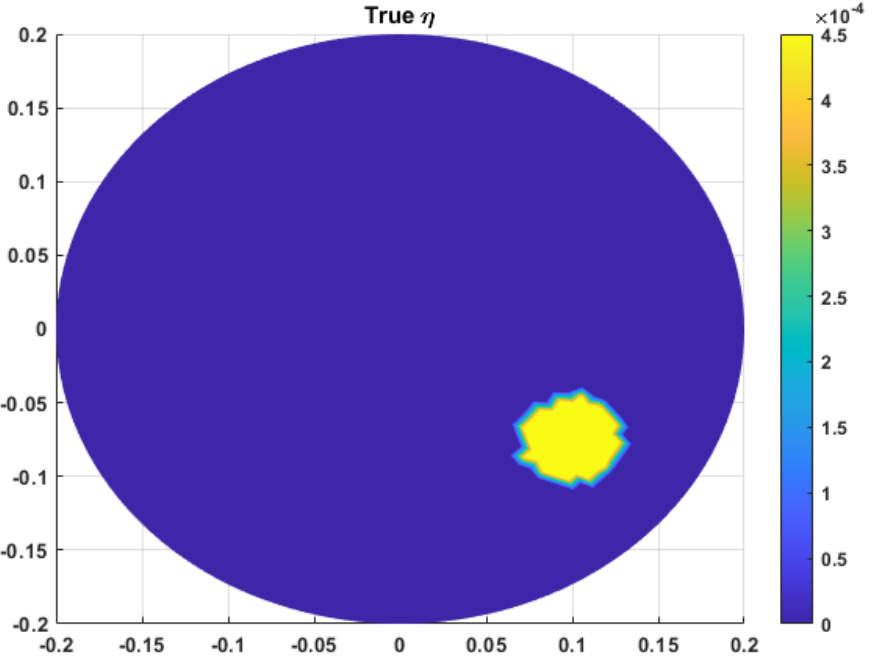} 
        \end{subfigure}
\bigskip 
    \begin{subfigure}{0.3\textwidth}
        \includegraphics[width=1.0\linewidth]{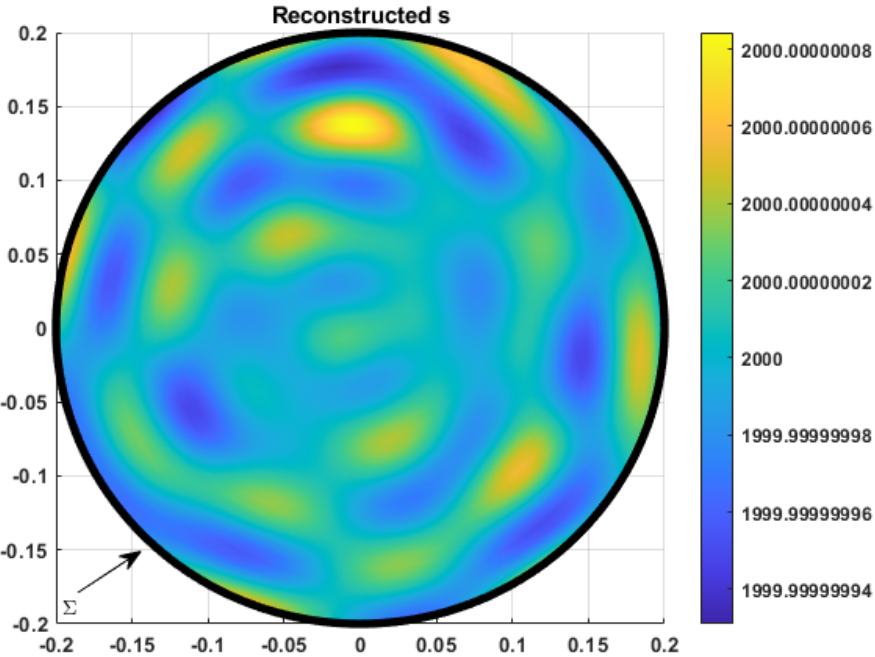} 
        \end{subfigure}
    \begin{subfigure}{0.3\textwidth}
        \includegraphics[width=1.0\linewidth]{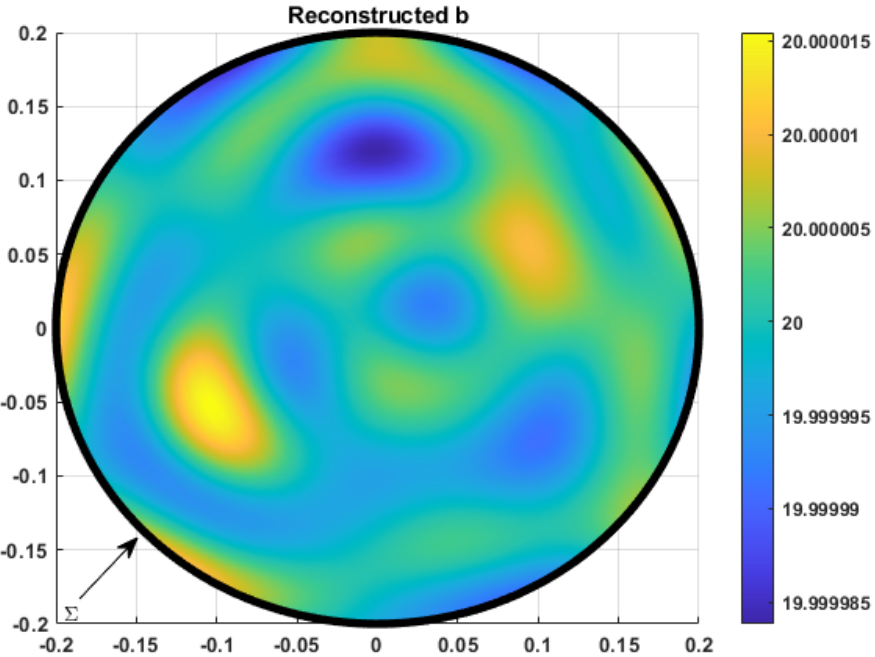} 
        \end{subfigure}
    \begin{subfigure}{0.3\textwidth}
        \includegraphics[width=1.0\linewidth]{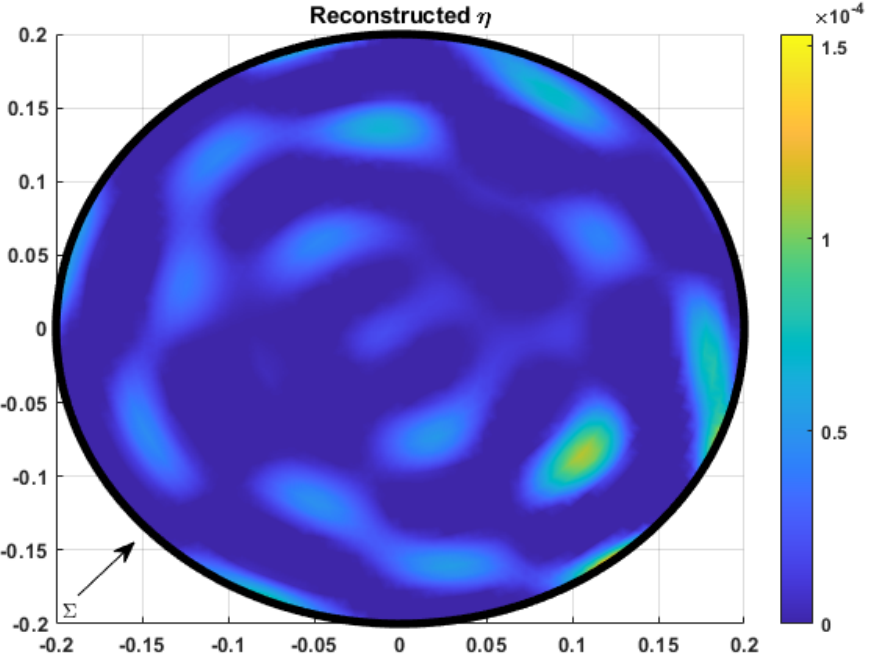} 
        \end{subfigure}
 \caption{First row: actual locations and intensities of three phantoms with separated supports in $(s,b,\eta)$; second row: reconstructions after 14 Newton iterations.}
 \label{fig:case:3}
\end{figure}
\section{Discussion}
\label{sec:conclusions}
For the formulation of the iterative reconstruction via a frozen Newton-type method, based on establishing range invariance of the all-at-once operator, it was necessary to work with a weaker parameter space than the originally assumed domain of the forward operator. The underlying reason is that the available analytical tools require the function $r$ to be boundedly invertible at $\vxz$. By our specific choice of $r$, we have $r^\prime(\vxz) = \text{id}$; see~\cite{Kaltenbacher23Convergence} (Theorem 4 in Appendix C).
In the present setting, this would entail the embedding $L^2(0,T;L^\infty(\Omega)) \hookrightarrow L^\infty(0,T;L^2(\Omega))$ which is not valid. Consequently, in order to close the required estimates, two alternatives arise: either increasing the regularity of the parameter space or strengthening the regularity of the state space. We here opted for the latter approach and imposed higher regularity on the state space. Future work will be devoted the former, by developing a Banach space framework for the iterative reconstruction method. In particular, in order to enhance reconstruction quality in tests with piecewise constant phantoms as those considered here, we intend to use total variation regularization. 
\section*{Acknowledgments}
This research was funded in part by the Austrian Science Fund (FWF) [10.55776/P36318].



%% file: linearised_uniqueness.tex
\def\pole{\mathfrak{p}}
\subsection{Linearized uniqueness for an all-at-once formulation in space-time}
We define the all-at-once forward operator  
\[
\mathbb{F}:= (\mathbb{F}_j^{mod},\mathbb{F}_j^{obs})_{j\in\{1,2,3\}}:(s,b,\eta,(u_j)_{j\in\{1,2,3\}})
\mapsto (\mathbb{F}_j^{mod}(s,b,\eta,u_j),\mathbb{F}_j^{obs}(u_j))_{j\in\{1,2,3\}}
\]
by 
\begin{equation}\label{eq:aao}
\begin{aligned}
&\mathbb{F}_j^{mod}(s,b,\eta,u_j)=\begin{cases}
    (b_j u_j -\eta u_j^2)_{tt} - s \Delta u_j - \Delta u_{j\,t} &\text{ in }(0,T_j) \times \Omega, \text{ time periodic}, \\ 
    \gamma u_j + \nabla u_j \cdot \vecn &\text{ on }(0,T_j)\times\partial\Omega,
\end{cases}\\
&\mathbb{F}_j^{obs}(u_j)= \text{tr}_\Sigma u_j,
\end{aligned}
\end{equation}
and linearize it at $(s^0,b^0,0,(u^0_j)_{j\in\{1,2,3\}})$ with positive constants $s^0$, $b^0$ and space-time separable reference states 
\begin{align}
\label{eq:reference:state:space:time}
u^0_j(x,t)=\phi(x)\,\psi_j(t),
\end{align}
where $\phi \in \mathcal{D}(-\Delta)$, $\phi \neq 0, \Delta \phi \neq 0$ almost everywhere in $\Omega$, and $\psi_j \in H^2(0,T_j)$ being $T_j$ periodic functions.
It is important to note that in such an all-at-once formulation, $u^0_j$ does not necessarily need to be a PDE solution corresponding to the coefficients $(s^0,b^0,0)$.
Note, that we set $\beta=0$ since we need the negative Laplacian on both levels of time differentiation (zeroth and first order) to share the same eigenfunctions.

Our aim is to show linearised uniqueness, that is, injectivity of the  linearised operator $\mathbb{F}^\prime(s^0,b^0,0,(u^0_j)_{j\in\{1,2,3\}})$, which requires us to conclude $(ds,db,d\eta,(du_j)_{j\in\{1,2,3\}}) \allowbreak =0$ from 
\begin{align}\label{Fprime0}
&\begin{cases}
    b^0(du_j)_{tt} - s^0 \Delta (du_j) - \Delta (du_j)_t \\
\qquad-(d\eta) \phi^2 \, (\psi_j^2)'' - (ds) \,\Delta \phi\,\psi_j  + (db) \phi\,\psi_j''=0  
\qquad \text{ in }(0,T_j) \times \Omega \\
\gamma (du)_j + \nabla (du)_j \cdot \vecn = 0 \,
\qquad\qquad\qquad\qquad\qquad\qquad\text{ on }(0,T_j)\times\partial\Omega
\end{cases} \\ \nonumber
&\text{ and } \text{tr}_\Sigma (du_j)=0, 
\end{align}
with time periodicity conditions.
Taking inner products with $e^{-\imath m\omega_j t}$ in time and with eigenfunctions $\varphi^{\ell,k}$, $\ell\in\mathbb{N}$, $k\in K^\ell$ of the impedance Laplacian, corresponding to eigenvalues $\lambda^\ell$, and abbreviating
\[
\begin{aligned}
&a_s^{\ell,k}:=\left( (ds) \,(-\Delta) \phi,\,\varphi^{\ell,k}\right)_{L^2(\Omega)},\quad
a_b^{\ell,k}:=\left( (db) \phi,\,\varphi^{\ell,k}\right)_{L^2(\Omega)},\quad \\
&a_\eta^{\ell,k}:=\left( -(d\eta) \,\phi^2,\,\varphi^{\ell,k}\right)_{L^2(\Omega)},\quad c^{\ell,k}_{j,m}:=\frac{2}{T_j}\int_0^{T_j}\left((du_j)(t),\,\varphi^{\ell,k}\right)_{L^2(\Omega)}\, e^{-\imath m\omega_j t}dt\\
\end{aligned}
\]
as well as
\[
\begin{aligned}
&\widehat{\psi}_{j,m}:=\frac{2}{T_j}\int_0^{T_j}\psi_j(t)\, e^{-\imath m\omega_jt}dt,
&&\widehat{\psi''}_{j,m}:=\frac{2}{T_j}\int_0^{T_j}\psi_j''(t)\, e^{-\imath m\omega_jt}dt, \\
&\widehat{{\psi^2}''}_{j,m}:=\frac{2}{T_j}\int_0^{T_j}(\psi_j^2)''(t)\, e^{-\imath m\omega_jt}dt,
&&\mathcal{A}_{j,s}(z):=\frac{2}{T_j}\int_0^{T_j}\psi_j(t)\, e^{-zt}dt, \\
&\mathcal{A}_{j,b}(z):=\frac{2}{T_j}\int_0^{T_j}\psi_j''(t)\, e^{-zt}dt,
&&\mathcal{A}_{j,\eta}(z):=\frac{2}{T_j}\int_0^{T_j}{\psi_j^2}''(t)\, e^{-zt}dt\
\end{aligned}
\]
we can conclude from \eqref{Fprime0}
\[
\begin{aligned}
&(-m^2\omega_j^2b^0 + s^0\lambda^\ell + \imath m\omega_j \lambda^\ell)\,c^{\ell,k}_{j,m}
+\widehat{\psi}_{j,m}\,a_s^{\ell,k}+\widehat{\psi''}_{j,m} a_b^{\ell,k}+\widehat{{\psi^2}''}_{j,m} a_\eta^{\ell,k} = 0 \\
&\hspace*{6.5cm} m,\,\ell\in\mathbb{N},\ k\in\{1,\ldots,K^\ell\},\ j\in\{1,2,3\}\\
&\sum_{\ell\in\mathbb{N}}\sum_{k=1}^{K^\ell} c^{\ell,k}_{j,m}\text{tr}_\Sigma\varphi^{\ell,k}=0\hspace*{3cm} m\in\mathbb{N},\ j\in\{1,2,3\},
\end{aligned}
\]
hence
\[
\begin{aligned}
&\sum_{\ell\in\mathbb{N}}\frac{1}{-m^2\omega_j^2 b^0 + s^0\lambda^\ell + \imath m\omega_j \lambda^\ell}\sum_{k=1}^{K^\ell} 
\varphi^{\ell,k}(x_0)
\sum_{\mathclap{q\in\{s,b,\eta\}}} \mathcal{A}_{j,q}(\imath m\omega_j)\, a_q^{\ell,k}
=0\\
&\hspace*{7.5cm} m\in\mathbb{N},\ j\in\{1,2,3\},\,x_0 \in \Sigma.
\end{aligned}
\]
By analytic continuation over $\frac{1}{z}$ from the infinite set $\{\frac{1}{\imath m\omega}\, : \,m\in\mathbb{N}\}$ accumulating at zero, to the open set 
\[
\mathbb{O}=\mathbb{C}\setminus 
\left\{\frac{1}{z}\,:\, z^2b^0+s^0\lambda^\ell+z \lambda^\ell=0\right\},
\]
we can extend the above statement to 
\begin{align}
\nonumber
&\sum_{\ell\in\mathbb{N}}\frac{1}{z^2b^0+s^0\lambda^\ell+z \lambda^\ell}\sum_{k=1}^{K^\ell} 
\varphi^{\ell,k}(x_0)
\sum_{\mathclap{q\in\{s,b,\eta\}}} \mathcal{A}_{j,q}(z)\, a_q^{\ell,k}
=0,
\,\frac{1}{z}\in\mathbb{O},\ j\in\{1,2,3\},\ x_0\in\Sigma.
\end{align}
An elementary computation shows that the poles are given by
$\pole_\ell=\frac{1}{2b^0}(-\lambda^\ell\pm\sqrt{(\lambda^\ell)^2-4b^0s^0\lambda^\ell})$ and 
in particular that they are single.
Moroever, for each $\ell \in \mathbb{N}$ we can choose the pole such that $\real{\pole_\ell} < 0$ by taking the negative branch.
Thus, for any $\ell'\in\mathbb{N}$, multiplying with $(z-\pole_{\ell'})$ and letting $z$ tend to $\pole_{\ell'}$ singles out the $\ell'$th term of the outermost sum and yields
\begin{equation}\label{Ajqaq0}
\begin{aligned}
&\sum_{k=1}^{K^{\ell'}} 
\varphi^{\ell',k}(x_0)
\sum_{\mathclap{q\in\{s,b,\eta\}}} \mathcal{A}_{j,q}(\pole_{\ell'})\, a_q^{\ell',k}
=0
\quad
\ell'\in\mathbb{N},\ j\in\{1,2,3\},\ x_0\in\Sigma.
\end{aligned}
\end{equation}
In the following we will skip the prime on $\ell$.\\
We now argue that for each $\ell\in\mathbb{N}$, the matrix 
\[
A^\ell:=(\mathcal{A}_{j,q}(\pole_{\ell}))_{j\in\{1,2,3\},\,q\in\{s,b,\eta\}}
\]
can be made nonsingular by a proper choice of $\psi_1$, $\psi_2$, $\psi_3$.
To this end, we set $\psi_3=2\psi_1$, so that 
\[
A^\ell= \left(\begin{array}{ccc}
A^\ell_{1,s}&A^\ell_{1,b}&A^\ell_{1,\eta}\\
A^\ell_{2,s}&A^\ell_{2,b}&A^\ell_{2,\eta}\\
2A^\ell_{1,s}&2A^\ell_{1,b}&4A^\ell_{1,\eta}\\
\end{array}\right), \quad
\text{det}(A^\ell)=2(A^\ell_{1,s}A^\ell_{2,b}-A^\ell_{2,s}A^\ell_{1,b})A^\ell_{1,\eta}
\]
Using integration by parts in time and the periodicity of $\psi_j$ we obtain 
\[
\begin{aligned}
A^\ell_{j,b} &= \frac{2}{T_j} \int_0^{T_j} \psi_j''(t) e^{-\pole_\ell t} dt = \frac{2}{T_j}\left( \left[ \psi_j'(t) e^{-\pole_\ell t} \right]_{t = 0}^{T_j} - \int_0^{T_j} \psi_j'(t) (-\pole_\ell) e^{-\pole_\ell t} dt \right)  \\
&= \frac{2}{T_j}\left( \psi_j'(0) (e^{-\pole_\ell T_j} - 1) + \pole_\ell (\psi_j(0) (e^{-\pole_\ell T_j} - 1) + \pole_\ell \int_0^{T_j} \psi_j(t)e^{-\pole_\ell t} dt \right) \\
&=\frac{2}{T_j}(e^{-\pole_\ell T_j} - 1)\left(\psi_j'(0) + \pole_\ell \psi_j(0) \right) + \pole_\ell^2 A^\ell_{j,s}.
\end{aligned}
\]
Hence, we conclude that
\[
\begin{aligned}
2(A^\ell_{1,s}A^\ell_{2,b}-A^\ell_{2,s}A^\ell_{1,b})A^\ell_{1,\eta}=4A^\ell_{1,\eta} \Bigl( &\frac{(e^{-\pole_\ell T_2} - 1)\left(\psi_2'(0) + \pole_\ell \psi_2(0) \right)}{T_2} A^\ell_{1,s} \\
&\qquad- \frac{(e^{-\pole_\ell T_1} - 1)\left(\psi_1'(0) + \pole_\ell \psi_1(0) \right)}{T_1} A^\ell_{2,s} \Bigr).
\end{aligned}
\]
Here, $e^{-\pole_\ell T_j}$ never attains one.
Taking $T_1 \neq T_2$ and requiring $\psi_1$ and $\psi_2$ to be chosen such that $\forall \ell \in \mathbb{N}$
\begin{align}
\label{eq:pole:condition:source}
\mathcal{A}_{1,s}(\pole_\ell) \neq \frac{T_2(e^{-\pole_\ell T_1} - 1)\left(\psi_1'(0) + \pole_\ell \psi_1(0) \right)}{T_1 (e^{-\pole_\ell T_2} - 1)\left(\psi_2'(0) + \pole_\ell \psi_2(0) \right)} \mathcal{A}_{2,s}(\pole_\ell), 
\end{align}
as well as
\begin{align}
\label{eq:pole:condition:source:2}
    \mathcal{A}_{1,\eta}(\pole_\ell) \neq 0 \quad \ell \in \mathbb{N},
\end{align}
see Remark~\ref{rem:example}, 
we arrive at
\begin{align}
\nonumber
    \text{det}(A^\ell) \neq 0.
\end{align}
As a consequence of nonsingularity of $A^\ell$, we can conclude from \eqref{Ajqaq0} that 
\[
\begin{aligned}
&\sum_{k=1}^{K^{\ell}} 
\varphi^{\ell,k}(x_0)\, a_q^{\ell,k}=0
\quad
\ell\in\mathbb{N},\ {q\in\{s,b,\eta\}},\ x_0\in\Sigma,
\end{aligned}
\]
that is, for the projections onto the eigenspaces $\mathbb{E}^\ell:=\text{span}\{\varphi^{\ell,k}\, : \, k\in\{1,\ldots,K^{\ell'}\} \}$, we obtain
\[
\text{tr}_\Sigma \text{Proj}_{\mathbb{E}^\ell} \tilde{q}=0, \qquad\tilde{q}\in\{-\Delta\phi (ds),\,-\Delta\phi (db),\,\phi^2(d\eta)\} 
\quad
\ell\in\mathbb{N}.
\]
In order to conclude that $a^{\ell,k}_q = 0$, we impose for any $\ell \in \mathbb{N}$ and ${q\in\{s,b,\eta\}}$ that
\begin{align}
\label{eq:lin:injectivity:structural:condition}
\left(\sum_{k=1}^{K^{\ell}} 
\varphi^{\ell,k}(x_0)\, a_q^{\ell,k}=0, \, x_0\in\Sigma\right) \Longrightarrow \left(a^{\ell,k}_q = 0, \, \forall k\in \{1, \ldots, K^\ell\}\right).
\end{align}
This structural assumption is satisfied by unique continuation under appropriate conditions on $\Sigma$, i.e., there exists a $x \in \partial \Omega$ and $r > 0$ such that $\partial \Omega \cap B_r(x) \subset \Sigma$, and by the fact that if $\varphi^{\ell,k} = 0$ on $\Sigma$, we also have $\nabla \varphi^{\ell,k} \cdot \vecn = 0$ on $\Sigma$ (due to the impedance boundary conditions),~\cite[cf. Theorem 4.1]{arendt2008} on the individual eigenspaces yields
\[
\text{Proj}_{\mathbb{E}^\ell} \tilde{q}=0, \qquad\tilde{q}\in\{-\Delta\phi (ds),\, \phi (db),\,-\phi^2(d\eta)\}
\quad
\ell\in\mathbb{N}.
\]
Therefore, by summing up over $\ell\in\mathbb{N}$ and using our assumption of $\phi$ and $\Delta\phi$ being nonzero almost everywhere in $\Omega$\footnote{e.g., $\phi=\phi^{n,k}$ some fixed $n$, $k$, which is also a choice that allows the construction of space-time separable approximate states}, we have $ds=0$, $db=0$, and $d\eta=0$, as desired. 
Thus, we have shown the following theorem.
\begin{theorem}
\label{thm:linearised:uniqueness}
Assume that the assumptions on the boundary and parameters in Theorem~\ref{thm:westervelt:nonlinear:solution:existence:uniqueness:source:boundary} hold, $\Sigma \subset \partial\Omega$ being either a discrete set of points satisfying the structural assumption in~\eqref{eq:lin:injectivity:structural:condition} or $\Sigma \subset \partial\Omega$ being an open subset in the boundary topology, and that the reference states take the form~\eqref{eq:reference:state:space:time}, where $\psi_j \in H^2(0,T_j)$ are periodic, $\phi \in \mathcal{D}(-\Delta), \phi \neq 0, \Delta\phi\neq 0$ a.e. in $\Omega$ and satisfy the conditions~\eqref{eq:pole:condition:source} and~\eqref{eq:pole:condition:source:2}. 
\newline
Then, $\mathbb{F}_j^\prime(s^0,b^0,0,u_j^0)(ds,db,d\eta,\du_j) = 0$ for $j \in \{1,2,3\}$ implies $(ds,db,d\eta,\allowbreak \du_1,\allowbreak \du_2, \du_3) = 0$.
\end{theorem}
\begin{remark} \label{rem:example}
As an example of a setting that satisfies the assumptions of Theorem~\ref{thm:linearised:uniqueness} 
consider $\psi_j(t) = \cos(\tfrac{2\pi}{T_j} t) + \mathfrak{a}$, $\mathfrak{a}\in (-\infty, -1) \cup(-1/4,\infty)$, for $j\in \{1,2\}$, with $T_1 \neq T_2 > 0$. Then, $\psi_1(0) = \psi_2(0) =1+\mathfrak{a}\neq 0$ 
$\psi_1'(0) = \psi_2'(0)=0$,
and $\mathcal{A}_{j,s}(z)$ reads
\[
\mathcal{A}_{j,s} (z) = \frac{2}{T_j}\int_{0}^{T_j} \psi_j(t) e^{-zt}\,dt = 
\begin{cases}
\frac{2(1-e^{-zT_j})(\mathfrak{a}(z^2+\omega_j^2) + z^2)}{z\,T_j(z^2 + \omega_j^2)}, & \quad z\in\mathbb{C}\setminus\{0\}\\
2\mathfrak{a},  &\quad z=0\\
\end{cases},
\]
where $\omega_j := \frac{2 \pi}{T_j}$. We chose the poles such that they have negative real part, hence it is sufficient to show that~\eqref{eq:pole:condition:source} holds for $\real{\pole_\ell} < 0$.  
\\
To this end, on the contrary assume that there exists $z \in \mathbb{C}$, with $\real{z} < 0$ such that
\[
\mathcal{A}_{1,s}(z) = \frac{T_2 (\psi_1'(0) + z\psi_1(0)) (e^{-z T_1}-1)}{T_1 (\psi_2'(0) + z\psi_2(0))(e^{-z T_2}-1)}\mathcal{A}_{2,s}(z),
\] 
which in our setting reads as
\begin{equation}\label{eq:A1A2}
    \frac{\mathcal{A}_{1,s}(z)}{\mathcal{A}_{2,s}(z)} = \frac{T_2 (e^{-z T_1}-1)}{T_1 (e^{-z T_2}-1)}.
\end{equation}
Here we made use of the fact that $\mathcal{A}_{2,s}(z)$ has no zeros and $\mathcal{A}_{1,s}(z)$ has no poles in the left half plane.
Using the explicit expression for $\mathcal{A}_{1,s}(z)$ and $\mathcal{A}_{2,s}(z)$ from above, we obtain
\begin{align}
    \nonumber
\frac{\mathcal{A}_{1,s}(z)}{\mathcal{A}_{2,s}(z)} = \frac{T_2 (e^{-z T_1}-1)(\mathfrak{a}(z^2 + \omega^2_1)(z^2+\omega_2^2) + z^2(z^2 + \omega_2^2))}{T_1 (e^{-z T_2}-1)(\mathfrak{a}(z^2 + \omega^2_1)(z^2+\omega_2^2) + z^2(z^2 + \omega_1^2))}.
\end{align}
With this, \eqref{eq:A1A2} is equivalent to
\begin{align}
    \nonumber
    \frac{(\mathfrak{a}(z^2 + \omega^2_1)(z^2+\omega_2^2) + z^2(z^2 + \omega_2^2))}{(\mathfrak{a}(z^2 + \omega^2_1)(z^2+\omega_2^2) + z^2(z^2 + \omega_1^2))}  = 1.
\end{align}
This boils down to $\omega_1=\omega_2$, which is a contradiction to $T_1 \neq T_2$. It remains to check whether $\mathcal{A}_{1,\eta}(z)\neq 0 $. However, $\mathcal{A}_{1,\eta}(z)$ reads
\begin{align}
    \nonumber
\mathcal{A}_{1,\eta}(z) &= \frac{2}{T_1} \int_{0}^{T_1} \left( \frac{\partial^2}{\partial t^2}(\cos(\tfrac{2\pi}{T_1} t) + \mathfrak{a})^2 \right) e^{-z t}  \,dt \\ \nonumber
&= \frac{16\pi^2}{T_1^3} z (1 - e^{-zT_1})\left(\frac{1}{z^2 + 4\omega_1^2} + \frac{\mathfrak{a}}{z^2 + \omega_1^2} \right) ,
\end{align}
which is non zero by the imposed domains on $z$ and $\mathfrak{a}$. We further note that, if $\psi \in H^2_{\text{per}}(0,T)$, it has a $C^1([0,T])$ representative and $\psi(0) = \psi(T)$, hence, there exists a $t_0 \in (0,T)$ such that $\psi_t(t_0) = 0$.
\end{remark}

%% file: IRGNM_L2L2.tex
\newcommand{\lts}{\tilde{s}}      
\newcommand{\ltb}{\tilde{b}}      
\newcommand{\lte}{\tilde{\eta}}   
\subsection{Iterative reconstruction of \texorpdfstring{$s$, $b$, and $\eta$}{s, b, and eta} from boundary measurements}
\label{sec:reconstruction:algo}
In what follows we consider the (exact) measurement data $h=F(s^\dag,b^\dag,\eta^\dag)$ on the boundary part $\Sigma$ under the excitation $g$ on 
$
\partial\Omega$ and our goal is to reconstruct $(s^\dag,b^\dag,\eta^\dag)$. We are facing two problems here. First, we do not have injectivity of the forward operator for arbitrary parameters $(s,b,\eta)$; second, measurement devices may be imperfect and we possibly will have to deal with noisy measurements $h^\delta$ and the lack of continuous invertibility of $F$, that is, ill-posedness of the inverse problem. 
We denote the noise level by $\delta>0$ such that 
\begin{equation}
\label{eq:noise:eq}
    \|F(s^\dag,b^\dag,\eta^\dag) - h^\delta\|_{L^2(0,T;L^2(\Sigma))} \leq \delta.
\end{equation}
A solution to the first problem is to consider an all-at-once formulation \eqref{eq:aao} using the aforementioned linearised uniqueness result. From physical measurements we obtain an observation vector $\vec{h}^\delta = (0, g_1, h_1^\delta, 0, g_2, h_2^\delta, 0, g_3, h_3^\delta)$ where $g_j$ are the corresponding boundary sources. With that we can conclude formal well-definedness of a frozen Newton method. For the second problem regularization needs to be applied. Convergence of the resulting iterative reconstruction scheme demands structural conditions on the forward operator. One of the conditions that allows for convergence is range invariance of the linearised all-at-once forward operator $\mathbb{F}^\prime$~\cite{Kaltenbacher23Convergence}. We achieve this by establishing effective increments $(\vec{ds}, \vec{db}, \vec{d\eta}, \vec{du})$, possibly time dependent, such that
\begin{equation}
    \label{eq:range:invariance}
    \vec{\mathbb{F}}(\vec{s},\vec{b},\vec{\eta},\vec{u}) - \vec{\mathbb{F}}(\vec{s^0},\vec{b^0},\vec{\eta^0},\vec{u^0}) = \vec{\mathbb{F}^\prime}(\vec{s^0},\vec{b^0},\vec{\eta^0},\vec{u^0})(\vec{ds}, \vec{db}, \vec{d\eta}, \vec{du}),
\end{equation}
where $\vec{\mathbb{F}} = (\mathbb{F}_1, \mathbb{F}_2, \mathbb{F}_3)$ and $\mathbb{F}_j^\prime(s^0_j,b^0_j,\eta^0_j,u^0_j)$ reads 
\[
\begin{aligned}
&\mathbb{F^\prime}_j^{mod}(s^0_j,b^0_j,\eta^0_j,u_j^0) (ds_j,db_j,d\eta_j,(du)_j)=\\ 
&\begin{cases}
(b^0_j(du)_j - 2\eta^0_j u_j^0 (du)_j - d\eta_j (u_j^0)^2 + db_j u^0_j)_{tt} - s_j^0 \Delta (du)_j -  \Delta (du)_{j\,t} - ds_j \Delta u_j^0  \\
\hspace*{6.5cm}\text{ in }(0,T_j)\times\Omega \text{ with time periodicity}\\ 
\gamma (du)_j+\nabla (du)_j\cdot\mathbf{n}  \text{ on } (0,T_j)\times \partial\Omega, \\
\end{cases}\\
&\mathbb{F^\prime}_j^{obs}((du)_j)= \text{tr}_\Sigma (du)_j,
\end{aligned}
\]
One readily checks that 
\begin{equation}\label{eq:defr}
\begin{aligned}
    &du_j := u_j- u^0_j, \\ 
    &db_j := b_j - b^0_j, \\
    &ds_j := ds_j (s_j,u_j) =  \frac{(s_j-s^0_j)\Delta u_j}{\Delta u^0_j},\\
    &d\eta_j := d\eta_j(\eta_j,b_j,u_j) = \frac{1}{(u^0_j)^2} \Bigl(  \eta^0_j(u_j-u^0_j)^2 + (\eta_j - \eta^0_j)u_j^2 - (b_j - b^0_j)(u_j - u^0_j)\Bigr)
\end{aligned}
\end{equation}
fulfills~\eqref{eq:range:invariance}, where we additionally assume that 
\begin{align}
\label{eq:reference:state:bounded:from:below}
\exists c_u > 0:  |u^0_j| \geq  c_u > 0,\, |\Delta u^0_j| \geq c_u > 0,\ \text{ a.e. in } (0,T_j) \times \Omega, \, j \in \{1,2,3\},
\end{align}
which is justified by the fact that $u^0_j$ does not need to be a solution to the underlying PDE. 
The time dependence of $u_j$ and $u^0_j$ leads to time dependent effective increments $d\tilde{\eta}_j$ and $d\tilde{s}_j$. We further notice that the effective increments are in $L^2(0,T_j;L^2(\Omega))$. 
Thus, we need to lift the parameters $s_j, b_j$, and $\eta_j$ to be time dependent, denoting their time-dependent versions by $\lts_j, \ltb_j$, and $\lte_j$. Note, that once we lift the parameters, we need as many copies as we have reference states. Since the increased dimensionality clearly counteracts uniqueness and as we aim for identifying only a single set of time-independent parameters $s(x)$, $b(x)$ and $\eta(x)$ we introduce the penalization operator
\begin{align}
    P_j: (\vec{\lts}, \vec{\ltb}, \vec{\lte}) \mapsto (\lts_j - \text{Proj}^{L^2(0,T_j)_w}_\text{const} \vec{\lts}, \ltb_j - \text{Proj}^{L^2(0,T_j)_w}_\text{const}\vec{\ltb}, \lte_j - \text{Proj}^{L^2(0,T_j)_w}_\text{const} \vec{\lte}),
\end{align}
where $\text{Proj}^{L^2(0,T_j)_w}_\text{const}$ is the averaging  $L^2_w$ projection on the space of time constant functions with a weight function $w: [0,T_j] \rightarrow \mathbb{R}^+$ with $w > 0$ and $w \in L^1(0,T_j)\cap L^2(0,T_j)$. In view of the linearised uniqueness result above we define
\begin{equation}\label{eq:defKr}
\begin{aligned}
    & K:= \vec{\mathbb{F}}^\prime(s^0,b^0,\eta^0,\vec{u}^{\,0}), \quad \vec{h^0} = \vec{h}^\delta - \vec{\mathbb{F}}(s^0,b^0,\eta^0, \vec{u}^{\,0}), \\ \nonumber
    & r: (\lts, \ltb, \lte,\vec{u}) \mapsto ({ds}_j(\lts_j, u_j), db_j, d\eta_j(\lte_j, \ltb_j, u_j), du_j)_{j \in \{1,2,3\}},
\end{aligned}
\end{equation}
where we set $\vec{x}^{\,0} := (s^0_j, b^0_j, \eta^0_j, u_j^{0})_{j \in \{1,2,3\}}$.
Hence, we can frame the inverse problem, whose solution we denote by $\vec{x}^{\,\dag}$, as a combination of an ill-posed linear problem and a well-posed nonlinear problem
\begin{align}
    \nonumber
    K\hat{r} & = \vec{h}^0,  \\ \nonumber r(\lts,\ltb, \lte, \vec{u}) & = \hat{r}, \\ \nonumber
    \vec{P}(\lts,\ltb,\lte) &= 0.
\end{align}
We make use of this problem structure by formulating a regularized frozen Newton type method where we expect $r(\vx) \approx \vx- \vxz$ in a sufficiently small neighborhood of $\vxd$ (see Lemma~\ref{lemma:r:cont:diff}) with iterates 
\begin{align}
\label{identification:iterates:regularized:formulation}
    \vec{x}_{n+1}^{\,\delta} \in \underset{ \vec{x} = (\lts,\ltb,\lte, \vec{u}) \in \mathcal{B}_\rho(\vec{x}^0)}{\text{argmin}} J_n(\vec{x};\, \vec{x}_n),
\end{align}
where we define
\begin{align}
\label{eq:cost:function}
    J_n(\vec{x};\, \vec{x}_n) :=\| K(\vec{x} - \vec{x}_n^{\,\delta}) + \vec{\mathbb{F}}(\vec{x}_n^{\,\delta}) - \vec{h}^\delta\|_Y^{p_1} + \alpha_n R(\vec{x}) + \|\vec{P}(\vec{x})\|_{Z}^2, 
\end{align}
with $p_1 \in [1, \infty)$, and some sufficiently small neighborhood $\mathcal{B}_\rho(\vec{x}^0)$ of a reference point $\vec{x}^0$ that should also contain the exact solution 
\[
\vec{x}^{\,\dag}\in\mathcal{B}_\rho(\vec{x}^0)\subseteq \mathcal{D}(\vec{\mathbb{F}})\subseteq X.
\]
Here we work in a function space setting 
\[
\begin{aligned}
\mathcal{D}(\vec{\mathbb{F}}) \subset X &= \prod_{j=1}^3  L^2(0,T_j;L^2(\Omega))^3 \times X_{u_j}, 
\end{aligned}
\]
with 
\[
X_{u_j} = H^1(0,T_j;H^{2 + \tau}(\Omega)) \cap \solspace_j
\hookrightarrow L^{\infty}(0,T_j;W^{2,\infty}(\Omega))
\]
where $\tau > \frac{d}{2}$ and $\solspace_j$ denotes the actual solution space (cf. \eqref{eq:defH})
corresponding to solutions with period $T_j$, which continuously embeds into $L^\infty(0,T_j;L^\infty(\Omega))$. 
Note that we aim to stay in a Hilbert space setting in order to avoid possible further nonlinearity and/or non-differentiability induced by duality mappings in Banach spaces.
In order to be able to estimate the nonlinear terms, we impose higher regularity on the state spaces to still allow Lebesgue spaces (more precisely their Hilbert space versions $L^2(0,T;L^2(\Omega))$) for the parameters, thus comprising the practically relevant case of discontinuities in the coefficients.
The data space is chosen as 
\[
Y = 
\begin{cases}
\displaystyle\prod_{j=1}^3 (X_{u_j}^*  \times L^2(0,T_j; L^2(\partial\Omega)) \times L^2(0,T_j;L^2(\Sigma)))\text{ if $\Sigma$ is a smooth manifold} \\
\displaystyle\prod_{j=1}^3 (X_{u_j}^* \times L^2(0,T_j, L^2(\partial\Omega)) \times L^2(0,T_j;l^2(\Sigma)))\text{ if $\Sigma$ is a discrete set},
\end{cases}
\]
and $R: X \rightarrow [0,\infty]$ a proper convex weakly lower semi-continuous functional. 

With $R(\vec{x}) = \|\vec{x} - \vec{x}^{\,0}\|_X^2$  and $p_1 = 2$ we can write~\eqref{identification:iterates:regularized:formulation} in terms of its necessary and sufficient (due to convexity) first order optimality conditions, $\tfrac{d}{dv} J_n(\vec{x} + v \vec{h})|_{v=0} = 0$,
\begin{align}
    \vec{x}_{n+1}^\delta = \vec{x}_n^\delta + (K^*K + \vec{P}^*\vec{P}+ \alpha_n \text{Id})^{-1} (K^*(\vec{h}^\delta - \vec{\mathbb{F}}(\vec{x}_n^\delta)) - \vec{P}^*\vec{P}\vec{x}_n^{\,\delta} + \alpha_n (\vec{x}^{\,0} - \vec{x}_n^{\,\delta})),
\end{align}
where $K^*$ denotes the Hilbert space adjoint of $K: X \rightarrow Y$. Taking into account~\eqref{eq:noise:eq} 
for data with noise level $\delta$ one has to stop the iteration (with stopping index $n^*(\delta)$) according to 
\begin{align}
\label{eq:stopping:index}
     n^*(\delta) \rightarrow \infty,\,\delta \sum_{k = 0}^{n^*(\delta)-1} c_\rho^k \alpha_{n^*(\delta) - k - 1}^{-1/2} \rightarrow 0 \quad\text{ as } \delta \rightarrow 0,
\end{align}
where $c_\rho \in (0,1)$ is like in Lemma~\ref{lemma:r:cont:diff} and with $\alpha_n = \alpha_0 q^n$ for some $q \in (0,1)$ this corresponds to the classical a priori choice $\alpha_{n^*(\delta)} \rightarrow 0$ and $\delta^2/\alpha_{n^*(\delta)} \rightarrow 0$ as $\delta \rightarrow 0$~\cite{Kaltenbacher23Convergence}.
\begin{lemma}
\label{lemma:r:cont:diff}
    Under condition~\eqref{eq:reference:state:bounded:from:below}, the map $r: X \rightarrow X$ is continuously Fr\'{e}chet differentiable, $r'(\vec{x}^{\,0})^{-1} \in L(X)$ and there exists $\rho > 0$ sufficiently small and a constant $c_\rho \in (0,1)$ such that for any $\vec{x}^{\,0}$ with 
    $ \|\vec{x}^{\,0}-\vec{x}^{\,\dag}\|_X<\rho$, we have the estimate
    \begin{equation}\label{eq:rid}
    \|r(\vec{x}) - (\vec{x} - \vec{x}^{\,0})\|_{X} \leq c_\rho \| \vec{x} - \vec{x}^{\,0}\|_{X} \quad \forall \vec{x} \in 
    B_\rho(\vec{x}^{\,0}).
    \end{equation}
    Moreover, $r$ is injective on $B_\rho(\vec{x}^{\,0})$.
\end{lemma}
\begin{proof}
Using the definition of $r$, cf. \eqref{eq:defr}, \eqref{eq:defKr}, as well as the assumption \eqref{eq:reference:state:bounded:from:below} we directly estimate 
\begin{equation}\label{eq:est_r}
\begin{aligned}
&\|r(\vec{x}) - (\vec{x} - \vec{x}^{\,0})\|_{X}^2
\underset{\text{$j$-th}}{=}\|\frac{1}{\Delta u^0_j}(s_j-s^0_j)(\Delta u_j-\Delta u^0_j)\|_{L^2(0,T_j;L^2(\Omega))}^2\\
&\quad+\|\frac{1}{(u^0_j)^2} (\eta^0_j(u_j-u^0_j)^2 + (\eta_j - \eta^0_j)(u_j^2-(u^0_j)^2) - (b_j - b^0_j)(u_j - u^0_j))\|_{L^2(0,T_j;L^2(\Omega))}^2\\
&\leq \frac{1}{c_u^2}\|s_j-s^0_j\|_{L^2(0,T_j;L^2(\Omega))}^2
\|\Delta u_j-\Delta u^0_j\|_{L^\infty(0,T_j;L^\infty(\Omega))}^2
\\&\quad+\frac{1}{c_u^4}\Bigl(\|\eta^0_j\|_{L^2(0,T_j;L^2(\Omega))}
\|u_j-u^0_j\|_{L^\infty(0,T_j;L^\infty(\Omega))}^2
\\&\quad\qquad+\|\eta_j - \eta^0_j\|_{L^2(0,T_j;L^2(\Omega))}
\|u_j+u^0_j\|_{L^\infty(0,T_j;L^\infty(\Omega))}
\|u_j-u^0_j\|_{L^\infty(0,T_j;L^\infty(\Omega))}
\\&\quad\qquad+\|b_j - b^0_j\|_{L^2(0,T_j;L^2(\Omega))}
\|u_j-u^0_j\|_{L^\infty(0,T_j;L^\infty(\Omega))}^2\Bigr)^2
\end{aligned}
\end{equation}
hence, using $X_{u_j} \hookrightarrow L^{\infty}(0,T_j;W^{2,\infty}(\Omega))$ with embedding constant $C_{u_j}$,
\[
\begin{aligned}
&\|r(\vec{x}) - (\vec{x} - \vec{x}^{\,0})\|_{X}
\leq  \left(\frac{C_{u_j}}{c_u} \rho 
+3\frac{C_{u_j}^2}{c_u^2} (\|\vec{x}^{\,0}\|_X+\rho)\rho\right)
\|\vec{x} - \vec{x}^{\,0}\|_{X}
\end{aligned}
\]
for $\vec{x} \in B_\rho(\vec{x}^{\,0})$.
By choosing $\rho>0$ small enough, we obtain \eqref{eq:rid} for some $c_\rho<1$.

Similarly, for any $h\in X$ we estimate
\begin{equation}\label{eq:est_rprime}
\begin{aligned}
&\|r'(\vec{x})h - h\|_{X}^2
\underset{\text{$j$-th}}{=}\|\frac{1}{\Delta u^0_j}(h_{s_j}(\Delta u_j-\Delta u^0_j)+(s_j-s^0_j)\Delta h_{u_j})\|_{L^2(0,T_j;L^2(\Omega))}^2\\
&\quad+\|\frac{1}{(u^0_j)^2} (
2\eta^0_j(u_j-u^0_j)h_{u_j} 
+2(\eta_j - \eta^0_j)h_{u_j} 
+h_{\eta_j}(u_j^2-(u^0_j)^2) 
\\ & \qquad \qquad
- (b_j - b^0_j)h_{u_j} 
- h_{b_j}(u_j - u^0_j))
\|_{L^2(0,T_j;L^2(\Omega))}^2\\
&\leq \frac{1}{c_u^2}\Bigl(
\|h_{s_j}\|_{L^2(0,T_j;L^2(\Omega))}
\|\Delta u_j-\Delta u^0_j\|_{L^\infty(0,T_j;L^\infty(\Omega))} 
\\ &\qquad\qquad
+\|s_j-s^0_j\|_{L^2(0,T_j;L^2(\Omega))}
\|\Delta h_{u_j}\|_{L^\infty(0,T_j;L^\infty(\Omega))}
\Bigr)^2
\\&\quad+\frac{1}{c_u^4}\Bigl(2\|\eta^0_j\|_{L^2(0,T_j;L^2(\Omega))}
\|u_j-u^0_j\|_{L^\infty(0,T_j;L^\infty(\Omega))}
\|h_{u_j}\|_{L^\infty(0,T_j;L^\infty(\Omega))}
\\&\qquad\qquad
+2\|\eta_j - \eta^0_j\|_{L^2(0,T_j;L^2(\Omega))}
\|h_{u_j}\|_{L^\infty(0,T_j;L^\infty(\Omega))}
\\ &\qquad\qquad
+\|h_{\eta_j}\|_{L^2(0,T_j;L^2(\Omega))}
\|u_j^2-(u^0_j)^2\|_{L^\infty(0,T_j;L^\infty(\Omega))}
\\&\qquad\qquad
+\|b_j - b^0_j\|_{L^2(0,T_j;L^2(\Omega))}
\|h_{u_j}\|_{L^\infty(0,T_j;L^\infty(\Omega))}^2
\\& \qquad\qquad
+\|h_{b_j}\|_{L^2(0,T_j;L^2(\Omega))}
\|u_j-u^0_j\|_{L^\infty(0,T_j;L^\infty(\Omega))}^2\Bigr)^2
\end{aligned}
\end{equation}
and conclude that 
\[
\|r'(\vec{x}) - \text{id}\|_{L(X)}\leq c_\rho
\]
with $c_\rho<1$, provided $\vec{x} \in B_\rho(\vec{x}^{\,0})$ with $\rho>0$ small enough.
This implies that $r'(\vec{x})$ is an isomorphism.
To obtain (Lipschitz) continuity of  $r'$, we consider
\[
\begin{aligned}
\|r'(\vec{x})h - r'(\tilde{\vec{x}})h\|_{X}^2
\underset{\text{$j$-th}}{=}&\|\frac{1}{\Delta u^0_j}(h_{s_j}(\Delta u_j-\Delta \tilde{u}_j)+(s_j-\tilde{s}_j)\Delta h_{u_j})\|_{L^2(0,T_j;L^2(\Omega))}^2\\
&+\|\frac{1}{(u^0_j)^2} (
2\eta^0_j(u_j-\tilde{u}_j)h_{u_j} 
+2(\eta_j - \tilde{\eta}_j)h_{u_j} 
+h_{\eta_j}(u_j^2-\tilde{u}_j)^2) 
\\ & \quad\quad
- (b_j - \tilde{b}_j)h_{u_j}
- h_{b_j}(u_j - \tilde{u}_j))
\|_{L^2(0,T_j;L^2(\Omega))}^2
\end{aligned}
\]
and estimate it analogously to \eqref{eq:est_r}, \eqref{eq:est_rprime}, which yields
\[
\|r'(\vec{x}) - r'(\tilde{\vec{x}})\|_{L(X)}\leq C_\rho
\|\vec{x} - \tilde{\vec{x}}\|_{X}
\]
with some constant $C_\rho$ that only depends on $c_u$, $\rho$, $C_{u_j}$, and $\|\vec{x}^{\,0}\|_X$.

Local injectivity of $r$ follows from the Inverse Function Theorem.
\end{proof}
We are now able to state our final convergence result for the iterates defined in~\eqref{identification:iterates:regularized:formulation}.

\begin{theorem}
    Let the conditions of Theorem~\ref{thm:linearised:uniqueness} on the observation set $\Sigma$ and the reference states $u^0_j$, $j \in \{1,2,3\}$ be satisfied and assume that the latter additionally satisfy ~\eqref{eq:reference:state:bounded:from:below}.  Let $\vxz \in B_\rho(\vxd)$, $\rho > 0$ sufficiently small and let the stopping index $n^*(\delta)$ be chosen as in~\eqref{eq:stopping:index}. Then the iterates $(\vec{x}_n^\delta)_{1\leq n \leq n^*(\delta)}$ in~\eqref{identification:iterates:regularized:formulation} are well defined, stay in $B_\rho(\vxz)$, and converge to $\vxd$ in $X$, that is, $\|\vec{x}_{n^*(\delta)}^\delta - \vec{x}^{\,\dag}\|_{X} \rightarrow 0$ as $\delta \rightarrow 0$; in the noiseless case ($\delta=0$, $n^*(\delta) = \infty$) we have $\| \vec{x}_n - \vec{x}^{\,\dag}\|_X \rightarrow 0$ as $n \rightarrow \infty$.
\end{theorem}
\begin{proof}
    Lemma~\ref{lemma:r:cont:diff} already provides two key ingredients to guarantee convergence of the iterates defined in~\eqref{identification:iterates:regularized:formulation} in terms of~\cite{Kaltenbacher23Convergence} (see Theorem 2.2, injectivity of $r$ and  $r^\prime(\vxz)$ being boundedly invertible). The crucial last ingredient to conclude convergence follows from the estimate
    \begin{align}
    \nonumber
        \| r(\vx) - r(\vxd) - (\vx - \vxd)\|_X \leq & \| r(\vx) - r(\vxd) - r^\prime(\vxd) (\vx - \vxd)\|_X  \\ \nonumber
        &\qquad\qquad\qquad + \|r^\prime(\vx) - r^\prime(\vxz)\|_{L(X)}\|\vx - \vxd\|_X,
    \end{align}
    where we have used Lemma~\ref{lemma:r:cont:diff} to obtain existence of $\theta \in (0,1)$ such that
    \begin{align}
    \nonumber
        \| r(\vx) - r(\vxd) - (x - \vxd)\|_X \leq \theta \|\vx - \vxd\|_X,
    \end{align}
    on a sufficiently small neighborhood of $\vxd$ containing $\vxz$.
\end{proof}

\def\nsp{\textrm{nsp}}
\def\Xcheck{\check{X}}
\def\Proj{\textrm{Proj}}
\arxivshow{
\begin{remark}
Referring to the analysis in \cite{Kaltenbacher23Convergence}, we point to the fact that we can always satisfy the crucial convergence condition 
\begin{equation}\label{convcond}
\nsp(K)^\bot\subseteq\nsp(P)
\end{equation}
from there by a proper choice of the projection operator $P$ as long as linearized uniqueness 
\begin{equation}\label{lin_uniqueness}
\nsp(K)\cap\Xcheck=\{0\}, 
\end{equation}
holds, cf., Theorem~\ref{thm:linearised:uniqueness}; in the present setting, $\Xcheck$ is the space of time constant coefficient functions with 
$s_1=s_2=s_3$, 
\arxivskip{$b_1=b_2=b_3$,} 
$\eta_1=\eta_2=\eta_3$.
Due to \eqref{lin_uniqueness}, with $\vx_0$, $\vx^\dagger\in \Xcheck$, by re-defining 
\begin{equation}\label{P}
P=\Proj_{\Xcheck^\bot}\Proj_{\nsp(K)},
\end{equation}
we retain $\nsp(K)\cap\nsp(P)=\{0\}$ since $v\in\nsp(K)$ and $0=\Proj_{\Xcheck^\bot}\Proj_{\nsp(K)}v=\Proj_{\Xcheck^\bot}v$ implies $v\in\nsp(K)\cap\Xcheck=\{0\}$, cf. \eqref{lin_uniqueness}, so that $\vx_0-\vx^\dagger\in(\nsp(K)\cap\nsp(P))^\bot$ trivially holds.
With \eqref{P}, the condition \eqref{convcond} in \cite[Lemma 2.5]{Kaltenbacher23Convergence} 
is obviously satisfied and from $\vx^\dagger\in \Xcheck$, again using \eqref{lin_uniqueness} we conclude $\Proj_{\nsp(K)} \vx^\dagger=0$, hence $P \vx^\dagger=0$. The same holds for $\vx_0$. 
\end{remark}
\begin{remark}
Note that $b$ actually does not need to be lifted, so we could simply replace $\ltb_j$ by $b$ above and set $db=b-b^0$.
\end{remark}
}

%% file: supplemental_material.tex
\section{Proof of Theorem~\ref{thm:two:sources:westervelt:linear:solution:existence:uniqueness:mixed:boundary:space:dependent}}
\label{appendix:proof:theorem:linear:well-posedness}

%
First, we derive a weak formulation 
of~\eqref{eq:two:sources:westervelt:linear:periodic:space:dependent}.
To this end, we augment the linearized Westervelt equation by $\nu \alpha u_t$, where we impose $\frac{\gamma}{\beta} > \nu > 0$ a.e. as assumed, which is justified by the boundedness from below of $\gamma$ and $\frac{1}{\beta}$ on $\partial \Omega$, which then reads
\begin{equation}
    \label{eq:two:sources:linear:westervelt:augmented:space:dependent}
    \alpha \left(u_{tt} + \nu u_t \right) -c^2 \Delta u - \mathfrak{b} \Delta u_t - \left(\nu\alpha - \mu \right) u_t + \delta u = f. 
\end{equation}
Integrating from $0$ to $T$ w.r.t. $t$ yields the compatibility condition
\begin{equation}
    \label{proof:linear:well:posedness:compatibility:condition}
    \int_0^T \g dt = \int_0^T \beta u_t + \gamma u + \nabla u \cdot \vecn \,dt = \int_0^T \gamma u + \nabla u \cdot \vecn \,dt,
\end{equation}
where we used the periodicity of $u_t$. Differentiating the boundary conditions w.r.t. $t$ yields
\begin{equation}
    \nonumber
    \beta u_{tt} + \gamma u_t + \nabla u_t \cdot \vecn =\g_t.
\end{equation}
Testing the first term of~\eqref{eq:two:sources:linear:westervelt:augmented:space:dependent} by $-\frac{1}{\alpha} \Delta v$ gives
\[
\begin{aligned}
    &\int_0^T \int_\Omega \left(\alpha u_{tt} + \nu\alpha u_t \right) (-\tfrac{1}{\alpha} \Delta v)\,dx\,dt = \\ \nonumber 
    &- \Bigl(\int_0^T \int_{\partial\Omega} (u_{tt} + \nu u_t) \nabla v \cdot \vecn\, dS(x) + \int_\Omega \nabla u_t \cdot \nabla v_t \,dx - \int_\Omega \nu\nabla u_t \cdot \nabla v\,dx\,dt  \\ \nonumber
    &\qquad -\left[\int_\Omega \nabla u_t \cdot \nabla v\,dx \right]_0^T \Bigr) \\ \nonumber
    & = \int_0^T \int_{\partial\Omega} \frac{1}{\beta}(\nabla u_t \cdot \vecn) (\nabla v \cdot \vecn)\, dS(x) - \int_\Omega \nabla u_t \cdot \nabla v_t \,dx + \int_\Omega \nu\nabla u_t \cdot \nabla v\,dx \\ \nonumber
    & +\int_{\partial\Omega} \left(\frac{\gamma}{\beta} - \nu\right) u_t (\nabla v \cdot \vecn)\, dS(x) - \int_{\partial\Omega} \frac{1}{\beta} \g_t (\nabla v \cdot \vecn)dS(x)\,dt + \left[\int_\Omega \nabla u_t \cdot \nabla v\,dx \right]_0^T.
\end{aligned}
\]
By this, we obtain the following variational form
\[
\begin{aligned}
\label{eq:two:sources:linear:westervelt:weak:form:boundary:source:space:dependent}
    &u \in U,\, u(0) = u(T),\, u_t(0) = u_t(T),\,h(0) = h(T),\\ 
    &\forall v \in H^1(0,T;H^2(\Omega)) \cap L^2(0,T;H^2(\Omega)),\,\zeta \in L^2(0,T;L^2(\partial \Omega))\\
    &\int_0^T \int_\Omega - \nabla u_t \cdot \nabla v_t + \nu\nabla u_t \cdot \nabla v + \frac{1}{\alpha}\left(c^2 \Delta u + \mathfrak{b} \Delta u_t + \left(\nu\alpha - \mu\right)u_t - \delta u\right) \Delta v \,dx\,dt  \\ \nonumber 
     & + \int_0^T \int_{\partial \Omega} \frac{1}{\beta} \left(\nabla u_t \cdot \vecn \right) \left(\nabla v \cdot \vecn\right) \,dS(x) + \int_{\partial \Omega} \left[h_t - \left(\beta u_t + \gamma u + \nabla u \cdot \vecn - \g \right) \right]  \zeta\,dS(x)\,dt \\ \nonumber
     & + \int_0^T\int_{\partial\Omega} \left(\frac{\gamma}{\beta} - \nu \right) u_t (\nabla v \cdot \vecn) \,dS(x)\,dt + \left[\int_\Omega \nabla u_t \cdot \nabla v \,dx\right]_0^T \\ \nonumber 
     &= \int_0^T \int_{\partial \Omega} \frac{1}{\beta} \g_t (\nabla v \cdot \vecn) \,dS(x)\,dt - \int_0^T \int_\Omega \frac{1}{\alpha}f \Delta v\, dx\, dt,
\end{aligned}
\]
where we introduced the auxiliary function $h$ to conclude the boundary condition $\beta u_t + \gamma u + \nabla u \cdot \vecn = g$ from its time differentiated version. Indeed, reversing the integration by parts step yields
\[
\begin{aligned}
    &u \in U,\, u(0) = u(T),\, u_t(0) = u_t(T), \int_0^T \gamma u + \nabla u \cdot \vecn \, dt = \int_0^T g \, dt \\
    &\text{and for all }v \in H^1(0,T;H^2(\Omega)) \cap L^2(0,T;H^2(\Omega)): \\
    & \int_0^T\int_\Omega  \left( \alpha\utt - c^2 \Delta u - \mathfrak{b} \Delta u_t + \mu u_t + \delta u - f \right) (-\tfrac{1}{\alpha} \Delta v) \,dx  \\ \nonumber
    & \qquad + \int_{\partial \Omega} \left(\beta \utt + \gamma \ut + \nabla u \cdot \vecn - \g_t\right)(\nabla v \cdot \vecn)\,dS(x)\,dt = 0.
\end{aligned}
\]

Second, we employ a Galerkin method. We consider the eigenvalue problem of the impedance Laplacian $-\Delta_\gamma$ given by
\begin{equation}
\label{eq:two:sources:laplace:eigenvalue:problem:space:dependent}
\begin{cases}
-\Delta \phi_k = \lambda_k \phi_k & \text{in } \Omega,\\
\gamma \phi_k + \nabla \phi_k \cdot \vecn = 0 & \text{on }  \partial \Omega.
\end{cases}   
\end{equation}
The eigenfunctions $\phi_k$ are orthonormal in $L^2(\Omega)$ and orthogonal in $H^1(\Omega)$. The eigenvalues of $-\Delta_\gamma$ fulfill $0 < \lambda_1 \leq \lambda_2 \leq \ldots \leq \lambda_k \rightarrow \infty$, $n \rightarrow \infty$. We further define $\eta_k := \text{tr}(\phi_k)$, which are dense in $L^2(\partial \Omega)$  by the trace theorem, and with $V_n := \text{lin}\left\{\phi_1, \ldots, \phi_n \right\} \times \text{lin}\{\eta_1, \ldots, \eta_n\}$, $\bigcup_{n\in \mathbb{N}} V_n$ is dense in $H^1(\Omega) \times L^2(\partial\Omega)$. Plugging the ansatz $(u_n(t,\cdot), h_n(t,\cdot)) := (\sum_{k=1}^n \mathfrak{a}_k(t)\phi_k, \sum_{k=1}^n \mathfrak{c}_k(t)\eta_k)$ into~\eqref{eq:two:sources:linear:westervelt:weak:form:boundary:source:space:dependent}, testing with $v = \phi_k$ and defining $\vec{z}_n(t) = (\mathfrak{a}_1(t), \ldots, \mathfrak{a}_n(t), \dot{\mathfrak{a}}_1(t), \ldots, \dot{\mathfrak{a}}_n(t),
\mathfrak{c}_1(t), \ldots, \mathfrak{c}_n(t))^T$, with periodicity conditions $\vec{z}_n(0)=\vec{z}_n(T)$, yields the following ODE system
\begin{equation}
    \label{eq:two:sources:westervelt:linear:boundary:galerkin:form:full:ODE:system:a:space:dependent}
\begin{aligned}
    &   
\begin{bmatrix} 
I_{n \times n} & 0_{n \times n} & 0_{n \times n} \\ 
0_{n \times n} & \mathbf{H} & 0_{n \times n}    \\
0_{n \times n} & 0_{n \times n}  & \mathbf{D}
\end{bmatrix}
\dot{\vec{z}}_n(t) =  
\begin{bmatrix} 
0_{n \times n} & I_{n \times n} & 0_{n \times n}\\
-\mathbf{C}    & -\mathbf{G} & 0_{n \times n}  \\
\gamma \mathbf{D} + \mathbf{S}& \beta\mathbf{D} & 0_{n \times n}
\end{bmatrix} 
\vec{z}_n(t)
+ \begin{bmatrix}
        0 \\ \mathbf{F} \\ -\mathbf{K}
    \end{bmatrix}
\end{aligned}
\end{equation}
Here
\begin{align*}
\mathbf{F}_j =& \left(\tfrac{1}{\beta}\g_t, \nabla \phi_j \cdot \vecn \right)_{L^2(\partial \Omega)} -\left(\tfrac{1}{\alpha} f,\Delta\phi_j\right)_{L^2(\Omega)}, \\
\mathbf{K}_j =& \left(\g, \eta_j \right)_{L^2(\partial \Omega)}, \\
\mathbf{H}_{i,j} =& \left(\nabla \phi_i, \nabla \phi_j \right)_{L^2(\Omega)}, \\
\mathbf{C}_{i,j} =& \left(\tfrac{1}{\alpha}c^2 \Delta \phi_i, \Delta \phi_j \right)_{L^2(\Omega)} - \left(\tfrac{\delta}{\alpha} \phi_i, \Delta \phi_j \right)_{L^2(\Omega)}, \\
\mathbf{G}_{i,j} =& 
\left(\nu \nabla \phi_i, \nabla \phi_j \right)_{L^2(\Omega)} 
+\left(\tfrac{1}{\alpha}(\mathfrak{b} \Delta \phi_i+(\nu\alpha-\mu)\phi_i), \Delta \phi_j \right)_{L^2(\Omega)}
\\&
+ \left( \tfrac{1}{\beta} \trace{\Omega}{(\nabla \phi_i\cdot\vecn)} + (\tfrac{\gamma}{\beta} - \nu) \phi_i, \trace{\Omega}{(\nabla \phi_j\cdot\vecn) } \right)_{L^2(\partial\Omega)},
\\
\mathbf{D}_{i,j} =& \left( \eta_i, \eta_j \right)_{L^2(\partial\Omega)}, \quad \mathbf{S}_{i,j} = (\nabla \phi_i \cdot \vecn, \eta_j)_{L^2(\partial \Omega)}. 
\end{align*}
The matrices $\mathbf{H}$, $\mathbf{D}$ are positive definite, hence regular and 
\begin{equation}
    \label{eq:two:sources:westervelt:linear:boundary:galerkin:form:full:ODE:system:b:space:dependent}
\begin{aligned}
    &   
\dot{\vec{z}}_n(t) =  
\begin{bmatrix} 
I_{n \times n} & 0_{n \times n} & 0_{n \times n} \\ 
0_{n \times n} & \mathbf{H} & 0_{n \times n}    \\
0_{n \times n} & 0_{n \times n}  & \mathbf{D}
\end{bmatrix}^{-1}
\left(
\begin{bmatrix} 
0_{n \times n} & I_{n \times n} & 0_{n \times n}\\
-\mathbf{C}    & -\mathbf{G} & 0_{n \times n}  \\
\gamma \mathbf{D} + \mathbf{S}& \beta\mathbf{D} & 0_{n \times n}
\end{bmatrix} 
\vec{z}_n(t)
+ \begin{bmatrix}
        0 \\ \mathbf{F} \\ \mathbf{-K}
    \end{bmatrix}
\right)\\
    &= \Tilde{\mathbf{A}}(t) \vec{z}_n(t) + \Tilde{\mathbf{F}}(t),
\end{aligned}
\end{equation}
where $\Tilde{\mathbf{A}} \in L^2(0,T)$ and $\Tilde{\mathbf{F}} \in L^2(0,T)$. Considering the energy estimates below, the only solution to the homogeneous equation fulfilling the boundary conditions $\vec{z}_n(0) = \vec{z}_n(T)$ is the zero solution. Hence, Carathéodory's Theorem~\cite[Chapter 3, page 121]{Walter1998} yields that there exists a unique solution to~\eqref{eq:two:sources:westervelt:linear:boundary:galerkin:form:full:ODE:system:b:space:dependent} with $\vec{z}_n(0) = \vec{z}_n(T)$ such that $(u_n, h_n) \in H^2(0,T;V_n^u) \times H^1(0,T;V_n^h)$.

Third, energy estimates. We test the spatial discretized version of~\eqref{eq:two:sources:linear:westervelt:weak:form:boundary:source:space:dependent} with $v =  \tfrac{\nu}{2} \un + \unt $, $\zeta = h_{n_t} - (\beta\unt + \gamma \un + \nabla \un \cdot \vecn - \g)$ and with the identities
\begin{align*}
    \frac{1}{2} \frac{d}{dt} \left( \nabla \unt , \nabla \unt \right)_{L^2(\Omega)} &= \left(\nabla \untt, \nabla \unt\right)_{L^2(\Omega)} \\
    \frac{1}{2} \frac{d}{dt} \| \nabla \un \|_{L^2(\Omega)}^2 &= \left(\nabla \un, \nabla\unt \right)_{L^2(\Omega)} \\
    \frac{1}{2} \frac{d}{dt} \| \sqrt{\tfrac{c^2}{\alpha}} \Delta \un \|_{L^2(\Omega)}^2 &=  \left(\left(\tfrac{c^2}{2\alpha}\right)_t \Delta \un, \Delta \un \right)_{L^2(\Omega)} + \left(\tfrac{c^2}{\alpha} \Delta \unt, \Delta \un \right)_{L^2(\Omega)} \\
    \frac{1}{2} \frac{d}{dt} \| \sqrt{\tfrac{\mathfrak{b}}{\alpha}} \Delta \un \|_{L^2(\Omega)}^2 &= \left(\tfrac{\mathfrak{b}}{\alpha} \Delta \unt, \Delta \un \right)_{L^2(\Omega)}  + \left(\left(\tfrac{\mathfrak{b}}{2\alpha}\right)_t \Delta \un, \Delta \un\right)_{L^2(\Omega)} \\
    \frac{1}{2} \frac{d}{dt} \| \nabla \un \cdot \vecn \|_{L^2(\partial\Omega)}^2 &= \left( \nabla \unt \cdot \vecn,  \nabla \un \cdot \vecn \right)_{L^2(\partial \Omega)} \\
    \frac{d}{dt} \left( \g, \nabla \un \cdot \vecn \right)_{L^2(\partial \Omega)} &= \left( \g_t, \nabla \un \cdot \vecn \right)_{L^2(\partial \Omega)} + \left(\g, \nabla \unt \cdot \vecn \right)_{L^2(\partial \Omega)}
\end{align*}
one obtains due to the periodicity of $\vec{z}_n$ (that is that of $\un$ and $\unt$) and the fact that $-\gamma \unt =\nabla \unt \cdot \vecn$ on $\partial \Omega$ (due to our ansatz setting)
\begin{align}
    & \text{lhs} := \tfrac{\nu}{2} \|\nabla \unt \|_{L^2(0,T;L^2(\Omega))}^2 + \tfrac{\nu}{2} \| \sqrt{\tfrac{c^2}{\alpha}} \Delta \un \|_{L^2(0,T;L^2(\Omega))}^2 + \| \sqrt{\tfrac{\mathfrak{b}}{\alpha}} \Delta \unt \|_{L^2(0,T;L^2(\Omega))} \\ \nonumber
    &+ \| \sqrt{\nu\gamma} \unt \|_{L^2(0,T;L^2(\partial \Omega))}^2 + \| h_{n\,t} - (\beta \unt + \gamma \un + \nabla \un \cdot \vecn - \g) \|_{L^2(0,T;L^2(\partial \Omega))}^2  = \\ \nonumber
    &  \int_0^T \int_{\partial \Omega} \frac{\g_t}{\beta} \, \nabla\left(\tfrac{\nu}{2} \un + \unt \right) \cdot \vecn \,dS(x)\,dt + \int_0^T\int_\Omega \left(\tfrac{1}{2\alpha} \right)_t\left(c^2 + \tfrac{\nu}{2} \mathfrak{b} \right) (\Delta \un)^2\,dx\,dt \\ \nonumber
    &- \int_0^T \int_\Omega \left(\tfrac{f}{\alpha} - \tfrac{\delta}{\alpha} u_n + (\nu  - \tfrac{\mu}{\alpha}) \unt\right) \left(\Delta \unt + \tfrac{\nu}{2} \Delta \un \right)\, dx\, dt =: \text{rhs}.
\end{align}
The right hand side can be further simplified using the above identities arriving at
\begin{align}
    \text{rhs} &=  \int_0^T \int_{\partial\Omega}  \tfrac{1}{\beta}\left(\g_t - \tfrac{\nu}{2} \g \right) \nabla \unt \cdot \vecn \, dS(x)\,dt + \int_0^T \int_\Omega \left(\tfrac{1}{2\alpha} \right)_t \left(c^2 + \tfrac{\nu}{2}b \right) (\Delta \un)^2 \,dx\,dt \\ \nonumber
    &- \int_0^T \int_\Omega \tfrac{1}{\alpha}\left( f - \delta \un + (\nu\alpha  - \mu) \unt\right) \left(\Delta \unt + \tfrac{\nu}{2} \Delta \un \right)\, dx\, dt.
\end{align}
Using H\"older's inequality, including the case $q\in[1,\infty)$ with $\tfrac{1}{2q} + \tfrac{q-1}{2q} = \tfrac{1}{2}$, Young's inequality and considering our ansatz space, there exists a $C_a > 0$ such that we arrive at
\begin{align}
    |\text{rhs}| & \leq \|\tfrac{1}{\beta}\|_{L^\infty(\partial\Omega)}\left( (1+\tfrac{\nu^2}{4}) \varepsilon_1  \|\g\|_{H^1(0,T;L^2(\partial \Omega))}^2 + \ \tfrac{1}{2 \varepsilon_1}  \|\nabla\unt \cdot \vecn \|_{L^2(0,T;L^2(\partial \Omega))}^2 \right) \\ \nonumber
    &+ (\| \left(\tfrac{1}{2\alpha} \right)_t\left(c^2 + \tfrac{\nu}{2} \mathfrak{b} \right)\|_{L^\infty(0,T;L^\infty(\Omega))}^2 + \tfrac{C_a}{2\varepsilon_4}\|\tfrac{\delta}{\alpha}\|_{L^2(0,T;L^2(\Omega))}^2)\|\un\|_{L^\infty(0,T;L^\infty(\Omega))}^2 \\ \nonumber
    &+\tfrac{\varepsilon_2 + \varepsilon_3 + \varepsilon_4}{2}\|\Delta u_{n\,t} + \tfrac{\nu}{2}\Delta u_n\|_{L^2(0,T;L^2(\Omega))}^2  \\ \nonumber
& + \tfrac{\| \nu  - \tfrac{\mu}{\alpha}\|^2_{L^\infty(0,T;L^{2q/(q-1)}(\Omega))}}{2\varepsilon_2} \|\unt\|_{L^2(0,T;L^{2q}(\Omega))}^2 + \tfrac{1}{2\varepsilon_3} \|\tfrac{f}{\alpha}\|_{L^2(0,T;L^2(\Omega))}^2.
\end{align}
Exploiting the embedding $H^{3/2}(\Omega) \subseteq L^{2q}(\Omega)$ for $q \in [1,\infty)$ we get
\begin{align}
    |\text{rhs}| & \leq \|\tfrac{1}{\beta}\|_{L^\infty(\partial\Omega)} \left( (1+\tfrac{\nu^2}{4})  \varepsilon_1  \|\g\|_{H^1(0,T;L^2(\partial \Omega))}^2 +  \tfrac{1}{2 \varepsilon_1}  \|\nabla\unt \cdot \vecn \|_{L^2(0,T;L^2(\partial \Omega))}^2 \right)  \\ \nonumber
    &+(\| \left(\tfrac{1}{2\alpha} \right)_t\left(c^2 + \tfrac{\nu}{2} \mathfrak{b} \right)\|_{L^\infty(0,T;L^\infty(\Omega))}^2 + \tfrac{C_a}{2\varepsilon_4}\|\tfrac{\delta}{\alpha}\|_{L^2(0,T;L^2(\Omega))}^2) \|\un \|_{L^\infty(0,T;L^\infty(\Omega))}^2 \\ \nonumber
    &+(\varepsilon_2 + \varepsilon_3 + \varepsilon_4) (\|\Delta u_{n\,t}\|_{L^2(0,T;L^2(\Omega)}^2 + \tfrac{\nu}{2}\|\Delta u_n\|_{L^2(0,T;L^2(\Omega))}^2 ) \\ \nonumber
    &+ \tfrac{\| \nu - \tfrac{\mu}{\alpha}\|^2_{L^\infty(0,T;L^{2q/(q-1)}(\Omega))}}{2\varepsilon_2} C_{H^{3/2} \rightarrow L^{2q}} ^2\|\unt \|_{L^2(0,T;H^{3/2}(\Omega))}^2 \\ \nonumber
    & + \tfrac{\|\tfrac{1}{\alpha} \|_{L^\infty(\Omega)}^2}{2\varepsilon_3} \| f \|_{L^2(0,T;L^2(\Omega))}^2.
\end{align}
Due to the employed ansatz space and differentiating with respect to time, one finds that $\unt = -\frac{1}{\gamma} \nabla \unt \cdot \vecn$ on $\partial \Omega$. Together with the imposed bounds on the Robin parameters this admits the following estimate from below of the left hand side 
\begin{align}
    \text{lhs} \geq &\tfrac{\nu}{2} \|\nabla \unt \|_{L^2(0,T;L^2(\Omega))}^2 + \tfrac{\nu}{2} \| \sqrt{\tfrac{c^2}{\alpha}} \Delta \un \|_{L^2(0,T;L^2(\Omega))}^2 + \|\sqrt{\tfrac{\mathfrak{b}}{\alpha}} \Delta \unt \|_{L^2(0,T;L^2(\Omega))} \\ \nonumber
    & + \nu c_\gamma \| \unt \|_{L^2(0,T;L^2(\partial \Omega))}^2,
\end{align}
where we skipped the term in $h_{n_t}$ and $c_\gamma > 0$ denote the lower bound of $\gamma$. By elliptic regularity of the impedance Laplacian (exploiting the fact that $\partial \Omega \in C^{1,1}$) and due to our ansatz space setting we obtain the following estimate
\begin{align}
    \|\unt\|_{H^{3/2}(\Omega)} \leq \Comega (\|\Delta \unt\|_{L^2(\Omega)} + \| \unt \|_{L^2(\Omega)}),
\end{align}
where $\Comega > 0$ does only depend on $\Omega$. By the imposed bounds on the coefficients and by controlling them using $\varepsilon_i > 0$, $i \in \{1,2,3,4\}$ and $\tfrac{\gamma}{\beta} > \nu >0$ we obtain a $C>0$ such that the following estimate holds
\begin{align}
    &(1+\tfrac{\nu^2}{4})\|\tfrac{1}{\beta}\|_{L^\infty(\partial\Omega)} \varepsilon_1  \|\g\|_{H^1(0,T;L^2(\partial \Omega))}^2 + \tfrac{\|\tfrac{1}{\alpha} \|_{L^\infty(\Omega)}^2}{2\varepsilon_3} \| f \|_{L^2(0,T;L^2(\Omega))}^2  
    \\ \nonumber
    \geq 
    &\tfrac{\nu}{2}||\nabla \unt||_{L^2(0,T;L^2(\Omega))}^2 +  \left(\nu c_\gamma- \|\tfrac{1}{\beta} \|_{L^\infty(\partial \Omega)} \tfrac{1}{2\varepsilon_1}\right) || \unt ||_{L^2(0,T;L^2(\partial \Omega))}^2  \\ \nonumber
    & + \tfrac{\nu}{2} \left( \| \sqrt{\tfrac{c^2}{\alpha}} \Delta \un \|_{L^2(0,T;L^2(\Omega))}^2 - (\varepsilon_2 + \varepsilon_3 + \varepsilon_4)\| \Delta\un\|^2_{L^2(0,T;L^2(\Omega))} \right)\\ \nonumber
     & - (\| \left(\tfrac{1}{2\alpha} \right)_t\left(c^2 + \tfrac{\nu}{2} \mathfrak{b} \right)\|_{L^\infty(0,T;L^\infty(\Omega))}^2 + \tfrac{C_a}{2\varepsilon_4}\|\tfrac{\delta}{\alpha}\|_{L^2(0,T;L^2(\Omega))}^2) \|\un \|_{L^\infty(0,T;L^\infty(\Omega))}^2 \\ \nonumber
    & + \|\sqrt{\tfrac{\mathfrak{b}}{\alpha}}\Delta \unt \|_{L^2(0,T;L^2(\Omega))}^2 - (\varepsilon_2 + \varepsilon_3 + \varepsilon_4) \| \Delta \unt \|_{L^2(0,T;L^2(\Omega))}^2 \\ \nonumber
    & - \tfrac{\| \nu - \tfrac{\mu}{\alpha}\|^2_{L^\infty(0,T;L^{2q/(q-1)}(\Omega))}}{2\varepsilon_2} C_{H^{3/2} \rightarrow L^{2q}}^2 ||\unt||_{L^2(0,T;H^{3/2}(\Omega))}^2 \\ \nonumber
    \geq & \Bigl( \tfrac{ \min\{\tfrac{\nu}{2}, \nu c_\gamma - \|\tfrac{1}{\beta} \|_{L^\infty(\partial \Omega)} \tfrac{1}{2\varepsilon_1}, \Tilde{b}_\alpha - (\varepsilon_2 + \varepsilon_3 + \varepsilon_4)\}}{\CPF \Comega^2} \\ \nonumber
    & - \tfrac{\| \nu - \tfrac{\mu}{\alpha}\|^2_{L^\infty(0,T;L^{2q/(q-1)}(\Omega))}}{2\varepsilon_2} C_{H^{3/2} \rightarrow L^{2q}}^2 \Bigr) \| \unt\|_{L^2(0,T,H^{3/2}(\Omega))}^2 \\ \nonumber
    & + \Bigl( \tfrac{\nu\tilde{c}_\alpha}{2} - (\varepsilon_2 + \varepsilon_3 + \varepsilon_4) - \| \left(\tfrac{1}{2\alpha} \right)_t\left(c^2 + \tfrac{\nu}{2} b \right)\|_{L^\infty(0,T;L^\infty(\Omega))}^2 \\ \nonumber
    & - \tfrac{C_a}{2\varepsilon_4}\|\tfrac{\delta}{\alpha}\|_{L^2(0,T;L^2(\Omega))}^2 \Bigr) \|\un \|_{L^\infty(0,T,L^\infty(\Omega))}^2 \\ \nonumber
    \geq & \tfrac{1}{C} \left(\| \unt\|_{L^2(0,T,H^{3/2}(\Omega))}^2 + \| \Delta \un \|_{L^2(0,T,L^2(\Omega))}^2 \right),
\end{align}
where $\CPF > 0$ is the constant arising in the Poincaré-Friedrichs inequality, $\Tilde{c}_\alpha >0$ is determined by the fact that $\tfrac{c^2}{\alpha} > 0$ is bounded from below in an a.e. sense in $(0,T) \times \Omega$, and in the same way we determine $\tilde{b}_\alpha > 0$. Hence, we obtain 
\begin{align}
\label{eq:two:sources:proof:estimate:unt:laplaceun:space:dependent}
    \| \unt\|_{L^2(0,T,H^{3/2}(\Omega))}^2 + \| \Delta \un \|_{L^2(0,T,L^2(\Omega))}^2 \leq C (\|f\|_{L^2(0,T;L^2(\Omega))}^2 + \|\g \|_{H^1(0,T;L^2(\partial \Omega))}^2) .
\end{align}
Likewise we obtain an estimate on $h_{n\,t}$, since we have already established that
\begin{align}
    \| h_{n\,t} - (\beta \unt + \gamma \un + \nabla \un \cdot \vecn - \g) \|_{L^2(0,T;L^2(\partial \Omega))}^2 \leq C (\| & f\|_{L^2(0,T;L^2(\Omega))}^2 \\ \nonumber
    &+ \|\g \|_{H^1(0,T;L^2(\partial \Omega))}^2).
\end{align}
Due to our ansatz space $\gamma \un + \nabla \un \cdot \vecn = 0$ on $\partial \Omega$ and using~\eqref{eq:two:sources:proof:estimate:unt:laplaceun:space:dependent} yields an $C > 0$ such that
\begin{align}
    \| h_{n\,t}\|_{L^2(0,T;L^2(\partial \Omega))} \leq C (\|f\|_{L^2(0,T;L^2(\Omega))} + \|\g\|_{H^1(0,T;L^2(\partial \Omega))}).
\end{align}
In order to obtain an estimate on $\untt$ we test~\eqref{eq:two:sources:linear:westervelt:weak:form:boundary:source:space:dependent} with $v = (-\Delta_\gamma)^{-1} \untt$ and $\zeta = 0$. Reversing the integration by parts step in~\eqref{eq:two:sources:linear:westervelt:weak:form:boundary:source:space:dependent} yields
\begin{align}
 \int_{0}^T\int_\Omega \untt \untt + \tfrac{1}{\alpha} \left( - c^2 \Delta - \mathfrak{b} \Delta \unt + \mu \unt + \delta \un + f \right) \untt \,dx\\ + \int_{\partial \Omega} \tfrac{1}{\beta}(\beta \untt + \gamma \unt + \nabla \unt \cdot \vecn - \g_t) (\nabla v \cdot \vecn) \,dS(x)\,dt = 0.
\end{align}
By our ansatz space we have that $\nabla v \cdot \vecn = -\gamma v$, since $v = (-\Delta_\gamma )^{-1} \untt$ on $\partial \Omega$, hence we obtain
\begin{align}
\label{eq:two:sources:proof:untt:estimate:space:dependent}
     \| \untt \|_{L^2(0,T;L^2(\Omega))} ^2 = & \int_{0}^T \left(\tfrac{1}{\alpha}( c^2 \Delta \un + \mathfrak{b} \Delta \unt - \mu \unt - \delta \un - f ), \untt \right)_{L^2(\Omega)} \\ \nonumber
    & + \langle \tfrac{\gamma}{\beta} \left(\beta \untt + \gamma \unt + \nabla \unt \cdot \vecn - \g_t \right), v\rangle_{(H^{3/2}(\partial \Omega)^*, H^{3/2}(\partial\Omega))} \,dt
\end{align}
Now, by the trace theorem and elliptic regularity of the impedance Laplacian we have
\begin{align}
    \label{eq:two:sources:boundary:trace:ellipticity:estimate:space:dependent}
    \| v \|_{H^{3/2}(\partial \Omega)} \leq \Ctrace \| v \|_{H^2(\Omega)} \leq \Ctrace \Comega \| \Delta_\gamma v \|_{L^2(\Omega)} = \Ctrace \Comega \| \untt \|_{L^2(\Omega)}.
\end{align}
\eqref{eq:two:sources:proof:untt:estimate:space:dependent} together with~\eqref{eq:two:sources:boundary:trace:ellipticity:estimate:space:dependent},~\eqref{eq:two:sources:proof:estimate:unt:laplaceun:space:dependent} and the Cauchy-Schwarz inequality yields an $C > 0$ such that
\begin{align}
    \| \untt \|_{L^2(0,T;L^2(\Omega))} \leq C (\|f\|_{L^2(0,T;L^2(\Omega))} + \|\g \|_{H^1(0,T;L^2(\partial \Omega))}).
\end{align}
Again exploiting our ansatz space, the problem formulation for the impedance Laplacian and~\eqref{eq:two:sources:proof:estimate:unt:laplaceun:space:dependent} we obtain an $C>0$ such that
\begin{align}
    \| \un \|_{L^2(0,T;L^2(\Omega))} \leq  C (\|f\|_{L^2(0,T;L^2(\Omega))} + \|\g \|_{H^1(0,T;L^2(\partial \Omega))}),
\end{align}
%
which together with elliptic regularity for the impedance Laplacian yields a $\tilde{C} > 0$ such that
\begin{align}
    \| \un \|_{L^2(0,T;H^{3/2}(\Omega))} \leq \Tilde{C}(\|f\|_{L^2(0,T;L^2(\Omega))} + \|\g \|_{H^1(0,T;L^2(\partial \Omega))}).
\end{align}
Passing to the limit. The above estimates show that $(\un)_{n\in\mathbb{N}}$ is a bounded sequence in 
$\solspace$ as defined in \eqref{eq:defH}\footnote{where again we have used elliptic regularity for the impedance Laplacian}
which is a Hilbert space and, hence, reflexive by the Hahn-Banach theorem. The Eberlein-\v{S}mulian theorem now yields a weakly convergent subsequence $(u_{n_j})_{j \in \mathbb{N}}$, $u_{n_j} \rightharpoonup u, j \rightarrow \infty$, and by the weak lower semicontinuity of the norm on $\solspace$
we have
\begin{align}
    \nonumber
    \| u \|_{\solspace} \leq C (\|f\|_{L^2(0,T;L^2(\Omega))} + \|\g \|_{H^1(0,T;L^2(\partial \Omega))}),
\end{align}
and due to linearity of the problem, we obtain that $u$ solves the PDE in a weak sense. Note, up to now, we did not use the assumed higher spatial regularity of $\g$. 

Higher regularity of solutions. If $u$ is a (weak) solution to~\eqref{eq:two:sources:linear:westervelt:weak:form:boundary:source:space:dependent}, then $\gamma u + \nabla u \cdot \vecn = -\beta u_t + \g$ is the trace of a $H^1(\Omega)$ function, hence we have
\begin{align}
    \| \gamma u + \nabla u \cdot \vecn \|_{L^2(0,T; H^{1/2}(\partial \Omega))} = \| -\beta u_t + \g \|_{L^2(0,T; H^{1/2}(\partial \Omega))} \leq C ||\g||_{H^1(0,T;H^{1/2}(\partial \Omega))},
\end{align}
and therefore, elliptic regularity yields a  $C > 0$ such that the following estimates holds
\begin{align}
    \|u\|_{L^2(0,T,H^2(\Omega))} \leq C (\|f\|_{L^2(0,T;L^2(\Omega))} + \|\g \|_{H^1(0,T;H^{1/2}(\partial \Omega))}).
\end{align}
Together with the estimates above we conclude that there exists a $C>0$ independent of $u$ such that
\begin{align}
    \|u\|_{\solspace}
    \leq C (\|f\|_{L^2(0,T;L^2(\Omega))} + \|\g \|_{H^1(0,T;H^{1/2}(\partial \Omega))}).
\end{align}
\qed
\section{Proof of Theorem~\ref{thm:westervelt:nonlinear:solution:existence:uniqueness:source:boundary}}
\label{appendix:proof:theorem:nonlinear:well-posedness}
We define $\mathcal{F}: \solspace \rightarrow \solspace$ by the solution $u=\mathcal{F}(v)$ of
\begin{equation}
\label{rmk:eq:westervelt:nonlinear:periodic:fixed:point}
\begin{cases}
bu_{tt} - s \Delta u -  \Delta u_t = \eta (v^2)_{tt} + \hh & \text{in } (0,T) \times \Omega,\\
\beta u_t + \gamma u + \nabla u \cdot \vecn = \g & \text{on } (0,T) \times \partial\Omega, \\
u(0) = u(T), \, u_t(0) = u_t(T) & \text{in } \Omega, \\
\end{cases}   
\end{equation}
$v \in X$. 
We will show that $\mathcal{F}$, restricted to a suitable ball $B_r(0) \subseteq \solspace$, $r > 0$, is a contraction and since $\solspace$ is a Banach space we obtain uniqueness and existence of a solution by the Banach fixed point Theorem.

Now, for $r>0$ fixed (with its size yet to be determined), let $v \in B_r(0)$ be arbitrary
In order to apply~\cref{thm:two:sources:westervelt:linear:solution:existence:uniqueness:mixed:boundary:space:dependent} we have to check whether $f:= \eta (v^2)_{tt} + \hh$ (together with $\alpha = b$, $\mu = 0$ and $\delta = 0$) fulfills the assumptions of~\cref{thm:two:sources:westervelt:linear:solution:existence:uniqueness:mixed:boundary:space:dependent}. For the right hand side we obtain 
\begin{align}
    \nonumber
    & || \eta  (v^2)_{tt} + h||_{L^2(0,T;L^2(\Omega))}  \leq 2 || \eta ||_{L^\infty(\Omega)} || v_t^2 +  v_{tt} v||_{L^2(0,T;L^2(\Omega))} + \|h\|_{L^2(0,T;L^2(\Omega))}.
\end{align}
Exploiting the Sobolev embedding Theorem (see,~\cite{evans2010} and~\cite{FAN2001749}) and by (real) interpolation~\cite{Amann95} we have 
\begin{align}
\nonumber
    &H^{3/4}(0,T;H^{13/8}(\Omega)) =  \\ \nonumber
    &H^{3/4}(0,T;[H^2(\Omega), H^{3/2}(\Omega)]_{3/4}) = [L^2(0,T;H^2(\Omega)), H^1(0,T;H^{3/2}(\Omega))]_{3/4},
\end{align} 
and 
\begin{align}
    \nonumber
    &H^{1/4}(0,T;H^{9/8}(\Omega)) =  \\ \nonumber
    &H^{1/4}(0,T;[H^{3/2}(\Omega), L^{2}(\Omega)]_{1/4}) = [L^2(0,T;H^{3/2}(\Omega)), H^1(0,T;L^{2}(\Omega))]_{1/4}.
\end{align}
With this we obtain 
\begin{align}
\label{eq:bochner:L:infinity:estimate:H}
||v||_{L^\infty(0,T;L^\infty(\Omega))} &\leq C_{H^{13/8}(\Omega) \rightarrow L^\infty(\Omega)} ||v||_{L^\infty(0,T;H^{13/8}(\Omega))} \\ \nonumber
& \leq  C_{H^{13/8}(\Omega) \rightarrow L^\infty(\Omega)}  C_{H^{3/4}(0,T) \rightarrow L^\infty(0,T)}||v||_{H^{3/4}(0,T;H^{13/8}(\Omega))} \\ \nonumber
& \leq  C_{H^{13/8}(\Omega) \rightarrow L^\infty(\Omega)}  C_{H^{3/4}(0,T) \rightarrow L^\infty(0,T)} 
 \\ \nonumber
&\qquad \qquad \qquad \qquad \qquad \qquad  ||v||_{H^{1}(0,T;H^{3/2}(\Omega))}^{3/4} ||v||_{L^{2}(0,T;H^{2}(\Omega))}^{1/4} \\ \nonumber
& \leq  C_{H^{13/8}(\Omega) \rightarrow L^\infty(\Omega)}  C_{H^{3/4}(0,T) \rightarrow L^\infty(0,T)} ||v||_{\solspace} \\ \nonumber
&\leq C_{H^{13/8}(\Omega) \rightarrow L^\infty(\Omega)}  C_{H^{3/4}(0,T) \rightarrow L^\infty(0,T)} r,
\end{align}
which yields that $v \in L^\infty(0,T;L^\infty(\Omega))$. We further show that $v_t^2$ is in  $L^2(0,T;L^2(\Omega))$:
\begin{align}
    \label{eq:bochner:squared:estiamte}
    || v_t^2 ||_{L^2(0,T;L^2(\Omega))} & \leq || v_t||_{L^4(0,T;L^4(\Omega))}^2 \\ \nonumber
    & \leq  C_{H^{1/4}(0,T) \rightarrow L^4(0,T)}^2 C_{H^{9/8}(\Omega) \rightarrow L^4(\Omega)}^2 ||v_t||_{H^{1/4}(0,T;H^{9/8}(\Omega))}^2 \\ \nonumber
    & \leq C_{H^{1/4}(0,T) \rightarrow L^4(0,T)}^2 C_{H^{9/8}(\Omega) \rightarrow L^4(\Omega)}^2 \\ \nonumber
    & \qquad \qquad \qquad \qquad \qquad \qquad  \left(||v_t||_{H^1(0,T;L^2(\Omega))}^{1/4} ||v_t||_{L^2(0,T;H^{3/2}(\Omega))}^{3/4} \right)^2 \\ \nonumber
    & \leq C_{H^{1/4}(0,T) \rightarrow L^4(0,T)}^2 C_{H^{9/8}(\Omega) \rightarrow L^4(\Omega)}^2 ||v||_{\solspace}^2 \\ \nonumber
    & \leq  C_{H^{1/4}(0,T) \rightarrow L^4(0,T)}^2 C_{H^{9/8}(\Omega) \rightarrow L^4(\Omega)}^2 r^2.
\end{align}
Now we obtain that $f = \eta(v^2)_{tt} + \hh \in L^2(0,T;L^2(\Omega))$
\begin{align}
    \nonumber
    & 2 || \eta ||_{L^\infty(\Omega)} || v_t^2 +  v_{tt} v||_{L^2(0,T;L^2(\Omega))} + \|\hh\|_{L^2(0,T;L^2(\Omega))}\\ \nonumber
    & \leq \|\hh\|_{L^2(0,T;L^2(\Omega))} + 4 r^2 ||\eta||_{L^\infty(\Omega)} \Bigl(C_{H^{1/4}(0,T) \rightarrow L^4(0,T)}^2 C_{H^{9/8}(\Omega) \rightarrow L^4(\Omega)}^2 \\ \nonumber
    & \qquad \qquad \qquad \qquad\qquad\qquad\qquad \qquad + C_{H^{13/8}(\Omega) \rightarrow L^\infty(\Omega)}  C_{H^{3/4}(0,T) \rightarrow L^\infty(0,T)} \Bigr).
\end{align}
We ensure that we do not encounter degeneracy by imposing the following smallness condition on $r>0$ 
\begin{align}
        \nonumber
        ||\eta||_{L^\infty(\Omega)} C_{H^{13/8}(\Omega) \rightarrow L^\infty(\Omega)}  C_{H^{3/4}(0,T) \rightarrow L^\infty(0,T)} \, r < \frac{1}{2}.
\end{align}
We note that this condition is sufficient to prevent the case of degeneracy, since 
\begin{align}
    \nonumber
  &\inf_{{(t,x) \in (0,T) \times \Omega}} (1 - 2 \eta v) \geq 1 - \sup_{{(t,x) \in (0,T) \times \Omega}} 2 \eta v \\ \nonumber
  &\geq 1 - 2 ||\eta||_{L^\infty(\Omega)}  C_{H^{13/8}(\Omega) \rightarrow L^\infty(\Omega)}  C_{H^{3/4}(0,T) \rightarrow L^\infty(0,T)} r > 0.
\end{align}
We further readily check that our setting fulfills the conditions on the parameters, since $\tfrac{\mu}{\alpha} = 0$ a.e. in $(0,T) \times \Omega$ and since $\alpha$ is constant in time we also have $\alpha_t = 0$ a.e. in $(0,T) \times \Omega$. Hence, we choose $\nu > 0$ small enough such that the conditions of ~\cref{thm:two:sources:westervelt:linear:solution:existence:uniqueness:mixed:boundary:space:dependent} are fulfilled. Now, the application of~\cref{thm:two:sources:westervelt:linear:solution:existence:uniqueness:mixed:boundary:space:dependent} is justified, and we obtain a $C>0$ such that
\begin{equation}
    \nonumber
    ||\mathcal{F}(v)||_{\solspace} \leq C (||\eta (v^2)_{tt}||_{L^2(0,T;L^2(\Omega))} + \|\hh\|_{L^2(0,T;L^2(\Omega))} + \|\g\|_{H^1(0,T;H^{1/2}(\Omega))}).
\end{equation}
In order to obtain a self-mapping we further impose the following conditions on $r$ and $\delta$
\begin{align}
        \label{eq:westervelt:nonlinear:condition:3:modified}
         C(4||\eta||_{L^\infty(\Omega)}( &C_{H^{1/4}(0,T) \rightarrow L^4(0,T)}^2 C_{H^{9/8}(\Omega) \rightarrow L^4(\Omega)}^2 \\ \nonumber
         &+  C_{H^{13/8}(\Omega) \rightarrow L^\infty(\Omega)}  C_{H^{3/4}(0,T) \rightarrow L^\infty(0,T)}) r^2 + \Lambda ) < r,
\end{align}
which can be achieved by choosing $\Lambda>0$ and $r>0$ sufficiently small. 

To show contractivity of $\mathcal{F}$, for any two $v_1, v_2 \in B_r(0)$ set $w = u_1 - u_2 = \mathcal{F}(v_1) - \mathcal{F}(v_2)$, and both solutions exist according to~\cref{thm:two:sources:westervelt:linear:solution:existence:uniqueness:mixed:boundary:space:dependent}. Then $w$ solves
\begin{equation}
\label{eq:westervelt:nonlinear:periodic:fixed:point:quadratic}
\begin{cases}
bw_{tt} - s \Delta w - \Delta w_t = \eta \left((v_1^2)_{tt} - (v_2^2)_{tt} \right)  & \text{in } (0,T) \times \Omega,\\
\beta w_t + \gamma w + \nabla w \cdot \vecn = 0 & \text{on } (0,T) \times \partial\Omega, \\
w(0) = w(T), \, w_t(0) = w_t(T) & \text{in } \Omega. \\
\end{cases}   
\end{equation}
Estimating $\eta \left((v_1^2)_{tt} - (v_2^2)_{tt} \right)$ like before, we apply~\cref{thm:two:sources:westervelt:linear:solution:existence:uniqueness:mixed:boundary:space:dependent} which yields 
\begin{equation}
    \nonumber
    ||w||_{\solspace} \leq C ||\eta \left((v_1^2)_{tt} - (v_2^2)_{tt} \right)||_{L^2(0,T;L^2(\Omega))}.
\end{equation}
We further have
\begin{align}
    (v_1^2)_{tt} - (v_2^2)_{tt} &= 2 (v_{1_t}^2 + v_1 v_{1_{tt}} - v_{2_t}^2 - v_2 v_{2_{tt}}) \\ \nonumber
    & = 2 \left( (v_{1_t} + v_{2_t})(v_{1_t} - v_{2_t}) +  v_1 v_{1_{tt}} -  v_2 v_{2_{tt}} \right) \\ \nonumber
    & = 2 \left( (v_{1_t} + v_{2_t})(v_{1_t} - v_{2_t}) +  v_1 v_{1_{tt}} -  v_2 v_{2_{tt}} - v_1 v_{2_{tt}} + v_1 v_{2_{tt}} \right) \\ \nonumber
    & = 2 \left( (v_{1_t} + v_{2_t})(v_{1_t} - v_{2_t}) + v_1(v_{1_{tt}} -  v_{2_{tt}}) + v_{2_{tt}} (v_1 - v_2) \right).
\end{align}
This yields the estimate
\begin{align}
    \nonumber
    & || (v_{1_t} + v_{2_t})(v_{1_t} - v_{2_t}) + v_1(v_{1_{tt}} -  v_{2_{tt}}) + v_{2_{tt}} (v_1 - v_2) ||_{L^2(0,T;L^2(\Omega))} \\ \nonumber
    & \leq  ||v_{1_t} + v_{2_t}||_{L^4(0,T;L^4(\Omega))}  ||v_{1_t} - v_{2_t}||_{L^4(0,T;L^4(\Omega))}  \\ \nonumber
    & \quad + ||v_1||_{L^\infty(0,T;L^\infty(\Omega))} ||v_{1_{tt}} -  v_{2_{tt}}||_{L^2(0,T;L^2(\Omega))} \\ \nonumber
    & \quad + ||v_1 - v_2||_{L^\infty(0,T;L^\infty(\Omega))} ||v_{2_{tt}}||_{L^2(0,T;L^2(\Omega))} \\ \nonumber
    &= ||v_{1_t} + v_{2_t}||_{L^4(0,T;L^4(\Omega))}  ||(v_{1} - v_{2})_t||_{L^4(0,T;L^4(\Omega))}  \\ \nonumber
    & \quad + ||v_1||_{L^\infty(0,T;L^\infty(\Omega))} ||(v_{1} -  v_{2})_{tt}||_{L^2(0,T;L^2(\Omega))} \\ \nonumber 
    & \quad + ||v_1 - v_2||_{L^\infty(0,T;L^\infty(\Omega))} ||v_{2_{tt}}||_{L^2(0,T;L^2(\Omega))} .
\end{align}
Applying the estimates from above we obtain 
\begin{align}
    ||\mathcal{F}(v_1) - \mathcal{F}(v_2)||_{\solspace} \leq rC_0 ||v_1 - v_2||_{\solspace},
\end{align}
where $C_0$ reads
\begin{align}
    4 C ||\eta||_{L^\infty(\Omega)} (C_{H^{13/8}(\Omega) \rightarrow L^\infty(\Omega)} & C_{H^{3/4}(0,T) \rightarrow L^\infty(0,T)} \\ \nonumber
    &+ C_{H^{1/4}(0,T) \rightarrow L^4(0,T)}^2 C_{H^{9/8}(\Omega) \rightarrow L^4(\Omega)}^2).
\end{align}
We further choose $r > 0$ small enough such that $rC_0 < 1$ yielding that $\mathcal{F}$ is a contraction on $B_r(0)$ and, thus, has a unique fixed point $u \in B_r(0)$ with $\mathcal{F}(u) = u$ solving~\eqref{eq:westervelt:nonlinear:periodic}. This concludes the proof.

\qed
\section{Multiharmonic ansatz for numerical computation of forward solutions}
\label{appendix:multiharmonic:ansatz}
In order to formulate a multiharmonic ansatz, we consider a boundary source of the form
\begin{align}
    g(t,x):= \real{\sum_{m=0}^N\widehat{g}_m(x)e^{\imath m \omega t}}.
\end{align}
This is exactly the form of the boundary sources obtained by our reference states $u^0$ (see Remark ~\ref{rem:example}). The periodicity imposed on the solution to~\eqref{eq:westervelt} and the source motivates the use of a multi-harmonic ansatz. By the given periodicity of the source $g(t,x)$, the space where we will project~\eqref{eq:westervelt} on has to take the form
\begin{align}
    X_{N} := \left\{ \real{\sum_{m=0}^N \alpha_{m}(x) e^{\imath m \omega  t}} : \alpha_{m} \in H^2(\Omega;\mathbb{C}) \right\}.
\end{align}
For $u,v \in X_{N}$, by the finite Cauchy product we have the following identity:
\def\hw{\widehat{w}}
\begin{align}
\nonumber
    &\text{Proj}_{X_N}(u^N v^N) = \frac{1}{2}\real{ \sum_{m=0}^{N} e^{\iota m \omega t} \sum_{j = 0}^{m}  \hat{u}_j(x) \hat{v}_{m-j}(x)} \\ \nonumber
    &+  \frac{1}{2} \real{ \sum_{j=0}^N \overline{\hat{u}_j(x)} \hat{v}_{j}(x) + \sum_{m=1}^N e^{\iota m \omega t} \sum_{k=m:2}^{2N-m}\left[ \overline{\hat{u}_{\frac{k-m}{2}}(x)} \hat{v}_{\frac{k+m}{2}}(x) + \hat{u}_{\frac{k+m}{2}}(x) \overline{\hat{v}_{\frac{k-m}{2}}(x)}\right] }.
\end{align}
We further have 
\begin{align}
    \nonumber
    \text{Proj}_{X_N}\left(b u^N_{tt} - s \Delta u^N - \Delta u^N_t \right) &= b\omega^2 \frac{1}{2} \sum_{m=0}^N - \hu_m(x) m^2 e^{\iota m \omega t} - \overline{\hu_m(x)} m^2 e^{-\iota m \omega t} \\ \nonumber
    & - s \frac{1}{2} \sum_{m=0}^N \Delta \hu_m(x) e^{\iota m \omega t} + \Delta \overline{\hu_m(x)} e^{-\iota m \omega t} \\ \nonumber 
    & - \iota \omega \frac{1}{2} \sum_{m=0}^N  \Delta\hu_m(x) m e^{\iota m \omega t} - \Delta\overline{\hu_m(x)} m e^{-\iota m \omega t},
\end{align}
hence projecting~\eqref{eq:westervelt} onto $X_{N}$ and by linear independence of $(e^{\imath m \omega t})_{m \in \mathbb{N}_0}$ we obtain that for $x \in \Omega$, $\hu_{m}(x)$ has to solve for
\begin{align}
   \label{eq:westervelt:helmholtz:system:nn}
    -m^2 \omega^2 b(x)\, \hu_{m}(x) - (s(x) + \imath m \omega) \Delta \hu_{m}(x) = - \tfrac{\eta(x) \omega ^2m^2}{2}&\Bigl(\sum_{k=0}^m \hu_{k}(x) \hu_{m-k}(x) \\ \nonumber 
    &+ 2\sum_{k=m:2}^{2N-m} \overline{\hu_{\frac{k-m}{2}}(x)} \hu_{\frac{k+m}{2}}(x)\Bigr),
\end{align}
with boundary condition (note that we set $\beta = 0$)
\begin{align}
    \gamma  \hat{u}_{m}(x) + \nabla \hat{u}_{m}(x) \cdot \vecn = \hat{g}_{m}(x),
\end{align}
in the case of $m > 0$ and if $m = 0$ we have
\begin{align}
    -s(x)  \Delta \hu_{0}(x) = 0,
\end{align}
in $\Omega$ and on the boundary $\partial \Omega$, $\hu_0$ has to fulfill
\begin{align}
    \gamma \hat{u}_{0}(x) + \nabla \hat{u}_{0}(x) \cdot \vecn = \hat{g}_{0}(x).
\end{align}
Dividing~\eqref{eq:westervelt:helmholtz:system:nn} by $(s + \imath m \omega)$ we obtain
\begin{align}
   \label{eq:westervelt:helmholtz:system:normalised}
    - \kappa_m(x)^2 m^2\, \hu_{m}(x) - \Delta \hu_{m}(x) = - \tfrac{\eta(x) \kappa_m(x)^2 m^2}{2b(x)} \Bigl(&\sum_{k=0}^m \hu_{k}(x) \hu_{m-k}(x)  \\ \nonumber 
    &+ 2\sum_{k=m:2}^{2N-m} \overline{\hu_{\frac{k-m}{2}}(x)} \hu_{\frac{k+m}{2}}(x)\Bigr),
\end{align}
where we set $\kappa_m(x)^2:= \tfrac{\omega^2 b^0}{s^0 + \imath m \omega}$. Well-posedness and convergence of an iterative scheme is discussed in~\cite{RK2025}.

Analogously, we project~\eqref{eq:westervelt:frechet:derivative} onto $X_{N}$ and find that for $m \neq 0$, $\hdu_{m}$ has to solve
\begin{align}
    \label{eq:westervelt:frechet:derivative:helmholtz:system}
    & - \kappa_{m}(x)^2 m^2 \hdu_{m}(x) - \Delta \hdu_{m}(x) = \\ \nonumber
    & -\eta^0 \tfrac{m^2 \omega^2}{s^0 + \imath m \omega} \Bigl( \sum_{i=0}^m \hu^0_{i}(x)\hdu_{m-i}(x) + \sum_{l=m:2}^{2N-m} \Bigl[ \overline{\hu^0_{\tfrac{l-m}{2}}(x)}\hdu_{\tfrac{l+m}{2}}(x)  \\ \nonumber 
    & \qquad\qquad\qquad \qquad \qquad \qquad\qquad\qquad\qquad\qquad+ \hu^0_{\tfrac{l+m}{2}}(x) \overline{\hdu_{\tfrac{l-m}{2}}(x)} \Bigr]\Bigr) \\ \nonumber
    & - d\eta(x) \tfrac{m^2 \omega^2}{2(s^0 + \imath m \omega)} \left( \sum_{i=0}^m \hu^0_{i}(x) \hu^0_{m-i}(x) + 2\sum_{l=m:2}^{2N-m} \overline{\hu^0_{\tfrac{l-m}{2}}(x)} \hu^0_{\tfrac{l+m}{2}}(x) \right)\\ \nonumber
    & +db(x) \tfrac{m^2 \omega^2}{s^0 + \imath m \omega} \hu^0_m(x) + ds(x) \tfrac{1}{s^0 + \imath m \omega} \Delta \hu^0_m(x),
\end{align}
with the homogeneous boundary conditions 
\begin{align}
    \label{eq:westervelt:frechet:derivative:helmholtz:boundary}
    \gamma(x) \hdu_{m}(x) + \nabla \hdu_{m} (x) \cdot \vecn = 0.
\end{align}
For $m=0$ we obtain that $\hdu_0$ has to solve
\begin{align}
    -s^0\Delta \hdu_0(x) = ds(x) \Delta \hu^0_0,
\end{align}
in $\Omega$ and is also equipped with homogeneous Robin boundary conditions. 

The adjoint state $p=K^*y$ for some $y\in L^2(0,T;L^2(\Sigma))$ has to fulfill 
\begin{align}
    &\int_0^T \int_\Omega \du [(b^0 - 2\eta^0u^0)p_{tt} - s^0 \Delta p + \Delta p_t]\,dx\,dt  \\ \nonumber 
    &+ \int_0^T \int_{\partial \Omega} \du [s^0(\gamma p + \nabla p \cdot \vecn) - (\gamma p_t + \nabla p_t \cdot \vecn)]\, dS(x)dt \\ \nonumber 
   &= \int_0^T\int_\Sigma \du y\, dS(x)\,dt, \text{ for all  }\du\in\solspace,
\end{align}
hence $p$ has to solve the following PDE in a weak sense
\begin{equation}
\label{eq:westervelt:adjoint:state}
\begin{cases}
(b^0 - 2 \eta^0 u^0) p_{tt} - s^0 \Delta p + \Delta p_t = 0 & \text{in}\, (0,T) \times \Omega,\\
\gamma ( s^0p - p_t) + \nabla (s^0p - p_t) \cdot \vecn = y & \text{on}\,  (0,T) \times \Sigma , \\
\gamma ( s^0p - p_t) + \nabla (s^0p - p_t) \cdot \vecn= 0 & \text{on}\,  (0,T) \times \partial\Omega \setminus \Sigma , \\
p(0,x) = p(T,x), \, p_t(0,x) = p_t(T,x) & x \in \Omega. \\
\end{cases}   
\end{equation}
The space dependent~\textit{gradient} with respect to $db$, $ds$, $d\eta$ reads
\begin{align}
\left(-\int_0^T u^0_{tt} p \,dt, \int_0^T \Delta u^0 p \,dt, \int_0^T ((u^0)^2)_{tt} p\,dt\right).
\end{align}
In view of ~\eqref{eq:westervelt:adjoint:state} the adjoint state $p$ projected to $X_N$ has to solve 
\newcommand{\hp}{\hat{p}}
\begin{align}
    & - \overline{\kappa_{m}(x)^2} m^2\hp_{m}(x) - \Delta \hp_{m}(x) = \\ \nonumber
    & -\eta^0 \tfrac{1}{s^0 - \imath m \omega} \Bigl( \sum_{j=0}^m j^2\omega^2 \overline{\hu^0_{j}(x)} \hp_{m-j}(x) \\ \nonumber
    & \qquad \qquad  \qquad + \sum_{l=m:2}^{2N-m} l^2 \left[\hu^0_{\tfrac{l-m}{2}}(x) (\tfrac{l+m}{2})^2\overline{\hp_{\tfrac{l+m}{2}}(x)} + \overline{\hu^0_{\tfrac{l+m}{2}}(x)} (\tfrac{l-m}{2})^2 \hp_{\tfrac{l-m}{2}}(x) \right]\Bigr) \\ \nonumber
\end{align}
equipped with the Robin boundary condition 
\begin{align}
\nonumber
\gamma(x) \hp_{m}(x) + \nabla \hp_{m} (x) \cdot \vecn = \mathds{1}_\Sigma(x) \tfrac{1}{s^0 - \imath m \omega} \hat{y}_{m} (x).
\end{align}